\documentclass[a4paper,11pt]{amsart}

\usepackage{amsmath, amsthm, amsfonts, amssymb, amscd}
\usepackage{color}
\usepackage[ansinew]{inputenc}
\usepackage[english]{babel}
\usepackage{graphicx}
\usepackage{tikz}
\usetikzlibrary{decorations.markings}
\usepackage[colorlinks]{hyperref}

\newtheorem{theorem}{Theorem}[section] 
\newtheorem{theoreme}{Theorem}

\newtheorem{corollary}[theorem]{Corollary}
\newtheorem{lemma}[theorem]{Lemma}

\newtheorem{sublemma}[theorem]{Sub-lemma}
\newtheorem{proposition}[theorem]{Proposition}
\newtheorem*{proposition*}{Proposition}

\newtheorem{claim}[theorem]{Claim}
\newtheorem*{question*}{Question}
\newtheorem*{theorem*}{Theorem}
\newtheorem*{claim*}{Claim}

\newtheorem*{corollary*}{Corollary}

\theoremstyle{definition}

\theoremstyle{remark}

\newtheorem*{remark*}{Remark}

\newcommand{\R}{\mathbb{R}}\newcommand{\N}{\mathbb{N}}
\newcommand{\Z}{\mathbb{Z}}\newcommand{\Q}{\mathbb{Q}}
\newcommand{\T}{\mathbb{T}}\newcommand{\C}{\mathbb{C}}
\newcommand{\D}{\mathbb{D}}\newcommand{\A}{\mathbb{A}}
\renewcommand{\S}{\mathbb{S}}
\newcommand{\G}{\mathcal{G}}
\newcommand{\F}{\mathcal{F}}
\newcommand{\dom}{\mathrm{dom}}

\def\FHL{\mathcal{FHL}}

\DeclareMathOperator{\Diff}{Diff}
\DeclareMathOperator{\Homeo}{Homeo}

\begin{document}

\author[P.-A. Guih\'eneuf]{Pierre-Antoine Guih\'eneuf}
\address{Institut de Math\'ematiques de Jussieu-Paris Rive Gauche, IMJ-PRG, Sorbonne Universit\'e, Universit\'e Paris-Diderot, CNRS, F-75005, Paris, France} 
\curraddr{}
\email{pierre-antoine.guiheneuf@imj-prg.fr}

\author[P. Le Calvez]{Patrice Le Calvez}
\address{Institut de Math\'ematiques de Jussieu-Paris Rive Gauche, IMJ-PRG, Sorbonne Universit\'e, Universit\'e Paris-Diderot, CNRS, F-75005, Paris, France \enskip \& \enskip Institut Universitaire de France}
\curraddr{}
\email{patrice.le-calvez@imj-prg.fr}

\author[A. Passeggi]{Alejandro Passeggi}
\address{CMAT, Facultad de Ciencias, Universidad de la Rep\'ublica, Uruguay.}
\curraddr{Igua 4225 esq. Mataojo. Montevideo, Uruguay.}
\email{alepasseggi@cmat.edu.uy}

\title[Conservative surface homeomorphisms with rational rotation]{Area preserving homeomorphisms of surfaces with rational rotational direction}

\begin{abstract}
Let $S$ be a closed surface of genus $g\geq 2$, furnished with a Borel probability measure $\lambda$ with total support. We show that if $f$ is a $\lambda$-preserving homeomorphism isotopic to the identity such that the rotation vector $\mathrm{rot}_f(\lambda)\in H_1(S,\R)$ is a multiple of an element of $H_1(S,\Z)$, then $f$ has infinitely many periodic orbits.

Moreover, these periodic orbits can be supposed to have their rotation vectors arbitrarily close to the rotation vector of any fixed ergodic Borel probability measure.
\end{abstract}

\maketitle

\bigskip
\noindent {\bf Keywords:} Rotation vector, maximal isotopy, transverse foliation

\bigskip
\noindent {\bf MSC 2020:}  37C25 37E30, 37E45

\tableofcontents

\section{Introduction}

\subsection{Rotation vector}\label{ss.rotationvector}

If $ S$ is a smooth compact boundaryless oriented surface of genus $g$, we denote $\mathrm{Homeo}(S)$ the space of homeomorphisms of $S$ furnished with the $C^0$-topology. This topology coincides with the uniform topology because $S$ is compact. The path-connected component of the identity map $\mathrm {Id}$, usually called the space of homeomorphisms \emph{isotopic to the identity}, will be denoted $\mathrm{Homeo}_*(S)$. A continuous path $I=(f_t)_{t\in[0,1]}$ joining the identity to a map $f\in \mathrm{Homeo}_*(S)$ is called an {\it identity isotopy} of $f$. We call  {\it trajectory} of a point $z\in S$ defined by $I$ the path $I(z): t\mapsto f_t(z)$ joining $z$ to $f(z)$.  

By compactness of $S$, one knows by Krylov-Bogolioubov's theorem that the set ${\mathcal M}(f)$ of $f$-invariant Borel probability measures is not empty. More precisely it is a non empty compact convex subset of the space $\mathcal M$ of Borel probability  measures furnished with the weak$^{\ast}$ topology. Remind that the {\it support} of $\mu$, denoted $\mathrm{supp}(\mu)$, is the smallest closed set of $\mu$-measure $1$. 
\medskip

Let us recall the definition of the {\it rotation vector} of a measure $\mu\in {\mathcal M}(f)$ (see \cite{Mats1}, \cite{Pol} or \cite{Sc}). Let $I=(f_t)_{t\in[0,1]}$ be an identity isotopy of $f$. Fix $z\in S$. The homotopy class of $I(z)$, relative to the endpoints,  contains a smooth path $\gamma$ joining  $z$ to $f(z)$. If $\alpha$ is a closed $1$-form,  the quantity $\int_{\gamma} \alpha$ does not depend on the choice of $\gamma$ and we denote it $\int_{I(z)} \alpha$. It is equal to $h(f(z))-h(z)$ if $\alpha$ is exact and $h$ is a primitive of $\alpha$.
One gets a real valued morphism $\alpha\mapsto \int_S \left(\int_{I(z)} \alpha\right) \, d\mu(z)$ defined on the space of closed $1$-forms, that vanishes on the space of exact $1$-forms because $\mu$ is invariant by $f$. So, it  induces a natural linear form on the first cohomology group $H^1(S, \R)$. Hence, there exists a homology class $\mathrm{rot}_{I}(\mu)\in H_1(S,\R)$, uniquely defined by the equation
$$\langle [\alpha], \mathrm{rot}_{I}(\mu)\rangle= \int_S \left(\int_{I(z)} \alpha\right) \, d\mu(z),$$
where $\alpha$ is any closed $1$-form, $[\alpha]\in H^1(S,\R)$ its cohomology class and 
$$\langle \enskip,\enskip \rangle: H^1(S,\R)\times H_1(S,\R)\to\R$$
the natural bilinear form. By definition $\mathrm{rot}_{I}(\mu)\in H_1(S,\R)$ is the rotation vector of $\mu$ (for the isotopy $I$). It is well known that two identity isotopies of $f$ are homotopic relative to the ends if the genus of $S$ is larger than $1$ (see \cite{H}). In that case, $\int_{I(z)} \alpha$ does not depend on $I$ and one can write 
\[\mathrm{rot}_{f}(\mu) = \mathrm{rot}_{I}(\mu).\]
If $O$ is a periodic orbit of $f$, one can define the rotation vector $\mathrm{rot}_I(O)$ of $O$ (or $\mathrm{rot}_f(O)$ if the genus of $S$ is larger than $1$) as being equal to the rotation vector of $\mu_O$, where $\mu_O$ is the probability measure equidistributed on $O$.
In particular we have $\mathrm{rot}_I(O)=0$ if $O$ is a contractible periodic orbit, which means that the loop $I^q(z)$ is homotopic to zero, if $z\in O$.

 Let us give an equivalent definition. Furnish $S$ with a Riemannian metric and for every points $z$, $z'$ in $S$, choose a path $\gamma_{z,z'}$ joining $z$ to $z'$ in such a way that the lengths of the paths $\gamma_{z,z'}$ are uniformly bounded. For every $z\in S$, and every $n\geq 1$, consider the path
$$I^n(z)= I(z) I(f(z))\cdots I(f^{n-1}(z)) $$
defined by concatenation, 
and the loop 
$$\Gamma_n(z)= I^n(z)\gamma_{f^n(z), z}.$$
One can prove that there exists a $\mu$-integrable function $\mathrm{rot}_f: S\to H_1(S,\R)$ such that for $\mu$-almost every point $z\in S$, the sequence $[\Gamma_n(z)]/n$ converges to $\mathrm{rot}_f(z)$. This allows to define
$$\mathrm{rot}_{I}(\mu)=\int \mathrm{rot}_f(z)\, d\mu(z).$$

Let us give a last definition that will be used in this article. In the whole text we will write $[\Gamma]\in H_1(S,\Z)$ for the homology class of an oriented loop $\Gamma\subset S$. Let $U\subset S$ be a topological open disk (meaning a simply connected domain) such that $\mu(U)\not=0$.
Write $\varphi_U:U\to U$ for the first return map of $f$ and $\tau_U :U\to \N\setminus\{0\}$ for the time of first return map. These maps are defined $\mu$-almost everywhere on $U$. Kac's Lemma \cite{K} tells us that $\varphi_U$ preserves the measure $\mu|_{U}$ and that $\tau_U$ is $\mu|_{U}$-integrable, and that moreover
$$\int_U \tau_U \, d\mu=\mu\left(\bigcup_{k\geq 0} f^k(U)\right) = \mu\left(\bigcup_{k\in\Z} f^k(U)\right).$$
We also denote by $\mu_U$ the normalized probability measure $\mu|_{U}/\mu(U)$. One can construct a map $\rho_U : U\to H_1(S,\Z)$ defined $\mu_U$-almost everywhere as follows: if $\varphi_U(z)$ is well defined, one closes the trajectory $I^{\tau_U(z)-1}(z)$ with a path $\gamma$ contained in $U$ that joins $\varphi_U(z)$ to $z$, and set $\rho_U(z)=[I^{\tau_U(z)-1}(z)\gamma]$, noting that $[I^{\tau_U(z)-1}(z)\gamma]$ is independent of the choice of $\gamma$. If the genus of of $S$ is bigger than 1 (what we suppose from now), then this map does not depend on the choice of $I$. It is easy to prove that the map  $\rho_U/\tau_U$ is uniformly bounded on $U$ and consequently that $\rho_U$ is $\mu_U$-integrable. So, by Birkhoff's theorem, there exist $\mu_{U}$-integrable functions $\rho_U{}^*: U\to H_1(S,\R)$ and $\tau_U{}^*: U\to \R$ such that for $\mu_U$-almost every point $z$ it holds that 
\begin{equation}\label{eq:def*}
\lim_{n\to+\infty} \frac1n \sum_{k=0}^{n-1} \rho_U( \varphi_U ^k(z))= \rho_U{}^*(z),\quad
\lim_{n\to+\infty} \frac1n \sum_{k=0}^{n-1} \tau_U( \varphi^k(z))= \tau_U{}^*(z).
\end{equation}
These quantities are related to the rotation number by the fact that for $\mu_U$-almost every point $z$, we have $\mathrm{rot}_f(z)=\rho_U{}^*(z)/\tau_U{}^*(z)$.

\subsection{The main theorem}\label{ss.maintheorem}

Let us begin this section by introducing the notion of {\it homotopical interval of rotation}. If $S$ is an oriented closed surface, denote ${\mathcal {FHL}}(S)$ the free homotopy  loop space of $S$. For every $\kappa\in {\mathcal {FHL}}(S)$ and every $\Gamma\in\kappa$, the homology class $[\Gamma]\in H_1(S,\Z)$ does not depend on the choice of $\Gamma$, we denote it $[\kappa]$. If $\Gamma:\R/\Z\to S$ is a loop and $k$ an integer, we can define the loop $\Gamma^k: t\mapsto \Gamma(kt)$.   For every $\kappa \in {\mathcal {FHL}}(S)$, every  $\Gamma\in\kappa$ and every $k\in\Z$, the free homotopy class of $\Gamma^k$ does not depend on the choice of $\Gamma$, we denote it $\kappa^k$. A homotopical interval of rotation of $f\in\mathrm{Homeo}_*(S)$ is a couple $(\kappa, r)$, where $\kappa\in {\mathcal {FHL}}(S)$ and $r$ is a positive integer, that satisfies the following: there exists an integer $s>0$ such that for every $p/q\in[0,1]\cap \Q$, one can find a point $z\in S$ of period at least $q/s$, such that the loop naturally defined by $I^{rq}(z)$ belongs to $\kappa^p$. In particular, we have $\mathrm{rot}_f(z) = p/(rq)[\kappa]$.

\medskip

Let us state the main result of the article.

\begin{theoreme} \label{th:main}
Let $ S$ be an oriented closed surface of genus $g\geq 2$. If $f\in \mathrm{Homeo}_*(S)$ preserves a Borel probability measure $\lambda$ such that
$\mathrm{supp}(\lambda)=S$ and $\mathrm{rot}_{f}(\lambda)\in \R H_1(S,\Z)$, then $f$ has infinitely many periodic points. 

More precisely, for every ergodic measure $\nu\in{\mathcal M}(f)$ that is not a Dirac measure at a contractible fixed point and every neighborhood $\mathcal U$ of $\mathrm{rot}_{f}(\nu)$ in $H_1(S,\R)$, there exists a homotopical interval of rotation $(\kappa, r)$ such that $[\kappa]/r\in \mathcal U$.
\end{theoreme}

Note that if $f$ satisfies the hypotheses of the theorem and is different from identity, then by ergodic decomposition it has an ergodic invariant probability measure $\nu$ that is not supported on a fixed point. 
Theorem~\ref{th:main} applies and implies the existence of a homotopical interval of rotation; in particular $f$ has an infinite number of periodic points, of arbitrarily large period, and of rotation vector arbitrarily close to 0. 
If $\mathrm{rot}_f(\lambda)\not=0$, the measure $\nu$ can be chosen such that $\mathrm{rot}_f(\nu)\not=0$ and consequently, $f$ has periodic orbits of arbitrary large period and with non zero rotation vector.
In any case, any ergodic Borel probability measure, supported on a contractible fixed point or not, has its rotation vector approximated by rotation vectors of an infinite number of periodic points. Remark that this property is also true for $f$ equal to the identity.

Before explaining what are the two different sources of creation of homotopical interval of rotation in Paragraph~\ref{ss.idea}, let us comment Theorem~\ref{th:main}. We start by giving a direct application.
If $\omega$ is a smooth area form on $S$, denote $\Diff^r_{\omega}(S)$, $1\leq r\leq \infty$, the space of $C^r$ diffeomorphisms of $S$ preserving $\omega$, endowed with the $C^r$-topology, and $\Diff^r_{\omega, *}(S)$ the connected component of $\Diff^r_{\omega}(S)$ that contains the identity. It is a classical fact that $\Diff^r_{\omega, *}(S)=\Diff^r_{\omega}(S)\cap \Homeo_*(S)$.

\begin{corollary} \label{co:intro}
Suppose that $g\geq 2$. Then, for any $1\leq r\leq \infty$, the set of maps $f\in\Diff^r_{\omega, *}(S)$ that have infinitely many periodic points is dense in  $\mathrm{Diff}^r_{\omega, *}(S)$.
\end{corollary}

\begin{proof}
There is no loss of generality by supposing that the measure $\mu_{\omega}$ naturally defined by $\omega$ is a probability measure. Note that the map $f\mapsto \mathrm{rot}_f(\mu_{\omega})$ is a morphism defined on  $\Diff^r_{\omega}(S)$. One can find a family of simple loops $(\Gamma_i)_{1\leq i \leq 2g}$ in $S$ such that the family $([\Gamma_i])_{1\leq i\leq 2g}$ generates $H_1(S, \R)$. For every $i\in\{1,\dots, 2g\}$ consider a closed tubular neighborhood $W_i$ of $\Gamma_i$. 
It is easy to construct a divergence free smooth vector field $\zeta_i$ supported on $W_i$ with an induced flow $(h_i^t)_{t\in\R}$ satisfying $\mathrm{rot}_{h_i^t}(\mu_{\omega})=t[\Gamma_i]$.  For every $t=(t_1, \dots, t_{2g})\in\R^{2g}$, define $f^t= h_1^{t_1}\circ \dots \circ h_{2g}^{t_{2g}}\circ f$. We have 
$$\mathrm{rot}_{f^t}(\mu_{\omega}) = \mathrm{rot}_f(\mu_{\omega}) + \sum_{i=1}^{2g} \mathrm{rot}_{h_i^{t_i}}(\mu_{\omega}) = \mathrm{rot}_f(\mu_{\omega}) + \sum_{i=1}^{2g} t_i [\Gamma_i].$$
So, we can find $t$ ``arbitrarily small'' such that $\mathrm{rot}_{f^t}(\mu_{\omega})\in  H_1(S,\Q)$. 
\end{proof}

\begin{remark*}
A very close version of the theorem has been proved independently by Rohil Prasad.
A very strong recent result of Cristofer-Prasad-Zhang \cite{CPrZ}, whose proof uses Periodic Floer Homology theory, asserts that if $\omega$ is a smooth area form on $S$, then for every $k\in \N\cup\{\infty\}$, the set of maps $f\in\Diff^{k}_{\omega}(S)$ that have a dense set of periodic points is dense in  $\mathrm{Diff}^{k}_{\omega}(S)$ (which of course implies that Corollary \ref{co:intro} holds in the smooth category, see also \cite {EH} and \cite{CPoPrZ}).
The following result is used in their proof: in the case where $f\in\Diff^{\infty}_{\omega,*}(S)$ and $\mathrm{rot}_f(\mu_{\omega})\in H_1(S,\Q)\setminus\{0\}$, the map $f$ has a periodic orbit with non zero rotation vector. Moreover they find an explicit upper bound of the period related to $\mathrm{rot}(\mu_{\omega})$ and to the genus of $S$.
As explained by Prasad \cite{Pr} in a recent note, a simple approximation process permits to extend this result to the case where $f\in\mathrm{Homeo}_*(S)$ preserves $\mu_{\omega}$ and satisfies $\mathrm{rot}_f(\mu_{\omega})\in H_1(S,\Q)\setminus\{0\}$.
Moreover a blow-up argument allows to extend the result in the case where $\mathrm{rot}_f(\mu_{\omega})\in \R H_1(S,\Z)\setminus\{0\}$.
Consequently it holds that $f$ has infinitely many periodic orbits of period arbitrarily large. This last point is a consequence of previous works where area preserving homeomorphisms with finitely many periodic points are characterized (\cite{AdT} in the case of the torus, \cite{Lec3} in the case of surfaces with higher genus). Using Oxtoby-Ulam theorem \cite{OxU} and the fact that every invariant probability measure is the barycenter of two invariant probability measures, the first one atomic and the second one with no atom, the measure $\mu_{\omega}$ can be replaced with any probability measure with total support.  In the present article, we give some precisions about the structure of the periodic points.
\end{remark*}

\begin{remark*} 
The theorem is untrue in the sphere and in the torus. Indeed, suppose that  $\alpha\in\R\setminus\Q$. 

The diffeomorphism $f_{\alpha}$ of the Riemann sphere $\S^2$ defined as follows
$$ f_{\alpha}(z)=\begin{cases} \infty & \mathrm{if} \enskip z=\infty,\\
e^{2i\pi\alpha} z&\mathrm{if} \enskip z\in\C,
\end{cases}$$
preserves a probability measure $\mu_{\omega}$ associated to an area form and has no periodic point but $0$ and $\infty$. If $I$ is an identity isotopy of $f$, then $\mathrm{rot}_I(\mu_{\omega})=0$ because $H_1(\S^2,\R)=0$. 

The diffeomorphism  
$$\begin{aligned} g_{\alpha}: \R^2/\Z^2&\longrightarrow \R^2/\Z^2\\
(x,y)&\longmapsto (x+(\alpha+\Z),y)\end{aligned}$$
preserves the area form $\omega=dx\wedge dy$ and has no periodic orbit. If $I=(R_{t\alpha})_{t\in[0,1]}$, then we have $\mathrm{rot}_I(\mu_{\omega})=\alpha(1,0)\in\R H_1(\T^2,\Z)$.   
\end{remark*}

\begin{remark*} 
In particular, the theorem asserts that if $\mathrm{rot}_{f}(\lambda)=0$, then there exists infinitely many periodic orbits. Moreover the set of periods is infinite if $f$ is not the identity because there exist ergodic invariant measures that are not Dirac measures at a fixed point. This result, that admits a version for the case $g=1$, was already known (see \cite{Lec2}). It is a generalization of a result stated in the differential setting (see \cite{FH}) which itself is the two dimensional version of what is called Conley conjecture, later proved in any dimension (see \cite{G}). Note that in \cite{Lec2} it is proved that if $f$ has finitely many fixed points, then there are infinitely many contractible periodic orbits.

 \end{remark*}

\begin{remark*}
The theorem was well known for the time one map of a conservative flow. Indeed, let $X$ be a (time independent) vector field of class $C^1$ whose flow preserves $\omega$. The equalities
$0=L_X\omega= i_X d\omega+di_X\omega$
tell us that the $1$-form $\beta=i_X\omega$ is closed. Moreover it is invariant by the flow of $X$ because $L_X\beta= i_Xd\beta+di_X(i_X\omega)=0$. If $f$ is the time one map of the flow  $(f^t)_{t\in\R}$ of $X$, then, denoting  $I=(f^t)_{t\in[0,1]}$, we know that for every closed $1$-form $\alpha$, we have
$$\begin{aligned}\langle [\alpha], \mathrm{rot}_{I}(\mu_{\omega})\rangle
&= \int_S \left(\int_{I(z)} \alpha\right) \, d\mu_\omega(z)\\
&= \int_S \left(\int_0^1\alpha(X(f_t(z))dt\right) \, d\mu_{\omega}(z)\\
&= \int_0^1\left(\int_S \alpha(X(f_t(z))d\mu_{\omega}(z)\right)dt\\
&=\int_S \alpha(X(z))d\mu_{\omega}(z)\\
 \end{aligned}.$$
Noting that $0=i_X(\alpha\wedge\omega)= i_X\alpha \, \wedge\,\omega -\alpha \wedge i_X \omega$ we deduce that
$$\langle [\alpha], \mathrm{rot}_{I}(\mu_{\omega})\rangle= \int_S \alpha\wedge \beta.$$
The fact that $\mathrm{rot}_{I}(\mu_{\omega})\in \R H_1(S,\Z)$ implies that $[\beta]\in \R H^1(S,\Z)$. Suppose for instance that $[\beta]\in H^1(S,\Z)$. Then there exists a function $H:S\to \R/\Z$ of class $C^2$ such that $\beta=dH$. Indeed, let us fix $z_0\in S$. For every point $z\in M$, the value modulo $1$, denoted $H(z)$,  of $\int_{\gamma} \beta$ does not depend on the $C^1$ path $\gamma$ joining $z_0$ to $z$. We get in that way a function $H:S\to \R/\Z$ of class $C^2$ such that $\beta=d H$. It is invariant by the flow of $X$ because 
$$L_X H = i_X dH+d i_XH=i_X\beta=i_X(i_X\omega)=0.$$

Denote $\mathrm{sing}(X)$ the set of singular points of $X$. Remind that the $\alpha$-limit set $\alpha(z)$ and the $\omega$-limit set $\omega(z)$ of a point $z\in S$ are the sets of subsequential limits of the sequences $(f^{-n}(z))_{n\geq 0} $ and $(f^{n}(z))_{n\geq 0} $ respectively. If $z$ is not singular, either the orbit of $z$ is periodic, or its limit sets $\alpha(z)$ and $\omega(z)$ are contained in $\mathrm{sing}(X)$. In particular the ergodic invariant probability measures that are non supported on a singular point are supported on a periodic orbit of $f$ lying on a periodic orbit of the flow with rational period, or supported on a whole periodic orbit of the flow with irrational period. The union $W$ of periodic orbits of the flow is non empty (by Sard's theorem) and open. Moreover every connected component $V$ of $W$ is annular (meaning homeomorphic to $\R/\Z\times \R$). The genus being at least two, there exist singular points. Furthermore $S$ is not a sphere. It implies that there exists at least one end of $V$ such that for every sequence $(z_n)_{n\geq 0}$ in $V$ converging to this end, the period of $z$ (for the flow) converges to $+\infty$. So the period is not constant on $V$.  It implies that $f$ has periodic points of arbitrarily large period. More precisely, the loops $\Gamma$ that appear in the Theorem are the simple loops contained in such a component $V$ that are non homotopic to zero in $V$ and suitably oriented.  Note that if $\mathrm{rot}_I(\mu_{\omega})\not=0$, there exits at least one connected component $V$ of $W$ such that $i_*(H_1(V,\Z))\not=\{0\}$, where $i_*:H_1(V,\Z)\to H_1(S, \Z)$ is the morphism naturally defined by the inclusion map $i:V\to S$, meaning that the periodic points in $V$ have non zero rotation vector. 
\end{remark*}

\begin{remark*}
The hypothesis $\mathrm{rot}_{f}(\lambda)\in \R H_1(S,\Z)$ is necessary to get the theorem. Indeed one can find smooth vector fields with finitely many singular points, whose flows preserves an area form $\omega$ and such that every orbit is dense if not reduced to a singular point. The time one map of this flow $f$ has no periodic points but the singular points. Of course it holds that $\mathrm{rot}_{f}(\lambda)\not\in \R H_1(S,\Z)$. Classical examples are given by translation flows in a minimal direction.  
\end{remark*}

\begin{remark*}
Corollary \ref{co:intro} was already known. In fact we have a much stronger result: the set of maps $f\in\Diff^r_{\omega}(S)$ that have a hyperbolic periodic point with transverse homoclinic intersection, is an open and dense subset of $\Diff^r_{\omega}(S)$  (see \cite{LecSa}). This result has been known for a long time in the case  where $g\leq 1$ (see \cite{Ad}, \cite{AdT}, \cite{D}, \cite{Ol1}, \cite{Ol2}, \cite{Pi}, \cite{R}). A difficult step in the proof of the case $g\geq 2$  is to show that the set of maps $f\in\Diff^r_{\omega,*}(S)$ having at least $2g-1$ periodic points is dense in $\Diff^r_{\omega,*}(S)$. \end{remark*}

\subsection{Idea of the proof}\label{ss.idea}

The main tool of the proof is the forcing theory developed in \cite{LecT1, LecT2, Lel}, which we introduce in Paragraphs~\ref{ss.intersections} and \ref{ss.maximal}. Using this tool, we analyse the possible configurations that can occur under the hypotheses of Theorem~\ref{th:main}. In most of the cases, we will find a rotational horseshoe (defined in Paragraph~\ref{ss:horseshoe}), which will allow us to get the conclusion of the theorem.  In only one case we will not be able to find such a horsheshoe and indeed, there are some examples of homeomorphisms satisfying the hypotheses of Theorem~\ref{th:main} and without topological horseshoe, for example time one maps of area preserving flows. The conclusion will be obtained using an improved version of Poincaré-Birkhoff Theorem~\ref{th:PB} in a suitable annulus. Caratheodory's theory of prime ends (see Paragraph~\ref{ss:Caratheodory}) will be used in this last case.

More precisely, one can find a suitable identity isotopy $I$ of $f$ and a singular oriented foliation $\F$ on $S$ whose regular set coincide with the set $\mathrm{dom}(I)$ of points with non trivial trajectory under the isotopy, that satisfy the following fundamental property: every non trivial trajectory $I(z)$ is homotopic in $\mathrm{dom}(I)$ to a path transverse to $\F$. Given an $f$-invariant ergodic probability measure $\nu$ such that $\nu(\mathrm{dom}(I))=1$, the proof starts by building an \emph{approximation} of a typical orbit for $\nu$ (Lemma~\ref{le:goodstrip}): it is an oriented loop $\Gamma_*$ transverse to $\F$, such that $[\Gamma_*]$ is close to $\mathrm{rot}_f(\nu)$, and such that, for $\nu$-almost every point $z$, the transverse path defined naturally by the whole orbit of $z$ draws this loop.  We will consider an annular covering space $\widehat{\mathrm{dom}}(I)$ of $\mathrm{dom}(I)$ where  $\Gamma_*$ is lifted to a non contractible simple loop $\hat\Gamma_*$. The isotopy $I|_{ \mathrm{dom}(I)}$ and the foliation $\F$ can be lifted to $\widehat{\mathrm{dom}}(I)$. The union of leaves that meet $\hat\Gamma_*$ is an open annulus $\tilde B$. Depending of the properties of the trajectories of typical points for the measure $\nu$ with respect to this annulus $\tilde B$, we get different conclusions: if they cross or visit this annulus (see Paragraph~\ref{ss:strips} for definitions), then we are able to find a topological rotational horseshoe, by means of the forcing theory results proved in Paragraph~\ref{sec:Dyna}; if they stay forever in this annulus then we prove that Poincaré-Birkhoff Theorem~\ref{th:PB} applies and implies the existence of an infinite number of periodic orbits. 

We strongly use, or develop, the results proved by Gabriel Lellouch in his PhD thesis \cite{Lel}. In particular we will need the main result of \cite{Lel}, where $\wedge$ denotes the natural intersection form on $H_1(S,\R)$ (see Paragraph \ref{ss.maximal}): if $\mu$ and $\mu'$ are two invariant probability measures such that $\mathrm{rot}_f(\mu)\wedge\mathrm{rot}_f(\mu')\not=0$, then $f$ has a rotational horseshoe. The hypothesis  $\mathrm{rot}_{f}(\lambda)\in \R H_1(S,\Z)$ will be used once: with the help of Atkinson's theorem \cite{At}, it will permit us to assume that $[\Gamma_*]\wedge\mathrm{rot}_{f}(\lambda)=0$.

\subsection{Acknowledgements}

We would like to thank Sobhan Seyfaddini for suggesting us this problem. While ending this article we received the recent note of Rohil Prasad. We would like to thank him for his useful comments.

A.P. was funded by Mathamsud Project TOMCAT 22-MATH-10.

\section{Definitions, notations and preliminaries}\label{s:preliminaries}

In the sequel, the letter $S$ will refer to a closed surface while the letter $\Sigma$ will refer to any surface (not necessarily compact, not necessarily connected). If $f$ is a surface homeomorphism, $\mu$ will refer to any $f$-invariant measure, $\lambda$ to an $f$-invariant measure with total support, and $\nu$ to an $f$-invariant ergodic measure.

\subsection{Loops and paths}\label{ss:loopsandpaths}

Let $\Sigma$ be an oriented surface (not necessarily closed, not necessarily boundaryless, not necessarily connected).  A {\it loop} of $\Sigma$ is a continuous map $\Gamma: \T\to \Sigma$, where $\T=\R/\Z$.  It will be called {\it essential} if it is not homotopic to a constant loop. A {\it path} of $\Sigma$ is a continuous map $\gamma: I\to \Sigma$ where $I\subset \R$ is an interval. A loop or a path will be called {\it simple} if it is injective. The {\it natural lift} of a loop $\Gamma: \T\to \Sigma$ is the path $\gamma:\R\to \Sigma$ such that $\gamma(t)=\Gamma(t+\Z)$. A {\it segment} is a simple path $\sigma:[a,b]\to \Sigma$, where $a<b$. The points $\sigma(a)$ and $\sigma(b)$ are the {\it endpoints} of $\sigma$. We will say that $\sigma$ {\it joins} $\sigma(a)$ to $\sigma(b)$. More generally if $A$ and $B$ are disjoint, we will say that $\sigma$ joins $A$ to $B$, if $\sigma(a)\in A$ and $\sigma(b)\in B$. A {\it line} is a proper simple path $\lambda:\R\to \Sigma$. As it is usually done we will use the same name and the same notation to refer to the image of a loop or a path $\gamma$. 

Note that a simple loop or a simple path is naturally oriented. Let $\Gamma$ be a simple loop of $\Sigma$, and denote $\Sigma'$ the connected component of $\Sigma$ it belongs to. If $\Sigma'\setminus\Gamma$ has two connected components, we say that $\Gamma$ \emph{separates} $\Sigma$; in this case the connected component that is located on the right of $\Gamma$ will be denoted $R(\Gamma)$ and the other one $L(\Gamma)$. We will use the same notations $R(\lambda)$, $L(\lambda)$ for a line $\lambda$ that separates the connected component it belongs to.

Let $f$ be an orientation preserving homeomorphism of $\Sigma$.  A {\it Brouwer line} of $f$ is a line $\lambda$ that separates $\Sigma$ such that $f(\lambda)\subset L(\lambda)$ and $f^{-1}(\lambda)\subset  R(\lambda)$. Equivalently it means that $f(\overline {L(\lambda)})\subset L(\lambda)$ or that $f^{-1}(\overline {R(\lambda)})\subset R(\lambda)$.

\subsection{Poincar\'e-Birkhoff theorem }\label{s:PB}
 
Let us consider the annulus $\A=\T\times I$, where $(0,1)\subset I\subset [0,1]$, and its universal covering space $\tilde \A=\R\times I$. We define the covering projection  $\tilde\pi:\enskip  (x,y)\mapsto (x+\Z,y) $ and the generating covering automorphism $T:(x,y)\mapsto (x+1,y)$. We denote $\tilde p_1: \tilde \A\to\R$ the projection on the first factor.

Let $f$ be a homeomorphism of $\A$  isotopic to the identity (meaning orientation preserving and fixing the boundary circles or ends) and $\tilde f$ a lift of $f$ to $\tilde \A$. The map $p_1\circ \tilde f-p_1$ lifts a continuous function $\psi_{\tilde f} : \A\to\R$ because  $\tilde f$ and $T$ commute. In particular, for every $ z\in\A$, for every lift $\tilde z\in \tilde\A$ of $z$ and every $n\geq 1$, we have 
$$\sum_{i=0}^{n-1} \psi_{\tilde f} (f^i(z))=p_1(\tilde f^n(\tilde z))-p_1(\tilde z).$$ 
Let $z$ be a positively recurrent point. Say that $f$ \emph{has $\mathrm{rot}_{\tilde f}(z)\in\R$ as a rotation number} if for every subsequence $(f^{n_k}(z))_{k\geq 0}$ of $(f^n(z))_{n\geq 0}$ that converges to $z$, we have 
$$\lim_{k\to+\infty} \frac1{n_k}\sum_{i=0}^{n_k-1} \psi_{\tilde f} (f^i(z))=\mathrm{rot}_{\tilde f}(z).$$
If $O$ is a periodic point of $f$ of period $q$, then there exists $p\in\Z$ such that for every $\tilde z\in\tilde \pi^{-1}(O)$ we have $\tilde f^q(\tilde z)=T^p(\tilde z)$. In this case, $p/q$ is the rotation number of $O$ for the lift $\tilde f$. 
We will use the following extension of the classical Poincar\'e-Birkhoff Theorem (see for example \cite{Lec1}):

\begin{theorem} \label{th:PB} 
Let $f$ be a homeomorphism of $\A$  isotopic to the identity and $\tilde f$ a lift of $f$ to $\tilde \A$. We suppose that there exist two positively recurrent points  $z_1$ and $z_2$, such that  $\mathrm{rot}_{\tilde f} (z_1)<\mathrm{rot}_{\tilde f} (z_2)$. Then: 
\begin{itemize} 
\item either, for every rational number $p/q\in (\mathrm{rot}_{\tilde f} (z_1),\mathrm{rot}_{\tilde f} (z_2))$, written in an irreducible way, there exists a periodic orbit $O$ of $f$ of period $q$ and rotation number $p/q$ for $\tilde f$;
\item or there exists an essential simple loop $\Gamma\subset\T\times (0,1)$ such that $f(\Gamma)\cap\Gamma=\emptyset$.
\end{itemize}
\end{theorem}

Of course, we have a similar result in an abstract annulus, meaning a topological space homeomorphic to $\A$.

\subsection{Homeomorphisms of hyperbolic surfaces}

Let $\Sigma$ be a connected oriented hyperbolic surface without boundary, meaning different from the sphere, the plane, the open annulus or the torus.
One can furnish $\Sigma$ with a complete Riemannian metric of constant negative curvature $-1$. The universal covering space of $\Sigma$ is the disk $\D=\{z\in \C\,\vert\, \vert z\vert <1\}$ and the group of covering transformations, denoted ${\mathcal G}$, is composed of  M\H obius automorphisms of $\D$. One can suppose that the metric is of first type, meaning that the closure in $\C$ of every ${\mathcal G}$-orbit contains $\S_1=\{z\in \C\,\vert\, \vert z\vert=1\}$ (see \cite{Mats2} for instance). Every hyperbolic element $T\in {\mathcal G}$ can be extended to a homeomorphism of $\overline{\D}$ having two fixed points on the boundary: a repelling fixed point $\alpha(T)$ and an attracting fixed point $\omega(T)$. For every $z\in \overline{\D}\setminus\{\alpha(T), \omega(T)\}$, it holds that
$$\lim_{k\to -\infty} T^kz =\alpha(T),\quad \lim_{k\to +\infty} T^kz =\omega(T).$$
The metric being of first type, the set of points $\alpha(T)$ and the set of points $\omega(T)$, $T$ among all hyperbolic automorphism, is dense in $\S_1$.
Every parabolic element $T\in {\mathcal G}$ can be extended to a homeomorphism of $\overline{\D}$ having one fixed point $\alpha\omega(T)$ on the boundary. For every $z\in \overline{\D}\setminus\{\alpha\omega(T)\}$, it holds that
$$\lim_{k\to \pm\infty} T^kz =\alpha\omega(T).$$

A homeomorphism $f$ of $\Sigma$ isotopic to the identity has a unique lift $\tilde f$ to $\D$ that commutes with the covering automorphisms. We will call it the {\it canonical lift} of $f$. It is well known that $\tilde f$ extends to a homeomorphism $\overline{\tilde f}$ of  $\overline{\D}$ that fixes every point of $\S_1$. If $T\in {\mathcal G}$ is hyperbolic, then $\tilde f$  lifts a homeomorphism $\hat f$ of $\hat \Sigma=\tilde \Sigma/T$. Moreover $\hat f$ extends to a homeomorphism of the compact annulus $\overline{\hat \Sigma}$ obtained by adding the two circles $\hat J= \tilde J/T$ and $\hat J'= \tilde J'/T$, where $\tilde J$ and $\tilde J'$ are the two connected components of $\S_1\setminus\{\alpha(T), \omega(T)\}$. Note that every point of $\hat J\cup \hat J'$ is fixed, with a rotation number equal to zero for the lift $\overline{\tilde f}\vert_{\overline{\D}\setminus\{\alpha(T),\,\omega(T)\}}$. Similarly, if $T\in {\mathcal G}$ is parabolic, then $\tilde f$  lifts a homeomorphism $\hat f$ of $\hat \Sigma=\tilde \Sigma/T$ that extends to a homeomorphism of $\overline{\hat \Sigma}$ obtained by adding the circle $(\S_1\setminus\{\alpha\omega(T)\})/T$ at one end of $\hat \Sigma$. Every point of this circle is fixed, with a rotation number equal to zero for the lift $\overline{\tilde f}\vert_{\overline{\D}\setminus\{\alpha\omega(T)\}}$.

\subsection{Caratheodory theory of prime ends}\label{ss:Caratheodory}

In this small subsection we state a result that will be used once in the article, consequence of what is called {\it prime end theory} (see  \cite{Math} for instance). Let $S$ be a closed surface of genus $\geq 1$ and $U$ an open annulus of $S$. Say that an end $e$ of $U$ is {\it singular} if there exists a point $z\in S$ and a neighborhood of $e$ in $U$ that is a punctured neighborhood of $z$ in $S$. Otherwise say that $e$ is {\it regular}. There is at least one regular end because $S$ is not the $2$-sphere. Suppose that $U$ is invariant by an orientation preserving homeomorphism $f$. Then the homeomorphism $f|_{ U}$ extends to a homeomorphism $\overline f_{U}$ of a larger annulus $\overline U_{\mathrm{pe}}$ obtained by blowing up each regular end of $U$ and replace it with the associated circle of prime ends. Moreover if $U$ is a connected component of the complement of a closed subset $X$ of $\mathrm{fix}(f)$, then the extended map fixes each point of the circles of prime ends. More precisely, suppose that $I=(f_t)_{t\in[0,1]}$ is an identity isotopy of $f$, such that $f_t(U)=U$ and $X\subset\mathrm{fix}(f_t)$ for every $t\in[0,1]$. Then, the rotation number of the points on the added circles (they are fixed) is equal to $0$, for the lift of $\overline f_{U}$ to the universal covering space of $\overline U_{\mathrm{pe}}$, that extends the lift of $f|_{ U}$ to the universal covering space of $U$, naturally defined by $I|_{ U}$.

\subsection{Rotational topological horseshoes}\label{ss:horseshoe}

Let $\Sigma$ be a connected oriented surface. Say that $Y\subset S$ is a {\it topological horseshoe} of $f\in \mathrm{Homeo}_*(S)$ if $Y$ is closed, invariant by a power $f^r$ of $f$, and if $f^r{}|_{ Y}$ admits a finite extension $g:Z\to Z$ on a Hausdorff compact space $Z$ such that:
\begin{itemize}
\item $g$ is an extension of the Bernouilli shift $\sigma :\{1,\dots, m\}^{\Z}\to \{1,\dots, m\}^{\Z}$, where $m\geq 2$;
\item the preimage of every $s$-periodic sequence of $\{1,\dots,m\}^{\Z}$  by the factor map contains  at least one $s$-periodic point of $g$.
\end{itemize}
It means that $g$ is a  homeomorphism of $Z$ that is semi-conjugated to $f^r{}|_{ Y}$ and that the fibers of the factor map are all finite with an uniform bound $M$ in their cardinality. Note that, if $h(f)$ denotes the topological entropy of $f$, then it holds that
$$ rh(f)=h(f^r)\geq  h(f^r{}|_{ Y})=h(g)\geq h(\sigma)=\log q,$$and that $f^{r}$ has at least $q^n/M$ fixed points for every $n\geq 1$.  

Suppose now that $S$ is a connected closed oriented surface. Say that a  topological horseshoe $Y$ of $f\in \mathrm{Homeo}_*(S)$ is a {\it rotational topological horseshoe of type $(\kappa,r)$}, where $\kappa\in {\mathcal {FHL}}(S)$ and $r$ is a positive integer, if there exists a positive integer $s$ such that for every $p/q\in[0,1]\cap \Q$, there exists a point $z\in Y$ of period at least $q/s$, such that the loop naturally defined by $I^{rq}(z)$ belongs to $\kappa^p$. In particular the horseshoe defines a homotopical interval of rotation. The rotational topological horseshoes that appear in the present article will be constructed in an annular covering of an invariant open set,  
satisfying the geometric definition given in \cite{PaPotSa}.

\section{Foliations on surfaces}\label{s.foliations}

In this section we will consider an oriented boundaryless surface $\Sigma$, not necessarily closed, not necessarily connected, and a non singular oriented topological foliation $\F$ on $\Sigma$. We will consider:
\begin{itemize}
\item the universal covering space $\tilde\Sigma$ of $\Sigma$;
\item the covering projection $\tilde\pi: \tilde\Sigma\to \Sigma$;
\item the group ${\mathcal G}$ of covering automorphisms;
\item the lifted foliation $\tilde{\F}$ on $\tilde{\Sigma}$.
\end{itemize}

For every point $z\in\Sigma$, we denote $\phi_{z}$ the leaf of $\F$ that contains $z$. If $\phi_z:\R\to\Sigma$ is a parametrization of $\phi_z$ inducing the orientation, such that $\phi_z(0)=z$, we set $\phi_z^+=\phi_z\vert_{[0,+\infty)}$ and $\phi_z^-=\phi_z\vert_{(-\infty,0]}$. Similarly, for every point $\tilde z\in\tilde\Sigma$, we denote $\tilde\phi_{\tilde z}$ the leaf of $\tilde{\F}$ that contains $\tilde z$ and we define in the same way $\tilde\phi_{\tilde z}^+$ and $\tilde\phi_{\tilde z}^-$.

\subsection{$\F$-transverse intersections}\label{ss.intersections}

A path $\gamma:J\to\Sigma$ is \emph {positively transverse}\footnote{In the whole text, ``transverse'' will mean ``positively transverse''.} to $\F$ if it locally crosses each leaf of $\F$ from the right to the left. Observe that every lift  $\tilde\gamma:J\to\tilde\Sigma$ of $\gamma$ is positively transverse to $\tilde{\F}$ and that for every $a<b$ in $J$: 
\begin{itemize}
\item $\tilde\gamma|_{[a,b]}$ meets once every leaf $\tilde\phi$ such that $R(\tilde\phi_{\tilde \gamma(a)})\subset R(\tilde\phi)\subset R(\tilde\phi_{\tilde \gamma(b)})$;
\item  $\tilde\gamma|_{[a,b]}$ does not meet any other leaf.
\end{itemize}
Two transverse paths $\tilde\gamma_1:J_1\to\tilde\Sigma$ and $\tilde\gamma_2:J_2\to\tilde\Sigma$ are said \emph{equivalent} if they meet the same leaves of $\tilde{\F}$. Two transverse paths $\gamma_1:J_1\to\Sigma$ and $\gamma_2:J_2\to\Sigma$ are \emph{equivalent} if there exists a lift  $\tilde\gamma_1:J_1\to\tilde\Sigma$ of $\gamma$ and a lift $\tilde\gamma_2:J_2\to\tilde\Sigma$ of $\gamma_2$ that are equivalent.

Let $\tilde\gamma_1:J_1\to \tilde\Sigma$ and $\tilde\gamma_2:J_2\to \tilde \Sigma$ be two transverse paths such that there exist $t_1\in J_1$ and $t_2\in J_2$ satisfying $\tilde\gamma_1(t_1)=\tilde\gamma_2(t_2)$. We will say that $\tilde\gamma_1$ and $\tilde\gamma_2$ have a \emph{$\tilde{\F}$-transverse intersection} at $\tilde\gamma_1(t_1)=\tilde\gamma_2(t_2)$ if there exist $a_1, b_1\in J_1$ satisfying $a_1<t_1<b_1$  and $a_2, b_2\in J_2$ satisfying $a_2<t_2<b_2$ such that:

\begin{itemize}
\item $\tilde\phi_{\tilde\gamma_1(a_1)}\subset L(\tilde\phi_{\tilde\gamma_2(a_2)}),\enskip \tilde\phi_{\tilde\gamma_2(a_2)}\subset L(\tilde\phi_{\widetilde\gamma_1(a_1)})$;

\item $\tilde\phi_{\widetilde\gamma_1(b_1)}\subset R(\tilde\phi_{\tilde\gamma_2(b_2)}),\enskip \tilde\phi_{\tilde\gamma_2(b_2)}\subset R(\tilde\phi_{\widetilde\gamma_1(b_1)})$;

\item every path joining $\tilde\phi_{\tilde\gamma_1(a_1)}$ to $\tilde\phi_{\tilde\gamma_1(b_1)}$ and every path joining $\tilde\phi_{\tilde\gamma_2(a_2)}$ to $\tilde\phi_{\tilde\gamma_2(b_2)}$ must intersect. 
\end{itemize}

\begin{figure}
\begin{center}

\tikzset{every picture/.style={line width=0.6pt}} 

\begin{tikzpicture}[x=0.75pt,y=0.75pt,yscale=-1.1,xscale=1.1]

\draw [color={rgb, 255:red, 245; green, 166; blue, 35 }  ,draw opacity=1 ]   (261.43,50.73) .. controls (259.25,85.76) and (271.02,112.03) .. (273.64,145.74) .. controls (276.25,179.45) and (300.23,213.16) .. (286.72,229.8) ;
\draw [shift={(266.37,101.91)}, rotate = 258.81] [fill={rgb, 255:red, 245; green, 166; blue, 35 }  ,fill opacity=1 ][line width=0.08]  [draw opacity=0] (8.04,-3.86) -- (0,0) -- (8.04,3.86) -- (5.34,0) -- cycle    ;
\draw [shift={(285.23,191.05)}, rotate = 251.75] [fill={rgb, 255:red, 245; green, 166; blue, 35 }  ,fill opacity=1 ][line width=0.08]  [draw opacity=0] (8.04,-3.86) -- (0,0) -- (8.04,3.86) -- (5.34,0) -- cycle    ;
\draw [color={rgb, 255:red, 245; green, 166; blue, 35 }  ,draw opacity=1 ]   (156.35,52.92) .. controls (164.64,69.56) and (165.07,76.12) .. (155.92,84.88) .. controls (146.76,93.64) and (130.19,100.2) .. (120.6,95.83) ;
\draw [shift={(162.66,72.57)}, rotate = 263.3] [fill={rgb, 255:red, 245; green, 166; blue, 35 }  ,fill opacity=1 ][line width=0.08]  [draw opacity=0] (8.04,-3.86) -- (0,0) -- (8.04,3.86) -- (5.34,0) -- cycle    ;
\draw [shift={(136.28,95.94)}, rotate = 340.46] [fill={rgb, 255:red, 245; green, 166; blue, 35 }  ,fill opacity=1 ][line width=0.08]  [draw opacity=0] (8.04,-3.86) -- (0,0) -- (8.04,3.86) -- (5.34,0) -- cycle    ;
\draw [color={rgb, 255:red, 245; green, 166; blue, 35 }  ,draw opacity=1 ]   (391,50.6) .. controls (378.79,61.98) and (385.69,73.93) .. (392.23,82.25) .. controls (398.77,90.57) and (420.57,95.39) .. (430.6,85.32) ;
\draw [shift={(385.1,69.68)}, rotate = 262.21] [fill={rgb, 255:red, 245; green, 166; blue, 35 }  ,fill opacity=1 ][line width=0.08]  [draw opacity=0] (8.04,-3.86) -- (0,0) -- (8.04,3.86) -- (5.34,0) -- cycle    ;
\draw [shift={(414.19,90.98)}, rotate = 185.63] [fill={rgb, 255:red, 245; green, 166; blue, 35 }  ,fill opacity=1 ][line width=0.08]  [draw opacity=0] (8.04,-3.86) -- (0,0) -- (8.04,3.86) -- (5.34,0) -- cycle    ;
\draw [color={rgb, 255:red, 245; green, 166; blue, 35 }  ,draw opacity=1 ]   (209.98,50.29) .. controls (206.49,112.03) and (133.24,101.96) .. (135.86,135.67) .. controls (138.48,169.38) and (228.73,165.88) .. (216.09,224.11) ;
\draw [shift={(176.16,100.54)}, rotate = 329.51] [fill={rgb, 255:red, 245; green, 166; blue, 35 }  ,fill opacity=1 ][line width=0.08]  [draw opacity=0] (8.04,-3.86) -- (0,0) -- (8.04,3.86) -- (5.34,0) -- cycle    ;
\draw [shift={(189.24,174.67)}, rotate = 208.12] [fill={rgb, 255:red, 245; green, 166; blue, 35 }  ,fill opacity=1 ][line width=0.08]  [draw opacity=0] (8.04,-3.86) -- (0,0) -- (8.04,3.86) -- (5.34,0) -- cycle    ;
\draw [color={rgb, 255:red, 245; green, 166; blue, 35 }  ,draw opacity=1 ]   (322.47,48.1) .. controls (322.47,103.27) and (410.46,94.79) .. (412.2,125) .. controls (413.94,155.21) and (323.34,172.01) .. (347.76,227.17) ;
\draw [shift={(364.02,94.87)}, rotate = 202.65] [fill={rgb, 255:red, 245; green, 166; blue, 35 }  ,fill opacity=1 ][line width=0.08]  [draw opacity=0] (8.04,-3.86) -- (0,0) -- (8.04,3.86) -- (5.34,0) -- cycle    ;
\draw [shift={(364.67,172.23)}, rotate = 320.57] [fill={rgb, 255:red, 245; green, 166; blue, 35 }  ,fill opacity=1 ][line width=0.08]  [draw opacity=0] (8.04,-3.86) -- (0,0) -- (8.04,3.86) -- (5.34,0) -- cycle    ;
\draw [color={rgb, 255:red, 84; green, 0; blue, 213 }  ,draw opacity=1 ][line width=1.5]    (155.92,84.88) .. controls (175.97,129.54) and (235.71,91.89) .. (273.64,145.74) .. controls (311.57,199.59) and (356.04,167.19) .. (381.77,186.02) ;
\draw [shift={(219.54,113.55)}, rotate = 188.99] [fill={rgb, 255:red, 84; green, 0; blue, 213 }  ,fill opacity=1 ][line width=0.08]  [draw opacity=0] (8.75,-4.2) -- (0,0) -- (8.75,4.2) -- (5.81,0) -- cycle    ;
\draw [shift={(327.1,178.63)}, rotate = 187.51] [fill={rgb, 255:red, 84; green, 0; blue, 213 }  ,fill opacity=1 ][line width=0.08]  [draw opacity=0] (8.75,-4.2) -- (0,0) -- (8.75,4.2) -- (5.81,0) -- cycle    ;
\draw [color={rgb, 255:red, 0; green, 116; blue, 201 }  ,draw opacity=1 ][line width=1.5]    (151.12,196.96) .. controls (196.46,209.66) and (233.96,202.65) .. (273.64,145.74) .. controls (313.31,88.82) and (366.07,123.85) .. (392.23,82.25) ;
\draw [shift={(225.84,192.34)}, rotate = 153.5] [fill={rgb, 255:red, 0; green, 116; blue, 201 }  ,fill opacity=1 ][line width=0.08]  [draw opacity=0] (8.75,-4.2) -- (0,0) -- (8.75,4.2) -- (5.81,0) -- cycle    ;
\draw [shift={(335.43,108.69)}, rotate = 169.41] [fill={rgb, 255:red, 0; green, 116; blue, 201 }  ,fill opacity=1 ][line width=0.08]  [draw opacity=0] (8.75,-4.2) -- (0,0) -- (8.75,4.2) -- (5.81,0) -- cycle    ;
\draw [color={rgb, 255:red, 245; green, 166; blue, 35 }  ,draw opacity=1 ]   (121.47,179.01) .. controls (134.99,171.57) and (152.43,182.51) .. (151.12,196.96) .. controls (149.81,211.41) and (137.17,224.98) .. (125.4,224.55) ;
\draw [shift={(144.21,181.45)}, rotate = 211.25] [fill={rgb, 255:red, 245; green, 166; blue, 35 }  ,fill opacity=1 ][line width=0.08]  [draw opacity=0] (8.04,-3.86) -- (0,0) -- (8.04,3.86) -- (5.34,0) -- cycle    ;
\draw [shift={(140.71,217.82)}, rotate = 311.55] [fill={rgb, 255:red, 245; green, 166; blue, 35 }  ,fill opacity=1 ][line width=0.08]  [draw opacity=0] (8.04,-3.86) -- (0,0) -- (8.04,3.86) -- (5.34,0) -- cycle    ;
\draw [color={rgb, 255:red, 245; green, 166; blue, 35 }  ,draw opacity=1 ]   (428.42,198.71) .. controls (412.72,179.45) and (387.87,180.76) .. (381.77,186.02) .. controls (375.66,191.27) and (355.61,205.72) .. (373.92,221.48) ;
\draw [shift={(403.94,183.63)}, rotate = 16.09] [fill={rgb, 255:red, 245; green, 166; blue, 35 }  ,fill opacity=1 ][line width=0.08]  [draw opacity=0] (8.04,-3.86) -- (0,0) -- (8.04,3.86) -- (5.34,0) -- cycle    ;
\draw [shift={(366.62,205.34)}, rotate = 290.52] [fill={rgb, 255:red, 245; green, 166; blue, 35 }  ,fill opacity=1 ][line width=0.08]  [draw opacity=0] (8.04,-3.86) -- (0,0) -- (8.04,3.86) -- (5.34,0) -- cycle    ;
\draw  [fill={rgb, 255:red, 0; green, 0; blue, 0 }  ,fill opacity=1 ] (270.32,145.74) .. controls (270.32,143.9) and (271.81,142.41) .. (273.64,142.41) .. controls (275.47,142.41) and (276.95,143.9) .. (276.95,145.74) .. controls (276.95,147.58) and (275.47,149.07) .. (273.64,149.07) .. controls (271.81,149.07) and (270.32,147.58) .. (270.32,145.74) -- cycle ;
\draw  [fill={rgb, 255:red, 0; green, 0; blue, 0 }  ,fill opacity=1 ] (388.92,82.25) .. controls (388.92,80.41) and (390.4,78.92) .. (392.23,78.92) .. controls (394.06,78.92) and (395.55,80.41) .. (395.55,82.25) .. controls (395.55,84.09) and (394.06,85.58) .. (392.23,85.58) .. controls (390.4,85.58) and (388.92,84.09) .. (388.92,82.25) -- cycle ;
\draw  [fill={rgb, 255:red, 0; green, 0; blue, 0 }  ,fill opacity=1 ] (378.45,186.02) .. controls (378.45,184.18) and (379.94,182.69) .. (381.77,182.69) .. controls (383.6,182.69) and (385.08,184.18) .. (385.08,186.02) .. controls (385.08,187.86) and (383.6,189.35) .. (381.77,189.35) .. controls (379.94,189.35) and (378.45,187.86) .. (378.45,186.02) -- cycle ;
\draw  [fill={rgb, 255:red, 0; green, 0; blue, 0 }  ,fill opacity=1 ] (147.8,196.96) .. controls (147.8,195.12) and (149.29,193.63) .. (151.12,193.63) .. controls (152.95,193.63) and (154.44,195.12) .. (154.44,196.96) .. controls (154.44,198.8) and (152.95,200.29) .. (151.12,200.29) .. controls (149.29,200.29) and (147.8,198.8) .. (147.8,196.96) -- cycle ;
\draw  [fill={rgb, 255:red, 0; green, 0; blue, 0 }  ,fill opacity=1 ] (152.6,84.88) .. controls (152.6,83.04) and (154.09,81.55) .. (155.92,81.55) .. controls (157.75,81.55) and (159.23,83.04) .. (159.23,84.88) .. controls (159.23,86.72) and (157.75,88.21) .. (155.92,88.21) .. controls (154.09,88.21) and (152.6,86.72) .. (152.6,84.88) -- cycle ;

\draw (284.07,145.38) node [anchor=west] [inner sep=0.75pt]  [font=\small]  {$\tilde{\gamma }_{1}( t_{1}) =\tilde{\gamma }_{2}( t_{2})$};
\draw (205.68,114.4) node [anchor=north] [inner sep=0.75pt]  [color={rgb, 255:red, 84; green, 0; blue, 213 }  ,opacity=1 ]  {$\tilde{\gamma }_{1}$};
\draw (244.15,176.67) node [anchor=south east] [inner sep=0.75pt]  [color={rgb, 255:red, 0; green, 116; blue, 201 }  ,opacity=1 ]  {$\tilde{\gamma }_{2}$};
\draw (394.23,78.85) node [anchor=south west] [inner sep=0.75pt]  [font=\small]  {$\tilde{\gamma }_{2}( b_{2})$};
\draw (383.77,189.42) node [anchor=north west][inner sep=0.75pt]  [font=\small]  {$\tilde{\gamma }_{1}( b_{1})$};
\draw (145.8,196.96) node [anchor=east] [inner sep=0.75pt]  [font=\small]  {$\tilde{\gamma }_{2}( a_{2})$};
\draw (153.92,81.48) node [anchor=south east] [inner sep=0.75pt]  [font=\small]  {$\tilde{\gamma }_{1}( a_{1})$};

\end{tikzpicture}

\caption{Example of $\tilde{\F}$-transverse intersection.\label{Fig:extransverse}}
\end{center}
\end{figure}
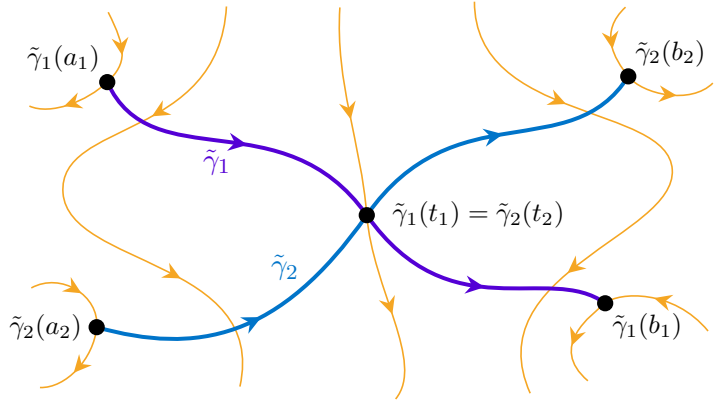

It means that there is a ``crossing'' between the two paths naturally defined by $\tilde\gamma_1$ and $\tilde\gamma_2$ in the space of leaves of $\widetilde{\F}$, which is a one-dimensional topological manifold, usually non Hausdorff (see Figure~\ref{Fig:extransverse}).

\medskip
Now, let  $\gamma_1:J_1\to \Sigma$ and $\gamma_2:J_2\to \Sigma$ be two transverse paths such that there exist $t_1\in J_1$ and $t_2\in J_2$ satisfying $\gamma_1(t_1)=\gamma_2(t_2)$. Say that $\gamma_1$ and $\gamma_2$ have a {\it ${\F}$-transverse intersection} at $\gamma_1(t_1)=\gamma_2(t_2)$ if $\tilde\gamma_1$ and $\tilde\gamma_2$ have a {\it ${\tilde{\F}}$-transverse intersection} at $\tilde\gamma_1(t_1)=\tilde\gamma_2(t_2)$, where $\tilde\gamma_1:J_1\to \tilde \Sigma$ and $\tilde\gamma_2:J_2\to \tilde  \Sigma$  are lifts of $\gamma_1$ and $\gamma_2$ such that  $\tilde\gamma_1(t_1)=\tilde\gamma_2(t_2)$.  If $\gamma_1=\gamma_2$ one speaks of a {\it $\F$-transverse self-intersection}. This means that if $\widetilde \gamma_1$ is a lift of $\gamma_1$, there exists $T\in\mathcal G$ such that $\widetilde\gamma_1$ and $T\widetilde\gamma_1$ have a $\widetilde{\F}$-transverse intersection at  $\widetilde\gamma_1(t_1)=T\widetilde\gamma_1(t_2)$.

\subsection{Recurrence, equivalence and accumulation} 
A transverse path $\gamma:\R\to \Sigma$ is {\it positively recurrent} if, for every $a<b$, there exist $c<d$, with $b<c$, such that $\gamma|_{[a,b]}$ and $\gamma|_{[c,d]}$  are equivalent. Similarly $\gamma$ is {\it negatively recurrent} if, for every $a<b$, there exist  $c<d$, with $d<a$,  such that $\gamma|_{[a,b]}$ and $\gamma|_{[c,d]}$  are equivalent. Finally $\gamma$ is {\it recurrent} if it is both positively and negatively recurrent.
\medskip 

Two transverse paths $\gamma_1:\R\to \Sigma$ and $\gamma_2:\R\to \Sigma$ are {\it equivalent at $+\infty$} if there exists $a_1$ and $a_2$ in $\R$ such that $\gamma_1{}|_{[a_1,+\infty)}$ and  $\gamma_2{}|_{[a_2,+\infty)}$ are equivalent. Similarly $\gamma_1$ and $\gamma_2$ are {\it equivalent at $-\infty$} if there exists $b_1$ and $b_2$ in $\R$ such that $\gamma_1{}|_{(-\infty,b_1]}$ and  $\gamma_2{}|_{(-\infty,b_2]}$ are equivalent.
\medskip

A transverse path $\gamma_1:\R\to \Sigma$  {\it accumulates positively} on the transverse path $\gamma_2:\R\to \Sigma$ if there exist real numbers $a_1$ and $a_2<b_2$ such that  $\gamma_1{}|_{[a_1,+\infty)}$ and $\gamma_2{}|_{[a_2,b_2)}$ are equivalent. Similarly, $\gamma_1$  {\it accumulates negatively} on $\gamma_2$ if there exist real numbers $b_1$ and $a_2<b_2$ such that  $\gamma_1{}|_{(-\infty,b_1]}$ and $\gamma_2{}|_{(a_2,b_2]}$ are equivalent. 
Finally $\gamma_1$ {\it accumulates} on $\gamma_2$ if it accumulates  positively or negatively on $\gamma_2$.

\subsection{Strips} \label{ss:strips}

We fix $T\in {\mathcal G}\setminus\{0\}$ and consider
\begin{itemize}
\item the annulus $\hat\Sigma=\tilde\Sigma/T$;
\item the covering projections $\pi: \tilde\Sigma\to \hat\Sigma$ and $\hat\pi: \hat\Sigma\to \Sigma$;
\item the foliation $\hat{\F}$ on $\hat{\Sigma}$ induced by $\tilde{\F}$.
\end{itemize}
Suppose that $\hat\Gamma_*$ is a simple loop transverse to $\hat{\F}$. Then, $\hat\Gamma_*$ is essential and $\tilde\gamma_*=\pi^{-1}(\hat\Gamma_*)$ is an oriented line of $\tilde \Sigma$, invariant by $T$ and  transverse to $\tilde{\F}$.
The set   
$$\hat B=\{\hat z \in \hat\Sigma \mid \hat \phi_{\hat z} \cap \hat\Gamma_*\not=\emptyset\}$$
is an open annulus which is $\hat{\F}$-saturated, meaning that it is a union of leaves. Similarly
$$\tilde B=\pi^{-1}(\hat B)=\{\tilde z \in \tilde\Sigma\,\vert\enskip \tilde\phi_{\tilde z} \cap \tilde\gamma_*\not=\emptyset\}$$ is an $\tilde{\F}$-saturated plane invariant by $T$. We will call such a set a {\it strip} or a {\it $T$-strip} if we want to be more precise.
The frontier of $\tilde B$, denoted $\partial \tilde B$, is a union of leaves (possibly empty) and can be written  $\partial\tilde B=\partial\tilde B^R\sqcup\partial\tilde B^L$, where 
$$\partial\tilde B^R=\partial\tilde B\cap R(\tilde\gamma_*)\,,\qquad \partial\tilde B^L=\partial\tilde B\cap L(\tilde\gamma_*).$$
Let us state some facts that can be proven easily (see \cite{LecT2} or \cite{Lel}). Note first that:

\begin{itemize}
\item if there is a leaf $\tilde \phi\subset \partial\tilde B$ that is invariant by  $T$, then the set $ \partial\tilde B^R$ or $ \partial\tilde B^L$ that contains $\tilde\phi$ is reduced to this leaf;
\item if $\tilde\gamma:\R\to\tilde\Sigma$ is transverse to $\tilde{\F}$, then the set  of real numbers $t$ such that $\gamma(t)\in\tilde B$ is an interval (possibly empty).

\end{itemize} 
Suppose now that $\tilde\gamma:\R\to\tilde\Sigma$ is transverse to $\tilde{\F}$ and that 
$$\big\{ t\in\R\,\vert\enskip \gamma(t)\in\tilde B\big\}=(a,b),$$ 
where $-\infty\leq a<b\leq\infty$. Say that
\begin{itemize}
\item $\tilde\gamma$ {\it draws $\tilde B$} if there exist $t<t'$ in $(a,b)$ such that $\tilde\phi_{\tilde\gamma(t')}=T\tilde\phi_{\tilde\gamma(t))}$.
\end{itemize}

If, moreover, we suppose that $-\infty<a<b<+\infty$, say that:
\begin{itemize}
\item $\tilde\gamma$ {\it crosses $\tilde B$ from the right to the left }if $\tilde\gamma(a)\in \partial\tilde B^R$ and $\tilde\gamma(b)\in \partial\tilde B^L$;
\item $\tilde\gamma$ {\it crosses $\tilde B$ from the left to the right} if $\tilde\gamma(a)\in \partial\tilde B^L$ and $\tilde\gamma(b)\in \partial\tilde B^R$;
\item $\tilde\gamma$ {\it visits $\tilde B$ on the right} if $\tilde\gamma(a)\in \partial\tilde B^R$ and $\tilde\gamma(b)\in \partial\tilde B^R$;
\item $\tilde\gamma$ {\it visits $\tilde B$ on the left} if $\tilde\gamma(a)\in \partial\tilde B^L$ and $\tilde\gamma(b)\in \partial\tilde B^L$.
\end{itemize}
We will say that $\tilde\gamma$ {\it crosses} $\tilde B$ if it crosses it from the right to the left or from the left to the right. Similarly, we will say that $\tilde\gamma$ {\it visits} $\tilde B$ if it visits it on the right or on the left. Note that $T(\tilde\gamma)$ satisfies the same properties as $\tilde\gamma$. Note also that if $\tilde\gamma$ visits $\tilde B$ on the right, then $\partial \tilde B^R$ is not reduced to a $T$-invariant leaf. An analogous property holds if $\tilde\gamma$ visits $\tilde B$ on the left. Finally, observe that at least one of the following situations occurs (the two last assertions are not incompatible):
\begin{itemize}
\item $\tilde\gamma$ crosses $\tilde B$;
\item $\tilde\gamma$ visits $\tilde B$;
\item $\tilde\gamma$ is equivalent to $\tilde \gamma_*$ at $+\infty$ or at $-\infty$;
\item $\tilde\gamma$ accumulates on $\tilde\gamma_*$ positively or negatively.
\end{itemize}

Let us conclude this list of properties by the following ones (see \cite[Section 2.1.2.c]{Lel}):

\begin{proposition}\label{prop:transversestrips}
We have the following results:
\begin{itemize}
\item If $\tilde\gamma$ visits and draws $\tilde B$, then 
$\tilde\gamma$ and $T(\tilde\gamma)$ have an $\tilde{\F}$-transverse intersection and so $\gamma=\tilde\pi\circ\tilde\gamma$ has an $\F$-transverse self intersection.

\item If $\tilde\gamma_1$ crosses $\tilde B$ from the right to the left,  if $\tilde\gamma_2$ crosses $\tilde B$ from the right to the left and at least one of the paths $\tilde\gamma_1$ or $\tilde\gamma_2$ draws $\tilde B$, then there exists $k\in \Z$ such that 
$\tilde\gamma_1$ and $T^k(\tilde\gamma_2)$ have an $\tilde{\F}$-transverse intersection, and so $\gamma_1=\tilde\pi\circ\tilde\gamma_1$ and $\gamma_2=\tilde\pi\circ\tilde\gamma_2$ have a  $\F$-transverse intersection.
\end{itemize} 
\end{proposition}

\subsection{More about the accumulation property}

In this final paragraph, we will suppose moreover than $\Sigma$ is connected and that $\Sigma\not=\R^2/\Z^2$. The goal  is to prove the following result that has its own interest and will be used in the sequel to prove Theorem \ref{th:main}. This statement is stronger than some results of \cite[Section 2.1.1]{Lel}.

\begin{proposition} \label{prop: nontransitivity3}
Suppose that $\gamma_1:\R\to \Sigma$ is a positively recurrent transverse path that accumulates positively on a transverse path $\gamma_2:\R\to \Sigma$. 
Then, there exists a transverse simple loop $\Gamma_*\subset\Sigma$ with the following properties.
\begin{enumerate}
\item The set $B$ of leaves met by $\Gamma_*$ is an open annulus of $\Sigma$.
\item The path $\gamma_1$ stays in $B$ and is equivalent to the natural lift of $\Gamma_*$.
\item If $\tilde\gamma_1$, $\tilde\gamma_2$ are lifts of $\gamma_1$, $\gamma_2$ to the universal covering space $\tilde\Sigma$ such that $\tilde\gamma_1\vert_{[a_1,+\infty)}$ is equivalent to $\tilde\gamma_2\vert_{[a_2,b_2)}$ and if $\tilde B$ is the lift of $B$ that contains $\tilde\gamma_1$, then one of the inclusions $\tilde\phi_{\tilde\gamma_2(b_2)}\subset \partial \tilde B^R$, $\phi_{\tilde\gamma_2(b_2)}\subset \partial \tilde B^L$ holds. In the first case, we have $\tilde B\subset L(\tilde \phi)$ for every $\tilde \phi\subset \partial \tilde B^R$ and in the second case, we have $\tilde B\subset R(\tilde \phi)$ for every $\tilde \phi\subset \partial \tilde B^L$. 
\end{enumerate}
\end{proposition}

An example of a situation where Proposition~\ref{prop: nontransitivity3} holds is depicted in Figure~\ref{fig:nontransitivity3}.

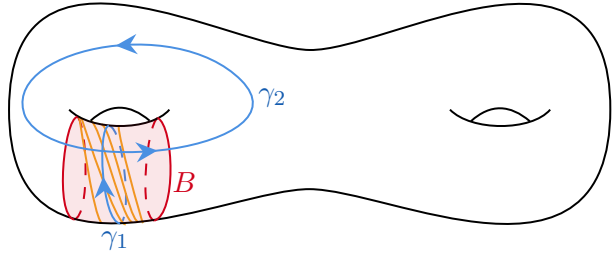
\begin{figure}
\begin{center}

\tikzset{every picture/.style={line width=0.75pt}} 

\begin{tikzpicture}[x=0.75pt,y=0.75pt,yscale=-1,xscale=1]
\draw [color={rgb, 255:red, 208; green, 2; blue, 27 }  ,draw opacity=1 ] [dash pattern={on 4.5pt off 4.5pt}]  (214.14,203.57) .. controls (219,208.43) and (220.71,253.29) .. (212.43,255.29) ;
\draw [color={rgb, 255:red, 208; green, 2; blue, 27 }  ,draw opacity=1 ] [dash pattern={on 4.5pt off 4.5pt}]  (254.14,204.14) .. controls (248.71,207.86) and (244.43,252.71) .. (252.43,255.86) ;
\draw [color={rgb, 255:red, 74; green, 144; blue, 226 }  ,draw opacity=1 ] [dash pattern={on 4.5pt off 4.5pt}]  (229.8,208.1) .. controls (237.2,209.1) and (241.6,254.7) .. (235.4,257.3) ;

\draw [color={rgb, 255:red, 245; green, 166; blue, 35 }  ,draw opacity=1 ]   (220.71,205.57) .. controls (225.4,206.7) and (236.8,257.3) .. (241.6,256.5) ;
\draw [color={rgb, 255:red, 245; green, 166; blue, 35 }  ,draw opacity=1 ]   (216.11,204.17) .. controls (218.6,211.7) and (233.4,258.5) .. (238.2,257.7) ;
\draw [color={rgb, 255:red, 245; green, 166; blue, 35 }  ,draw opacity=1 ]   (226.11,207.37) .. controls (230.8,208.5) and (239.8,255.1) .. (244.2,256.3) ;
\draw [color={rgb, 255:red, 245; green, 166; blue, 35 }  ,draw opacity=1 ]   (234.71,208.77) .. controls (234.6,214.1) and (245.4,253.5) .. (246.8,257.1) ;
\draw [color={rgb, 255:red, 245; green, 166; blue, 35 }  ,draw opacity=1 ]   (215,203.57) .. controls (216.49,212.3) and (222.06,252.3) .. (226.06,256.9) ;

\draw   (180,200) .. controls (180.2,101.4) and (291.4,169.4) .. (330,170) .. controls (368.6,170.6) and (480.2,101) .. (480,200) .. controls (479.8,299) and (369.8,240.2) .. (330,240) .. controls (290.2,239.8) and (179.8,298.6) .. (180,200) -- cycle ;
\draw    (210,200) .. controls (220.6,211) and (250.6,211) .. (260,200) ;
\draw    (400,200) .. controls (410.6,211) and (440.6,211) .. (450,200) ;
\draw    (220.71,205.57) .. controls (230.57,196.57) and (240.29,197.14) .. (250.14,205.57) ;
\draw    (410.43,205.86) .. controls (420.29,196.86) and (430,197.43) .. (439.86,205.86) ;
\draw [color={rgb, 255:red, 208; green, 2; blue, 27 }  ,draw opacity=1 ][line width=0.75]    (214.14,203.57) .. controls (207.29,205) and (202.71,252.71) .. (212.43,255.29) ;
\draw [color={rgb, 255:red, 208; green, 2; blue, 27 }  ,draw opacity=1 ]   (254.14,204.14) .. controls (262.43,208.71) and (263.57,252.43) .. (252.43,255.86) ;

\fill [color={rgb, 255:red, 208; green, 2; blue, 27 }  ,fill opacity=.1][line width=0.75]    
(214.14,203.57) .. controls (207.29,205) and (202.71,252.71) .. 
(212.43,255.29) .. controls +(7,2) and +(-7,2) .. 
(252.43,255.86) .. controls +(263.57-252.43,252.43-255.86) and +(262.43-254.14,208.71-204.14) .. (254.14,204.14) 
.. controls +(-10,5) and +(10,5) ..  (214.14,203.57);

\draw [color={rgb, 255:red, 74; green, 144; blue, 226 }  ,draw opacity=1 ]   (188.14,203.57) .. controls (205.29,236.71) and (329.57,219.57) .. (295.86,185.29) .. controls (262.14,151) and (174.43,170.43) .. (188.14,203.57) -- cycle ;
\draw [shift={(254.01,220.61)}, rotate = 535.36] [fill={rgb, 255:red, 74; green, 144; blue, 226 }  ,fill opacity=1 ][line width=0.08]  [draw opacity=0] (10.72,-5.15) -- (0,0) -- (10.72,5.15) -- (7.12,0) -- cycle    ;
\draw [shift={(233.23,167.82)}, rotate = 355.82] [fill={rgb, 255:red, 74; green, 144; blue, 226 }  ,fill opacity=1 ][line width=0.08]  [draw opacity=0] (10.72,-5.15) -- (0,0) -- (10.72,5.15) -- (7.12,0) -- cycle    ;
\draw [color={rgb, 255:red, 74; green, 144; blue, 226 }  ,draw opacity=1 ]   (229.8,208.1) .. controls (223.8,209.3) and (226.2,256.1) .. (235.4,257.3) ;
\draw [shift={(227,233.78)}, rotate = 86.21] [fill={rgb, 255:red, 74; green, 144; blue, 226 }  ,fill opacity=1 ][line width=0.08]  [draw opacity=0] (10.72,-5.15) -- (0,0) -- (10.72,5.15) -- (7.12,0) -- cycle    ;

\draw (303,187.65) node [anchor=north west][inner sep=0.75pt]  [color={rgb, 255:red, 20; green, 92; blue, 176 }  ,opacity=1 ]  {$\gamma_2$};
\draw (224.5,259.65) node [anchor=north west][inner sep=0.75pt]  [color={rgb, 255:red, 20; green, 92; blue, 176 }  ,opacity=1 ]  {$\gamma_1$};

\draw (260,230) node [anchor=north west][inner sep=0.75pt]  [color={rgb, 255:red, 208; green, 2; blue, 27 }  ,opacity=1 ]  {$B$};

\end{tikzpicture}

\caption{An example where Proposition~\ref{prop: nontransitivity3} holds. \label{fig:nontransitivity3}}
\end{center}
\end{figure}

In Proposition~\ref{prop:conditionhorsehoestrip} we will get additional properties when the paths are supposed to be trajectories that are typical for some ergodic $f$-invariant measures.

\begin{proof}
Let us start with a lemma.

\begin{lemma}\label{LemEquivSubpath}
Let $ \Gamma : \T\to\Sigma$ be a transverse loop, $\tilde\gamma : \R\to\widetilde \dom(I)$ a lift of $\Gamma$ and $\tilde B$ the strip that contains $\tilde\gamma$.  Let $T\in\G$ be the deck transformation associated to $\tilde B$. Suppose that there exists a deck transformation $R\in\G$ and $a\in\R$ such that $\tilde\gamma|_{[a,a+1]}$ is equivalent to a subpath of $R\tilde\gamma$. Then $\tilde\gamma|_{[a,a+1)} \cap R\tilde\gamma \neq \emptyset$. 
\end{lemma}

Note that if moreover $\Gamma$ is a simple path, then the conclusion of the lemma implies that $R\in\langle T\rangle$.

This lemma can be reduced easily to the following fact.

\begin{sublemma}\label{SlemFoliationLoop}
Let $\F$ be a singular foliation on $\Sigma$, and  $ \Gamma : \T\to\Sigma$ a loop of $\Sigma$ that is transverse to $\F$. Then, there exists $z\in\Gamma$ such that $\phi_z^+$ does not meet $\Gamma$ but at the end point.
\end{sublemma}

\begin{proof}[Proof of Lemma~\ref{LemEquivSubpath}]
By Sub-lemma~\ref{SlemFoliationLoop}, there exist $z$, $z'$ in $\Gamma$ (possibly equal) such that $\phi_z^+$ and $\phi_{z'}^-$ do not meet $\Gamma$ but at their end point. Denote $\tilde z$, $\tilde z'$ the respective lifts of $z$, $z'$ that belong to $\tilde\gamma|_{[a,a+1)}$. We know that $\tilde\phi_{\tilde z}^+\cap R\tilde \gamma = \emptyset$, and that $\tilde\phi_{\tilde z'}^-\cap R\tilde \gamma\neq\emptyset$. We deduce that $R\tilde \gamma\cap \tilde\gamma|_{[a,a+1)}\neq \emptyset$. 
\end{proof}

\begin{proof}[Proof of Sub-lemma~\ref{SlemFoliationLoop}]
Fix $z\in\Gamma$. The loop $\Gamma$ being transverse to $\F$, there are finitely many parameters $t\in\T$ such that $z=\Gamma(t)$. Consequently, there exists a compact neighborhood $W_z$ of $z$, a homeomorphism $\Phi_z: W_z\to[-1,1]^2$ and a finite set $I_z$ such that:

\begin{itemize}
\item $\Phi_z$ sends $z$ onto $(0,0)$;
\item $\Phi_z$ sends $\F|_{  W_z}$ onto the vertical foliation oriented upward; 
\item we have $\Phi_z(\Gamma\cap W_z)= \bigcup_{i\in I_z}\mathrm{gr}(\psi_{i,z})$, where $\psi_{i,z}: [-1,1]\to[-1,1]$ is a continuous function satisfying $\psi_{i,z}(0)=0$.\end{itemize}
Here the notation $\mathrm{gr}(\psi)$ denotes the graph of $\psi:[-1,1]\to [-1,1]$ oriented  from the right to the left. See Figure~\ref{FigLocalConfig} for an example of such a configuration.

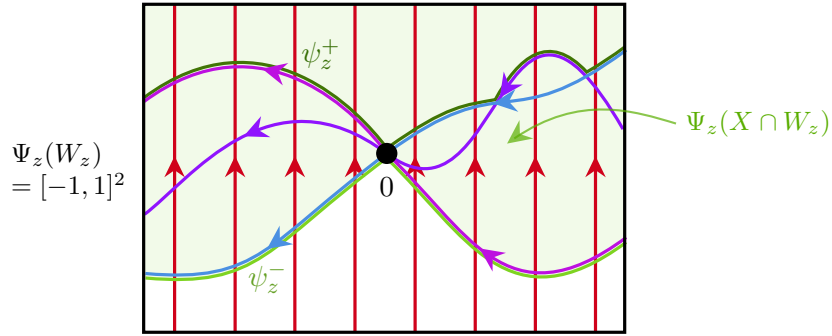
\begin{figure}
\begin{center}

\tikzset{every picture/.style={line width=1.2pt}} 

\begin{tikzpicture}[x=0.75pt,y=0.75pt,yscale=-1.5,xscale=1.5]

\draw [draw opacity=0][fill={rgb, 255:red, 126; green, 211; blue, 33 }  ,fill opacity=0.1 ]   (198.93,60.17) .. controls (199,108.79) and (199.6,125.79) .. (199.46,151.47) .. controls (235,156.59) and (244.2,139.79) .. (279.75,111.22) .. controls (303.2,133.39) and (322,169.79) .. (359.38,139.39) .. controls (359.12,108.26) and (359.6,88.19) .. (358.93,60.17) ;
\draw [color={rgb, 255:red, 208; green, 2; blue, 27 }  ,draw opacity=1 ]   (250,60) -- (250,170) ;
\draw [shift={(250,110.9)}, rotate = 90] [fill={rgb, 255:red, 208; green, 2; blue, 27 }  ,fill opacity=1 ][line width=0.08]  [draw opacity=0] (7.14,-3.43) -- (0,0) -- (7.14,3.43) -- (4.74,0) -- cycle    ;
\draw [color={rgb, 255:red, 208; green, 2; blue, 27 }  ,draw opacity=1 ]   (270,60) -- (270,170) ;
\draw [shift={(270,110.9)}, rotate = 90] [fill={rgb, 255:red, 208; green, 2; blue, 27 }  ,fill opacity=1 ][line width=0.08]  [draw opacity=0] (7.14,-3.43) -- (0,0) -- (7.14,3.43) -- (4.74,0) -- cycle    ;
\draw [color={rgb, 255:red, 208; green, 2; blue, 27 }  ,draw opacity=1 ]   (290,60) -- (290,170) ;
\draw [shift={(290,110.9)}, rotate = 90] [fill={rgb, 255:red, 208; green, 2; blue, 27 }  ,fill opacity=1 ][line width=0.08]  [draw opacity=0] (7.14,-3.43) -- (0,0) -- (7.14,3.43) -- (4.74,0) -- cycle    ;
\draw [color={rgb, 255:red, 208; green, 2; blue, 27 }  ,draw opacity=1 ]   (310,60) -- (310,170) ;
\draw [shift={(310,110.9)}, rotate = 90] [fill={rgb, 255:red, 208; green, 2; blue, 27 }  ,fill opacity=1 ][line width=0.08]  [draw opacity=0] (7.14,-3.43) -- (0,0) -- (7.14,3.43) -- (4.74,0) -- cycle    ;
\draw [color={rgb, 255:red, 208; green, 2; blue, 27 }  ,draw opacity=1 ]   (330,60) -- (330,170) ;
\draw [shift={(330,110.9)}, rotate = 90] [fill={rgb, 255:red, 208; green, 2; blue, 27 }  ,fill opacity=1 ][line width=0.08]  [draw opacity=0] (7.14,-3.43) -- (0,0) -- (7.14,3.43) -- (4.74,0) -- cycle    ;
\draw [color={rgb, 255:red, 208; green, 2; blue, 27 }  ,draw opacity=1 ]   (350,60) -- (350,170) ;
\draw [shift={(350,110.9)}, rotate = 90] [fill={rgb, 255:red, 208; green, 2; blue, 27 }  ,fill opacity=1 ][line width=0.08]  [draw opacity=0] (7.14,-3.43) -- (0,0) -- (7.14,3.43) -- (4.74,0) -- cycle    ;
\draw [color={rgb, 255:red, 208; green, 2; blue, 27 }  ,draw opacity=1 ]   (230.13,60) -- (230.13,170) ;
\draw [shift={(230.13,110.9)}, rotate = 90] [fill={rgb, 255:red, 208; green, 2; blue, 27 }  ,fill opacity=1 ][line width=0.08]  [draw opacity=0] (7.14,-3.43) -- (0,0) -- (7.14,3.43) -- (4.74,0) -- cycle    ;
\draw [color={rgb, 255:red, 208; green, 2; blue, 27 }  ,draw opacity=1 ]   (210,60) -- (210,170) ;
\draw [shift={(210,110.9)}, rotate = 90] [fill={rgb, 255:red, 208; green, 2; blue, 27 }  ,fill opacity=1 ][line width=0.08]  [draw opacity=0] (7.14,-3.43) -- (0,0) -- (7.14,3.43) -- (4.74,0) -- cycle    ;
\draw [color={rgb, 255:red, 126; green, 211; blue, 33 }  ,draw opacity=1 ]   (200.33,151.8) .. controls (238.61,155.23) and (240.62,141.54) .. (280.62,111.54) ;
\draw [color={rgb, 255:red, 126; green, 211; blue, 33 }  ,draw opacity=1 ]   (281.01,112.08) .. controls (306.6,137.08) and (321.1,168.99) .. (360.18,139.83) ;
\draw [color={rgb, 255:red, 65; green, 117; blue, 5 }  ,draw opacity=1 ]   (280.91,108.07) .. controls (300.86,93.7) and (306.91,93.81) .. (316.66,91.98) .. controls (324.46,77.7) and (334.41,68.98) .. (347.08,82.73) .. controls (352.16,79.72) and (355.48,77.51) .. (359.84,74.24) ;
\draw [color={rgb, 255:red, 144; green, 19; blue, 254 }  ,draw opacity=1 ]   (200.01,130.44) .. controls (214.7,119.42) and (241.99,80.77) .. (280.57,110.03) .. controls (319.15,139.3) and (318.3,30.15) .. (359.15,101.87) ;
\draw [shift={(232.96,104.27)}, rotate = 336.35] [fill={rgb, 255:red, 144; green, 19; blue, 254 }  ,fill opacity=1 ][line width=0.08]  [draw opacity=0] (8.04,-3.86) -- (0,0) -- (8.04,3.86) -- (5.34,0) -- cycle    ;
\draw [shift={(318.13,91.36)}, rotate = 304.37] [fill={rgb, 255:red, 144; green, 19; blue, 254 }  ,fill opacity=1 ][line width=0.08]  [draw opacity=0] (8.04,-3.86) -- (0,0) -- (8.04,3.86) -- (5.34,0) -- cycle    ;
\draw [color={rgb, 255:red, 65; green, 117; blue, 5 }  ,draw opacity=1 ]   (199.77,91.36) .. controls (239.77,61.36) and (271.69,96.33) .. (280.39,107.98) ;
\draw [color={rgb, 255:red, 74; green, 144; blue, 226 }  ,draw opacity=1 ]   (280.57,110.03) .. controls (320.57,80.03) and (320.2,105.35) .. (360.2,75.35) ;
\draw [shift={(315.89,93.26)}, rotate = 352.84] [fill={rgb, 255:red, 74; green, 144; blue, 226 }  ,fill opacity=1 ][line width=0.08]  [draw opacity=0] (8.04,-3.86) -- (0,0) -- (8.04,3.86) -- (5.34,0) -- cycle    ;
\draw [color={rgb, 255:red, 74; green, 144; blue, 226 }  ,draw opacity=1 ]   (200.29,150.29) .. controls (238.57,153.71) and (240.57,140.03) .. (280.57,110.03) ;
\draw [shift={(240.71,142.08)}, rotate = 323.91] [fill={rgb, 255:red, 74; green, 144; blue, 226 }  ,fill opacity=1 ][line width=0.08]  [draw opacity=0] (8.04,-3.86) -- (0,0) -- (8.04,3.86) -- (5.34,0) -- cycle    ;
\draw [color={rgb, 255:red, 189; green, 16; blue, 224 }  ,draw opacity=1 ]   (280.57,110.03) .. controls (306.77,134.78) and (320.2,168.2) .. (360.2,138.2) ;
\draw [shift={(311.43,141.87)}, rotate = 35.27] [fill={rgb, 255:red, 189; green, 16; blue, 224 }  ,fill opacity=1 ][line width=0.08]  [draw opacity=0] (8.04,-3.86) -- (0,0) -- (8.04,3.86) -- (5.34,0) -- cycle    ;
\draw [color={rgb, 255:red, 189; green, 16; blue, 224 }  ,draw opacity=1 ]   (199.63,92.78) .. controls (239.63,62.78) and (271.91,97.92) .. (280.57,110.03) ;
\draw [shift={(239.09,81.59)}, rotate = 13.44] [fill={rgb, 255:red, 189; green, 16; blue, 224 }  ,fill opacity=1 ][line width=0.08]  [draw opacity=0] (8.04,-3.86) -- (0,0) -- (8.04,3.86) -- (5.34,0) -- cycle    ;
\draw   (199.71,60) -- (359.71,60) -- (359.71,170) -- (199.71,170) -- cycle ;
\draw  [draw opacity=0][fill={rgb, 255:red, 0; green, 0; blue, 0 }  ,fill opacity=1 ] (277.03,110.03) .. controls (277.03,108.08) and (278.62,106.49) .. (280.57,106.49) .. controls (282.53,106.49) and (284.11,108.08) .. (284.11,110.03) .. controls (284.11,111.99) and (282.53,113.57) .. (280.57,113.57) .. controls (278.62,113.57) and (277.03,111.99) .. (277.03,110.03) -- cycle ;
\draw [color={rgb, 255:red, 90; green, 174; blue, 0 }  ,draw opacity=0.7 ][line width=0.8]   (376.78,100) .. controls (362.89,95.82) and (344.89,89.94) .. (322.98,105.01) ;
\draw [shift={(320.58,106.73)}, rotate = 323.21] [fill={rgb, 255:red, 90; green, 174; blue, 0 }  ,fill opacity=0.7 ][line width=0.08]  [draw opacity=0] (8.04,-3.86) -- (0,0) -- (8.04,3.86) -- (5.34,0) -- cycle    ;

\draw (280.57,116.97) node [anchor=north] [inner sep=0.75pt]    {$0$};
\draw (241,145.33) node [anchor=north] [inner sep=0.75pt]  [color={rgb, 255:red, 101; green, 160; blue, 37 }  ,opacity=1 ]  {$\psi _{z}^{-}$};
\draw (251,83) node [anchor=south west] [inner sep=0.75pt]  [color={rgb, 255:red, 55; green, 110; blue, 5 }  ,opacity=1 ]  {$\psi _{z}^{+}$};
\draw (199.37,116.25) node [anchor=east] [inner sep=0.75pt]  [font=\small]  {$ \begin{array}{l}
\Psi _{z}( W_{z})\\
=[ -1,1]^{2}
\end{array}$};
\draw (378.78,100) node [anchor=west] [inner sep=0.75pt]  [font=\small,color={rgb, 255:red, 90; green, 174; blue, 0 }  ,opacity=1 ]  {$\Psi _{z}( X\cap W_{z})$};

\end{tikzpicture}

\caption{\label{FigLocalConfig} Local configuration of the path $\Gamma$ and the foliation $\F$ (in red) around the point $0=\Psi_z(z)$.}
\end{center}
\end{figure}

Consider the two continuous functions 
$$\psi_z^-=\min_{i\in I_z} \psi_{i,z}\,,\quad \psi_z^+=\max_{i\in I_z} \psi_{i,z}$$ and define 
$$\gamma_z^{-}=\Phi_z^{-1}(\mathrm{gr}(\psi_z^-))\,,\quad\gamma_z^{+}=\Phi_z^{-1}(\mathrm{gr}(\psi_z^+)).$$
We will argue by contradiction by supposing that for any $z\in\Gamma$, the path $\phi_z^+$ meets $\Gamma$ in a point that is not the end point. In that case, for every $z\in\Gamma$, there exists a sub-path $\delta_z:[0,1]\to\Sigma$ of $\phi_z^+$ such that 
$$\delta_z(0)=z,\quad
\delta_z(1)\in \Gamma,\quad
\delta_z((0,1))\cap\Gamma=\emptyset.$$
In particular we can define a first return map $\theta:\Gamma\to\Gamma$ by setting $\theta(z)= \delta_z(1)$. We will prove that $X=\bigcup_{z\in\Gamma} \delta_z([0,1])$ is a compact sub-surface with boundary.
Note that for every $z\in\Gamma$, the function $\theta$ induces a homeomorphism from a compact neighborhood $\alpha_z$ of $z$ in $\gamma_z^{+}$ to a compact neighborhood $\omega_{z}$ of $\theta(z)$ in $\gamma_{\theta(z)}^{-}$ and consequently that every point $\delta_z(t)$, $ t\in(0,1)$, belongs to the interior of $X$. Note also that for every $z\in\Gamma$, the set $\Phi_z^{-}(\{ (x,y)\,\vert\enskip y\geq \psi_z^{-}(x)\})$ is included in $X$. 
By compactness, one can cover $\Gamma$ with finitely many $\alpha_z$, $z\in\Gamma$. 
We deduce that the image of $\theta$, denoted $\mathrm{im}(\theta)$,  is the union of finitely many compact subsets (the corresponding $\omega_{z}$) and therefore is compact. We deduce also that $X$ is compact because for every $z\in\Gamma$, the set $\bigcup_{z'\in\alpha_z} \delta_{z'}([0,1])$ is compact. Now, observe that for every $z\in\Gamma$ and every $z'\in \gamma_z^{-}$, the sets $ \gamma_{z'}^{-}$ and  $\gamma_z^{-}$ coincide in a neighborhood of $z'$.  It implies that $\mathrm{im}(\theta)\cap\gamma_z^{-}$  is an open subset of  $\gamma_z^{-}$.
By connectedness of $\gamma_z^{-}$, either $\gamma_z^-$ is contained in $\mathrm{im}(\theta)$ or it is disjoint from  $\mathrm{im}(\theta)$. In the first case, $W_z$ is contained in $X$, in the second case $W_z\cap X=\Phi_z^{-1}(\{ (x,y)\,\vert\enskip y\geq \psi_z^{-}(x)\})$: we have proved that $X$ is a compact sub-surface of $\Sigma$ (possibly with boundary). Note that for every $z\in\partial X$ it holds that $\phi_z^+\setminus\{z\}\subset \mathrm{int}(X)$ (in other terms the foliation is pointing inward on the boundary).
 
By hypothesis, $\Sigma$ is connected and different from $\R^2/\Z^2$. So, it does not bear a non-singular foliation. We deduce that $X$ is a surface with boundary. More precisely it is homeomorphic to the closed annulus because it bears a non singular foliation. Let $\Psi: X\to\S^2$ be a topological embedding compatible with the usual orientations. The loop $\Psi(\Gamma)$ is homologous to $0$ in $\S^2$ and one can define a dual function  $\delta:\S^2\setminus\Psi(\Gamma)\to\Z$. Such a function is defined by the following property: for every $z$, $z'$ in $\S^2\setminus\Psi(\Gamma)$ and every path $\beta$ joining $z$ to $z'$, the algebraic intersection number $\Psi(\Gamma)\wedge \beta$ is equal to $\delta(z')-\delta(z)$. Let $U$ be a connected component of $\S^2\setminus\Psi(\Gamma)$ where $\delta$ reaches its maximum. The set $\Psi(\Gamma)$ being connected, the closure of $U$ is a topological disk. Moreover the fact that $\delta$ reaches its maximum in $U$ implies that for every $z\in \partial U$ it holds that $\phi_z^+\setminus\{z\}\subset U$. So $U$ is not a connected component of  $\S^2\setminus \Psi(X)$ and it holds that $\overline U\subset \psi(X)$. Summarizing, we have found a closed topological disk bearing a non-singular foliation pointing inward on the boundary. We have got a contradiction. 
\end{proof}

\medskip

Let us explain how to construct the simple loop $\Gamma_*$ that appears in Proposition~\ref{prop: nontransitivity3}. As $\gamma_1$ is positively recurrent, there exist two numbers $c_1<c'_1$, with $c_1>a_1$, such that $\phi_{\gamma_1(c_1)} = \phi_{\gamma_1(c'_1)}$ (see Figure~\ref{Fignontransitivity3} for these different points). It implies that $\gamma_1\vert_{[c_1,c'_1]}$ is equivalent to a transverse path $\gamma_* :[c_1,c'_1]\to\Sigma $ such that $\gamma_*(c_1) = \gamma_*(c'_1)$. The set
$$X=\big\{(t,t')\in [c_1,c'_1]^2\,\vert \enskip t<t'\enskip\mathrm{and}\enskip  \gamma_*(t) = \gamma_*(t')\big\}$$ 
is non empty (because it contains $(c_1,c'_1)$) and compact. Indeed, it is closed in $\{(t,t')\in [c_1,c'_1]^2\,\vert \enskip t<t'\}$, an its closure in the compact set $\{(t,t')\in [c_1,c'_1]^2\,\vert \enskip t\leq t'\}$ does not contain any couple $(t,t)$. The function $(t,t')\mapsto t'-t$ being continuous and positive on $X$, reaches its minimum at a couple $(c''_1,c'''_1)$. So, replacing $(c_1,c'_1)$ with $(c''_1, c'''_1)$ if necessary, one can always suppose that the loop $\Gamma_*$ naturally defined by $\gamma_*$ is simple. We denote $B$ the union of leaves met by $\Gamma_*$. 

By hypothesis there exist two lifts $\tilde\gamma_1$ and $\tilde\gamma_2$ of respectively $\gamma_1$ and $\gamma_2$ to $\tilde\Sigma$ such that $\tilde\gamma_1|_{[a_1, +\infty)}$ and $\tilde\gamma_2|_{[a_2, b_2)}$ are equivalent. We denote $\tilde B$ the strip that lifts $B$ and contains $\tilde \gamma_1\vert_{[c_1,c'_1]}$. We denote $\tilde \gamma_*$ a lift of $\Gamma_*$ that lies inside $\tilde B$ and  $T\in\mathcal G$ the primitive deck transformation associated to $\tilde B$ (chosen accordingly to the orientation of $\tilde\gamma_*$).

\begin{lemma}\label{LemInclus}
The path $\tilde\gamma_1|_{[c_1,+\infty)}$ is included in $\tilde B$.
\end{lemma}

\begin{proof}
We will argue by contradiction and suppose it is not. Then there exists $d_1>c'_1$, uniquely defined, such that $\tilde\gamma_1(d_1)\notin\tilde B$ and  $\tilde\gamma_1\vert_{[c_1,d_1)}\subset\tilde B$.

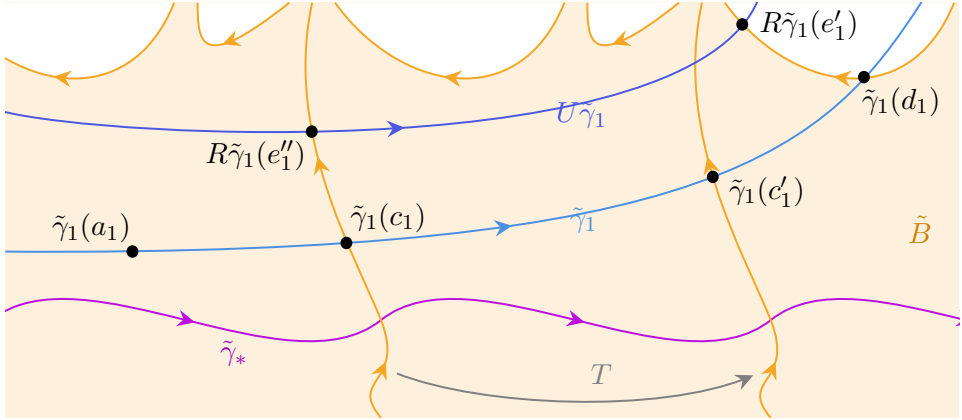
\begin{figure}
\begin{center}

\tikzset{every picture/.style={line width=0.75pt}} 

\begin{tikzpicture}[x=0.75pt,y=0.75pt,yscale=-1.1,xscale=1.1]
\clip  (36.6,30.17) rectangle (469,220) ;

\draw  [draw opacity=0][fill={rgb, 255:red, 245; green, 166; blue, 35 }  ,fill opacity=0.15 ] (36.6,30.17) -- (469,30.17) -- (469,220) -- (36.6,220) -- cycle ;
\draw [color={rgb, 255:red, 245; green, 166; blue, 35 }  ,draw opacity=1 ][fill={rgb, 255:red, 255; green, 255; blue, 255 }  ,fill opacity=1 ]   (185.98,30.17) .. controls (214.06,69.7) and (268.96,82.65) .. (289.97,30.17) ;
\draw [shift={(235.25,63.96)}, rotate = 7.23] [fill={rgb, 255:red, 245; green, 166; blue, 35 }  ,fill opacity=1 ][line width=0.08]  [draw opacity=0] (8.04,-3.86) -- (0,0) -- (8.04,3.86) -- (5.34,0) -- cycle    ;
\draw [color={rgb, 255:red, 245; green, 166; blue, 35 }  ,draw opacity=1 ][fill={rgb, 255:red, 255; green, 255; blue, 255 }  ,fill opacity=1 ]   (300.37,30.17) .. controls (300.57,47.36) and (300.37,63.67) .. (341.96,30.17) ;
\draw [shift={(311.58,49.2)}, rotate = 339.28] [fill={rgb, 255:red, 245; green, 166; blue, 35 }  ,fill opacity=1 ][line width=0.08]  [draw opacity=0] (8.04,-3.86) -- (0,0) -- (8.04,3.86) -- (5.34,0) -- cycle    ;
\draw [color={rgb, 255:red, 245; green, 166; blue, 35 }  ,draw opacity=1 ][fill={rgb, 255:red, 255; green, 255; blue, 255 }  ,fill opacity=1 ]   (362.76,30.17) .. controls (390.83,69.7) and (445.74,82.65) .. (466.75,30.17) ;
\draw [shift={(412.02,63.96)}, rotate = 7.23] [fill={rgb, 255:red, 245; green, 166; blue, 35 }  ,fill opacity=1 ][line width=0.08]  [draw opacity=0] (8.04,-3.86) -- (0,0) -- (8.04,3.86) -- (5.34,0) -- cycle    ;
\draw [color={rgb, 255:red, 245; green, 166; blue, 35 }  ,draw opacity=1 ][fill={rgb, 255:red, 255; green, 255; blue, 255 }  ,fill opacity=1 ]   (123.59,30.17) .. controls (123.8,47.36) and (123.59,63.67) .. (165.18,30.17) ;
\draw [shift={(134.8,49.2)}, rotate = 339.28] [fill={rgb, 255:red, 245; green, 166; blue, 35 }  ,fill opacity=1 ][line width=0.08]  [draw opacity=0] (8.04,-3.86) -- (0,0) -- (8.04,3.86) -- (5.34,0) -- cycle    ;
\draw [color={rgb, 255:red, 189; green, 16; blue, 224 }  ,draw opacity=1 ]   (206.78,175.33) .. controls (248.37,141.83) and (341.96,208.83) .. (383.56,175.33) ;
\draw [shift={(299.22,176.34)}, rotate = 193.9] [fill={rgb, 255:red, 189; green, 16; blue, 224 }  ,fill opacity=1 ][line width=0.08]  [draw opacity=0] (8.04,-3.86) -- (0,0) -- (8.04,3.86) -- (5.34,0) -- cycle    ;
\draw [color={rgb, 255:red, 189; green, 16; blue, 224 }  ,draw opacity=1 ]   (383.56,175.33) .. controls (425.15,141.83) and (518.74,208.83) .. (560.33,175.33) ;
\draw [shift={(476,176.34)}, rotate = 193.9] [fill={rgb, 255:red, 189; green, 16; blue, 224 }  ,fill opacity=1 ][line width=0.08]  [draw opacity=0] (8.04,-3.86) -- (0,0) -- (8.04,3.86) -- (5.34,0) -- cycle    ;
\draw [color={rgb, 255:red, 189; green, 16; blue, 224 }  ,draw opacity=1 ]   (30,175.33) .. controls (71.59,141.83) and (165.18,208.83) .. (206.78,175.33) ;
\draw [shift={(122.44,176.34)}, rotate = 193.9] [fill={rgb, 255:red, 189; green, 16; blue, 224 }  ,fill opacity=1 ][line width=0.08]  [draw opacity=0] (8.04,-3.86) -- (0,0) -- (8.04,3.86) -- (5.34,0) -- cycle    ;
\draw [color={rgb, 255:red, 245; green, 166; blue, 35 }  ,draw opacity=1 ]   (175.58,30.17) .. controls (162.48,90.69) and (195.34,146.08) .. (206.78,175.33) .. controls (218.22,204.59) and (191.6,201.91) .. (206.78,220) ;
\draw [shift={(178,100.57)}, rotate = 75.04] [fill={rgb, 255:red, 245; green, 166; blue, 35 }  ,fill opacity=1 ][line width=0.08]  [draw opacity=0] (8.04,-3.86) -- (0,0) -- (8.04,3.86) -- (5.34,0) -- cycle    ;
\draw [shift={(209.01,194.24)}, rotate = 113.35] [fill={rgb, 255:red, 245; green, 166; blue, 35 }  ,fill opacity=1 ][line width=0.08]  [draw opacity=0] (8.04,-3.86) -- (0,0) -- (8.04,3.86) -- (5.34,0) -- cycle    ;
\draw [color={rgb, 255:red, 245; green, 166; blue, 35 }  ,draw opacity=1 ]   (352.36,30.17) .. controls (339.26,90.69) and (372.12,146.08) .. (383.56,175.33) .. controls (394.99,204.59) and (368.37,201.91) .. (383.56,220) ;
\draw [shift={(354.78,100.57)}, rotate = 75.04] [fill={rgb, 255:red, 245; green, 166; blue, 35 }  ,fill opacity=1 ][line width=0.08]  [draw opacity=0] (8.04,-3.86) -- (0,0) -- (8.04,3.86) -- (5.34,0) -- cycle    ;
\draw [shift={(385.79,194.24)}, rotate = 113.35] [fill={rgb, 255:red, 245; green, 166; blue, 35 }  ,fill opacity=1 ][line width=0.08]  [draw opacity=0] (8.04,-3.86) -- (0,0) -- (8.04,3.86) -- (5.34,0) -- cycle    ;
\draw [color={rgb, 255:red, 74; green, 144; blue, 226 }  ,draw opacity=1 ]   (36.4,144) .. controls (360,146.4) and (409.44,96.33) .. (452,29.83) ;
\draw [shift={(266.27,132.18)}, rotate = 171.76] [fill={rgb, 255:red, 74; green, 144; blue, 226 }  ,fill opacity=1 ][line width=0.08]  [draw opacity=0] (8.04,-3.86) -- (0,0) -- (8.04,3.86) -- (5.34,0) -- cycle    ;
\draw  [fill={rgb, 255:red, 0; green, 0; blue, 0 }  ,fill opacity=1 ] (188.85,140.05) .. controls (188.85,138.72) and (189.85,137.65) .. (191.09,137.65) .. controls (192.32,137.65) and (193.33,138.72) .. (193.33,140.05) .. controls (193.33,141.37) and (192.32,142.45) .. (191.09,142.45) .. controls (189.85,142.45) and (188.85,141.37) .. (188.85,140.05) -- cycle ;
\draw  [fill={rgb, 255:red, 0; green, 0; blue, 0 }  ,fill opacity=1 ] (355.05,109.83) .. controls (355.05,108.5) and (356.05,107.43) .. (357.28,107.43) .. controls (358.52,107.43) and (359.52,108.5) .. (359.52,109.83) .. controls (359.52,111.16) and (358.52,112.23) .. (357.28,112.23) .. controls (356.05,112.23) and (355.05,111.16) .. (355.05,109.83) -- cycle ;
\draw  [fill={rgb, 255:red, 0; green, 0; blue, 0 }  ,fill opacity=1 ] (423.6,64.2) .. controls (423.6,62.88) and (424.6,61.8) .. (425.83,61.8) .. controls (427.07,61.8) and (428.07,62.88) .. (428.07,64.2) .. controls (428.07,65.53) and (427.07,66.61) .. (425.83,66.61) .. controls (424.6,66.61) and (423.6,65.53) .. (423.6,64.2) -- cycle ;
\draw [color={rgb, 255:red, 128; green, 128; blue, 128 }  ,draw opacity=1 ]   (214,199.75) .. controls (252.55,214.94) and (332.71,218.29) .. (373.24,201.63) ;
\draw [shift={(375.67,200.58)}, rotate = 155.79] [fill={rgb, 255:red, 128; green, 128; blue, 128 }  ,fill opacity=1 ][line width=0.08]  [draw opacity=0] (8.04,-3.86) -- (0,0) -- (8.04,3.86) -- (5.34,0) -- cycle    ;
\draw  [fill={rgb, 255:red, 0; green, 0; blue, 0 }  ,fill opacity=1 ] (91.82,143.96) .. controls (91.82,142.63) and (92.82,141.56) .. (94.06,141.56) .. controls (95.29,141.56) and (96.29,142.63) .. (96.29,143.96) .. controls (96.29,145.29) and (95.29,146.36) .. (94.06,146.36) .. controls (92.82,146.36) and (91.82,145.29) .. (91.82,143.96) -- cycle ;
\draw [color={rgb, 255:red, 245; green, 166; blue, 35 }  ,draw opacity=1 ][fill={rgb, 255:red, 255; green, 255; blue, 255 }  ,fill opacity=1 ]   (9.01,30.17) .. controls (37.09,69.7) and (91.99,82.65) .. (113,30.17) ;
\draw [shift={(58.28,63.96)}, rotate = 7.23] [fill={rgb, 255:red, 245; green, 166; blue, 35 }  ,fill opacity=1 ][line width=0.08]  [draw opacity=0] (8.04,-3.86) -- (0,0) -- (8.04,3.86) -- (5.34,0) -- cycle    ;
\draw [color={rgb, 255:red, 74; green, 88; blue, 226 }  ,draw opacity=1 ]   (34.5,79.75) .. controls (83.93,92.61) and (346.66,104.9) .. (377.8,28.9) ;
\draw [shift={(217.55,87.34)}, rotate = 176.19] [fill={rgb, 255:red, 74; green, 88; blue, 226 }  ,fill opacity=1 ][line width=0.08]  [draw opacity=0] (8.04,-3.86) -- (0,0) -- (8.04,3.86) -- (5.34,0) -- cycle    ;
\draw  [fill={rgb, 255:red, 0; green, 0; blue, 0 }  ,fill opacity=1 ] (173.15,89.13) .. controls (173.15,87.81) and (174.15,86.73) .. (175.39,86.73) .. controls (176.63,86.73) and (177.63,87.81) .. (177.63,89.13) .. controls (177.63,90.46) and (176.63,91.53) .. (175.39,91.53) .. controls (174.15,91.53) and (173.15,90.46) .. (173.15,89.13) -- cycle ;
\draw  [fill={rgb, 255:red, 0; green, 0; blue, 0 }  ,fill opacity=1 ] (368.41,40.11) .. controls (368.41,38.78) and (369.41,37.71) .. (370.65,37.71) .. controls (371.88,37.71) and (372.89,38.78) .. (372.89,40.11) .. controls (372.89,41.43) and (371.88,42.51) .. (370.65,42.51) .. controls (369.41,42.51) and (368.41,41.43) .. (368.41,40.11) -- cycle ;

\draw (191.63,136.41) node [anchor=south west] [inner sep=0.75pt]    {$\tilde{\gamma }_{1}( c_{1})$};
\draw (363.32,116.61) node [anchor=west] [inner sep=0.75pt]    {$\tilde{\gamma }_{1}( c'_{1})$};
\draw (423.62,67.34) node [anchor=north west][inner sep=0.75pt]    {$\tilde{\gamma }_{1}( d_{1})$};
\draw (306.72,193.9) node [anchor=north] [inner sep=0.75pt]  [color={rgb, 255:red, 128; green, 128; blue, 128 }  ,opacity=1 ]  {$T$};
\draw (140.08,183.53) node [anchor=north] [inner sep=0.75pt]  [color={rgb, 255:red, 159; green, 1; blue, 191 }  ,opacity=1 ]  {$\tilde{\gamma }_{*}$};
\draw (94.68,141.11) node [anchor=south east] [inner sep=0.75pt]    {$\tilde{\gamma }_{1}( a_{1})$};
\draw (291.05,122.53) node [anchor=north west][inner sep=0.75pt]  [color={rgb, 255:red, 74; green, 144; blue, 226 }  ,opacity=1 ]  {$\tilde{\gamma }_{1}$};
\draw (284.83,74.46) node [anchor=north west][inner sep=0.75pt]  [color={rgb, 255:red, 74; green, 88; blue, 226 }  ,opacity=1 ]  {$U\tilde{\gamma }_{1}$};
\draw (174.13,91.8) node [anchor=north east] [inner sep=0.75pt]    {$R\tilde{\gamma }_{1}( e''_{1})$};
\draw (377.05,40.29) node [anchor=west] [inner sep=0.75pt]    {$R\tilde{\gamma }_{1}( e'_{1})$};
\draw (444.5,126.4) node [anchor=north west][inner sep=0.75pt]  [color={rgb, 255:red, 205; green, 129; blue, 3 }  ,opacity=1 ]  {$\tilde{B}$};

\end{tikzpicture}

\caption{The different objects appearing in the proof of Proposition~\ref{prop: nontransitivity3}, Lemma~\ref{LemInclus} and Claim~\ref{ClaimReturn}. The leaves are in orange.\label{Fignontransitivity3}}
\end{center}
\end{figure}

\begin{claim}\label{ClaimReturn}
There exists a deck transformation $R\in\G$ and real numbers $e_1<e'_1$, with $e_1\geq a_1$, such that either $R\tilde\gamma_1\vert_{[e_1,e'_1]}$ draws and crosses $\tilde B$, or it draws and visits $\tilde B$.
\end{claim}

\begin{proof}
Note that to prove this claim one has to find $R\in\G$ and $e_1<e'_1$ such that $R\tilde\gamma_1|_{[e_1,e_1']}$ draws $\tilde B$ and both $R\tilde\gamma_1(e_1)$ and $R\tilde\gamma_1(e'_1)$ do not belong to $\tilde B$.

As $\gamma_1$ is positively recurrent, there exist real numbers $e''_1<e'_1$, with $e''_1>d_1$, and a deck transformation $R\in\G$ such that $R\tilde\gamma_1|_{[e''_1,e'_1]}$ is equivalent to $\tilde\gamma_1|_{[c_1,d_1]}$; in particular:
\begin{itemize}
\item $\tilde\gamma_1|_{[c_1,c'_1]}$ is equivalent to a subpath of $R\tilde\gamma_1|_{[e''_1,e'_1]}$;
\item $R\tilde\gamma_1([e''_1,e'_1))\subset \tilde B$ and $R\tilde\gamma_1(e'_1)\notin \tilde B$.
\end{itemize}
To prove the claim, it is sufficient to show that $R\tilde\gamma_1([a_1,e'_1))\not\subset \tilde B$, because in that case there exists $e_1\in [a_1,e''_1]$ such that $R\tilde\gamma_1((e_1,e'_1))\subset \tilde B$ and $R\tilde\gamma_1(e_1)\notin \tilde B$. 

We argue by contradiction. Suppose that $R\tilde\gamma_1([a_1,e'_1))$ is contained in $\tilde B$. Then $\tilde\gamma_1([a_1,e'_1))$ is contained in $R^{-1}(\tilde B)$. 
Recall that there exists $t$ such that $\tilde \gamma_*\vert_{[t,t+1]}$ is equivalent to $\tilde\gamma_1|_{[c_1, c'_1]}$ which is a subpath of $\tilde\gamma_1|_{[a_1,e'_1)}$. It implies that $\tilde \gamma_*\vert_{[t,t+1]}$ is equivalent to a subpath of $R^{-1}\tilde\gamma_*$ because $\tilde\gamma_1([a_1,e'_1))$ is contained in $R^{-1}(\tilde B)$. Lemma~\ref{LemEquivSubpath} applies and ensures that $R^{-1}\in\langle T\rangle$. As $\tilde B$ is invariant by $T$, the condition $R\tilde\gamma_1([a_1,e'_1))\subset \tilde B$ gives $\tilde\gamma_1([a_1,e'_1))\subset \tilde B$. This contradicts the condition $\tilde\gamma_1(d_1)\notin \tilde B$, because $a_1<d_1<e'_1$.
\end{proof}

As $\gamma_1$ is positively recurrent, there exist sequences $(e_{1,n})_{n\ge 0}$ and $(e'_{1,n})_{n\ge 0}$ with $a_1<e_{1,n}<e'_{1,n}<e_{1,n+1}$, and a sequence $(R_n)_{n\ge 0}$ of deck transformations, such that $R_n \tilde\gamma_1|_{[e_{1,n},e'_{1,n}]}$ is equivalent to $R\tilde \gamma_1|_{[e_1,e'_1]}$. As $\tilde\gamma_1$ accumulates on $\tilde\gamma_2$, a similar statement holds for $\tilde\gamma_2$: there exist sequences $(e_{2,n})_{n\ge 0}$ and $(e'_{2,n})_{n\ge 0}$ with $a_2<e_{2,n}<e'_{2,n}<e_{2,n+1}<b_2$ such that $R_n \tilde\gamma_2|_{[e_{2,n},e'_{2,n}]}$, is equivalent to $R\tilde \gamma_1|_{[e_1,e'_1]}$. Note that the $R_n$ are all different because every leaf of $\tilde{\F}$ intersects $\tilde \gamma_2([a_2,b_2])$ at most once.

We have two possibilities given by Claim~\ref{ClaimReturn}: either  $R\tilde\gamma_1|_{[e_1,e_1']}$ draws and crosses $\tilde B$, or it draws and visits $\tilde B$.  

Suppose that $R\tilde\gamma_1|_{[e_1,e_1']}$ draws and crosses $\tilde B$. 
In this case, for any $n\in\N$, the path $R_n \tilde\gamma_2|_{[e_{2,n},e'_{2,n}]}$ intersects $\tilde\gamma_*$. Replacing $R_n$ with $T^{k_N}\circ R_n$ for a certain $k_N\in\Z$ if necessary, one can suppose that $R_n \tilde\gamma_2|_{[e_{2,n},e'_{2,n}]}$ intersects $\tilde\gamma_*|_{[t,t+1]}$ and so $R_n^{-1}(\tilde\gamma_*|_{[t,t+1]})$ intersects $\tilde \gamma_2([a_2,b_2])$. It contradicts the fact that the action of $\G$ on compact subsets is proper.

Suppose now that $R\tilde\gamma_1|_{[e_1,e_1']}$ draws and visits $\tilde B$. Then $R\tilde\gamma_1|_{[e_1,e_1']}$ and $TR\tilde\gamma_1|_{[e_1,e_1']}$ have an $\tilde{\F}$-transverse intersection. One deduces that for any $n\in\N$, one has $R_n \tilde\gamma_2|_{[e_{2,n},e'_{2,n}]}$ and $TR\tilde\gamma_1|_{[e_1,e_1']}$ have an $\tilde{\F}$-transverse intersection because $R_n \tilde\gamma_2|_{[e_{2,n},e'_{2,n}]}$ and  $R\tilde\gamma_1|_{[e_1,e_1']}$  are equivalent. Consequently, it holds that $R_n \tilde\gamma_2|_{[e_{2,n},e'_{2,n}]}\cap TR\tilde\gamma_1|_{[e_1,e_1']}\not =\emptyset$ and so that $R_n \tilde\gamma_2|_{[a_2,b_2]}\cap TR\tilde\gamma_1|_{[e_1,e_1']}\not =\emptyset$. It contradicts once again the fact that the action of $\G$ on compact subsets is proper.
This finishes the proof of Lemma~\ref{LemInclus}.
\end{proof}

By Lemma~\ref{LemInclus}, we know that $\tilde\gamma_1|_{[c_1,+\infty)}$ stays in $\tilde B$. We first prove that $\tilde\gamma_1$ cannot accumulate in $\tilde\gamma_*$.

Indeed, otherwise, as $\gamma_1$ is positively recurrent, there exist deck transformations $(R_n)_{n\ge 0}\in \G$ and parameters $d_n<d'_n$ both going to $+\infty$ such that $\tilde\gamma_1|_{[d_n,d'_n]}$ is equivalent to $R_n\tilde\gamma_1|_{[c_1,c'_1]}$, which is itself equivalent to $R_n\tilde\gamma_*|_{[t,t+1]}$. The fact that $\tilde\gamma_1$ accumulates in $\tilde\gamma_*$ implies that $R_n\notin\langle T \rangle$ eventually. Recall that for any $n$, the path $\tilde\gamma_1|_{[d_n,d'_n]}$ is equivalent to a subpath of $\tilde\gamma_*$; this allows to apply Lemma~\ref{LemEquivSubpath} to the simple path $\Gamma_*$, which implies that $R_n\in\langle T \rangle$, a contradiction.

Hence, there exists $t_1\in\R$ such that $\tilde\gamma_1|_{[c_1,+\infty)}$ is equivalent to $\tilde\gamma_*\vert_{[t_1,+\infty)}$.
Moreover it is equivalent to $\tilde\gamma_2\vert_{[c_2,b_2)}$, where $c_2\in[a_2,b_2]$. It implies that $\tilde\phi_{\tilde\gamma_2(b_2)}\subset\partial\tilde B$. We do not lose generality by supposing that $\tilde\phi_{\tilde\gamma_2(b_2)}\subset\partial\tilde B^L$. We choose $a'_2\in[c_2,b_2)$ such that $\tilde\gamma_2([a'_2,b_2])\in L(\tilde\gamma_*)$.

\begin{lemma}\label{lem:NotLeft}
For every leaf $\tilde\phi\subset\partial\tilde B^L$ it holds that $\tilde B\subset R(\tilde \phi)$.
\end{lemma}

\begin{proof}
See Figure~\ref{Fig:NotLeft} for an example of configuration of the proof.
Suppose that there exists a leaf $\tilde\phi_0\subset\partial\tilde B^L$ such that $\tilde B\subset L(\tilde \phi)$. One can find a transverse path $\tilde \gamma_3:[a_3,b_3]\to\tilde\Sigma$ such that $\tilde\gamma_3(a_3)\in\tilde\phi_0$ and $\tilde\gamma_3((a_3,b_3])\subset \tilde B$. Such a path enters in $\tilde B$ by the left. By taking a smaller $b_3$ if necessary, we can suppose moreover than $\tilde\gamma_3([a_3,b_3])\subset L(\tilde \gamma_*)$. We will prove that it prevents $\tilde\gamma_1$ accumulating positively in $\tilde\gamma_2$. If $\tilde \lambda$ is an oriented line of $\tilde B$, denote $R_{\tilde B}(\tilde\lambda)$ the connected component of $\tilde B\setminus\tilde\lambda$ located on the right of $\tilde\lambda$ and $L_{\tilde B}(\tilde\lambda)$ the connected component of $\tilde B\setminus\tilde\lambda$ located on the left of $\tilde\lambda$. One defines two oriented lines $\tilde\lambda_2$,  $\tilde\lambda_3$ of $\tilde B$ by setting
$$\tilde\lambda_2=(\tilde\gamma_2\vert_{[a'_2,b_2)})^{-1} \tilde\phi^+_{\tilde\gamma_2(a'_2)}, \quad \tilde\lambda_3=\tilde\gamma_3\vert_{(a_3,b_3]} \tilde\phi^+_{\tilde \gamma_3(b_3)}.$$ 
The line $\tilde\gamma_*$ intersects $\tilde\phi_{\tilde\gamma_2(a'_2)}$ in a unique point $\tilde z_2$ and we have $\tilde z_2\in\tilde\phi^+_{\tilde\gamma_2(a'_2)}$.
Similarly, $\tilde\gamma_*$ intersects $\tilde\phi_{\tilde\gamma_3(b_3)}$ in a unique point $\tilde z_3$ and we have $\tilde z_3\in\tilde\phi^+_{\tilde\gamma_3(b_3)}$.
Denote $\tilde\sigma_2\subset\tilde\phi_{\tilde\gamma_2(a_2)}$ the segment that joins $\tilde\gamma_2(a'_2)$ to $\tilde z_2$ and $\tilde\sigma_3\subset\tilde\phi_{\tilde\gamma(b_3)}$ the segment that joins $\tilde\gamma_3(b_3)$ to $\tilde z_3$.
By compactness of all segments, if $n$ is large enough, then we have
$$T^n\left(\tilde\gamma_3([a_3,b_3])\cup \tilde\sigma_3\right)\cap \left(\tilde\gamma_2([a'_2,b_2])\cup \sigma_2\right)=\emptyset.$$
Moreover, one can suppose that 
$$T^n\tilde \phi_{\tilde\gamma_3(b_3)}\subset L(\tilde\phi_{\tilde\gamma_2(a'_2)}).$$ 
The fact that $\tilde \gamma_2([a'_2,b_2])$ and  $\tilde \gamma_3([a_3,b_3])$ are included in $L(\tilde\gamma_*)$ while $\tilde \phi^+_{\tilde\gamma_2(b_2)}$ and $\tilde \phi^+_{\tilde\gamma_3(b_3)}$ are included in $\overline {R(\tilde\gamma_*)}$ tells us that
$$T^n\tilde \phi_{\tilde\gamma(b_3)}\cap \left(\tilde\gamma_2([a'_2,b_2])\cup \tilde\sigma_2\right)=\emptyset.$$
We deduce that the lines $\tilde\lambda_2$ and $T^n\tilde\lambda_3$ are disjoint.

\begin{figure}
\begin{center}
\tikzset{every picture/.style={line width=1.2pt}} 

\begin{tikzpicture}[x=0.75pt,y=0.75pt,yscale=-1.3,xscale=1.3]

\draw  [color={rgb, 255:red, 0; green, 0; blue, 0 }  ,draw opacity=1 ][fill={rgb, 255:red, 245; green, 166; blue, 35 }  ,fill opacity=0.08 ] (262.83,50) -- (600,50) -- (600,260) -- (262.83,260) -- cycle ;
\clip (262.83,50) rectangle (600,260);

\draw [color={rgb, 255:red, 245; green, 166; blue, 35 }  ,draw opacity=1 ][fill={rgb, 255:red, 255; green, 255; blue, 255 }  ,fill opacity=1 ][line width=0.75]    (530,50) .. controls (538.61,107.47) and (556.88,134.58) .. (575.72,136.92) .. controls (598.07,139.71) and (621.23,107.62) .. (630,50) ;
\draw [shift={(543.72,103.2)}, rotate = 248.4] [fill={rgb, 255:red, 245; green, 166; blue, 35 }  ,fill opacity=1 ][line width=0.08]  [draw opacity=0] (8.04,-3.86) -- (0,0) -- (8.04,3.86) -- (5.34,0) -- cycle    ;
\draw [shift={(616.41,100.13)}, rotate = 114.7] [fill={rgb, 255:red, 245; green, 166; blue, 35 }  ,fill opacity=1 ][line width=0.08]  [draw opacity=0] (8.04,-3.86) -- (0,0) -- (8.04,3.86) -- (5.34,0) -- cycle    ;
\draw [color={rgb, 255:red, 245; green, 166; blue, 35 }  ,draw opacity=1 ][fill={rgb, 255:red, 255; green, 255; blue, 255 }  ,fill opacity=1 ][line width=0.75]    (420,50) .. controls (430.39,69.71) and (446.18,100.5) .. (465.19,107.41) .. controls (481.54,113.36) and (500.27,101.64) .. (520,50) ;
\draw [shift={(436.01,78.45)}, rotate = 56.42] [fill={rgb, 255:red, 245; green, 166; blue, 35 }  ,fill opacity=1 ][line width=0.08]  [draw opacity=0] (8.04,-3.86) -- (0,0) -- (8.04,3.86) -- (5.34,0) -- cycle    ;
\draw [shift={(498.3,92.43)}, rotate = 305.21] [fill={rgb, 255:red, 245; green, 166; blue, 35 }  ,fill opacity=1 ][line width=0.08]  [draw opacity=0] (8.04,-3.86) -- (0,0) -- (8.04,3.86) -- (5.34,0) -- cycle    ;
\draw [color={rgb, 255:red, 245; green, 166; blue, 35 }  ,draw opacity=1 ][fill={rgb, 255:red, 255; green, 255; blue, 255 }  ,fill opacity=1 ][line width=0.75]    (300,50) .. controls (307.45,99.73) and (322.14,126.73) .. (338.16,134.62) .. controls (362.63,146.66) and (390.23,114.15) .. (400,50) ;
\draw [shift={(312.22,99.48)}, rotate = 249.95] [fill={rgb, 255:red, 245; green, 166; blue, 35 }  ,fill opacity=1 ][line width=0.08]  [draw opacity=0] (8.04,-3.86) -- (0,0) -- (8.04,3.86) -- (5.34,0) -- cycle    ;
\draw [shift={(384.97,103.33)}, rotate = 116.45] [fill={rgb, 255:red, 245; green, 166; blue, 35 }  ,fill opacity=1 ][line width=0.08]  [draw opacity=0] (8.04,-3.86) -- (0,0) -- (8.04,3.86) -- (5.34,0) -- cycle    ;
\draw [color={rgb, 255:red, 65; green, 117; blue, 5 }  ,draw opacity=1 ][line width=1.5]    (354,136.96) .. controls (350.35,142.46) and (347.03,146.89) .. (331,151.57) .. controls (316.72,179.62) and (324.52,220.93) .. (332,260.28) ;
\draw [color={rgb, 255:red, 74; green, 144; blue, 226 }  ,draw opacity=1 ]   (360,50) .. controls (373.48,65.44) and (374.68,141.44) .. (330,150) ;
\draw [shift={(366.7,103.24)}, rotate = 101.43] [fill={rgb, 255:red, 74; green, 144; blue, 226 }  ,fill opacity=1 ][line width=0.08]  [draw opacity=0] (8.04,-3.86) -- (0,0) -- (8.04,3.86) -- (5.34,0) -- cycle    ;
\draw [color={rgb, 255:red, 65; green, 117; blue, 5 }  ,draw opacity=1 ][line width=1.5]    (472,108.96) .. controls (475.62,128.19) and (482.77,139.32) .. (499,142) .. controls (507,182.03) and (421,230.5) .. (487.47,260.42) ;
\draw [color={rgb, 255:red, 74; green, 144; blue, 226 }  ,draw opacity=1 ]   (460,50) .. controls (473.48,65.44) and (463.48,137.44) .. (500,140) ;
\draw [shift={(472.19,104.05)}, rotate = 259.03] [fill={rgb, 255:red, 74; green, 144; blue, 226 }  ,fill opacity=1 ][line width=0.08]  [draw opacity=0] (8.04,-3.86) -- (0,0) -- (8.04,3.86) -- (5.34,0) -- cycle    ;
\draw [color={rgb, 255:red, 245; green, 166; blue, 35 }  ,draw opacity=1 ]   (500,140) .. controls (511.48,183.44) and (420.68,230.63) .. (490,260) ;
\draw [shift={(472.2,207.15)}, rotate = 301.29] [fill={rgb, 255:red, 245; green, 166; blue, 35 }  ,fill opacity=1 ][line width=0.08]  [draw opacity=0] (8.04,-3.86) -- (0,0) -- (8.04,3.86) -- (5.34,0) -- cycle    ;
\draw [color={rgb, 255:red, 245; green, 166; blue, 35 }  ,draw opacity=1 ]   (330,150) .. controls (315.88,177.83) and (320.68,214.23) .. (330,260) ;
\draw [shift={(321.81,208.38)}, rotate = 266.05] [fill={rgb, 255:red, 245; green, 166; blue, 35 }  ,fill opacity=1 ][line width=0.08]  [draw opacity=0] (8.04,-3.86) -- (0,0) -- (8.04,3.86) -- (5.34,0) -- cycle    ;
\draw [color={rgb, 255:red, 144; green, 19; blue, 254 }  ,draw opacity=1 ]   (410,200) .. controls (433.83,200.17) and (478.22,182.26) .. (530,190) .. controls (581.78,197.74) and (621.83,200.17) .. (650,200) ;
\draw [shift={(472.91,190.14)}, rotate = 172.05] [fill={rgb, 255:red, 144; green, 19; blue, 254 }  ,fill opacity=1 ][line width=0.08]  [draw opacity=0] (8.04,-3.86) -- (0,0) -- (8.04,3.86) -- (5.34,0) -- cycle    ;
\draw [shift={(593.38,197.52)}, rotate = 185.04] [fill={rgb, 255:red, 144; green, 19; blue, 254 }  ,fill opacity=1 ][line width=0.08]  [draw opacity=0] (8.04,-3.86) -- (0,0) -- (8.04,3.86) -- (5.34,0) -- cycle    ;
\draw [color={rgb, 255:red, 144; green, 19; blue, 254 }  ,draw opacity=1 ]   (170,200) .. controls (193.83,200.17) and (238.22,182.26) .. (290,190) .. controls (341.78,197.74) and (381.83,200.17) .. (410,200) ;
\draw [shift={(232.91,190.14)}, rotate = 172.05] [fill={rgb, 255:red, 144; green, 19; blue, 254 }  ,fill opacity=1 ][line width=0.08]  [draw opacity=0] (8.04,-3.86) -- (0,0) -- (8.04,3.86) -- (5.34,0) -- cycle    ;
\draw [shift={(353.38,197.52)}, rotate = 185.04] [fill={rgb, 255:red, 144; green, 19; blue, 254 }  ,fill opacity=1 ][line width=0.08]  [draw opacity=0] (8.04,-3.86) -- (0,0) -- (8.04,3.86) -- (5.34,0) -- cycle    ;
\draw  [fill={rgb, 255:red, 0; green, 0; blue, 0 }  ,fill opacity=1 ] (471.7,108.54) .. controls (471.7,107.61) and (472.45,106.86) .. (473.38,106.86) .. controls (474.31,106.86) and (475.06,107.61) .. (475.06,108.54) .. controls (475.06,109.47) and (474.31,110.22) .. (473.38,110.22) .. controls (472.45,110.22) and (471.7,109.47) .. (471.7,108.54) -- cycle ;
\draw  [fill={rgb, 255:red, 0; green, 0; blue, 0 }  ,fill opacity=1 ] (350.51,136.67) .. controls (350.51,135.74) and (351.26,134.99) .. (352.19,134.99) .. controls (353.12,134.99) and (353.87,135.74) .. (353.87,136.67) .. controls (353.87,137.6) and (353.12,138.35) .. (352.19,138.35) .. controls (351.26,138.35) and (350.51,137.6) .. (350.51,136.67) -- cycle ;
\draw  [fill={rgb, 255:red, 0; green, 0; blue, 0 }  ,fill opacity=1 ] (319.42,194.44) .. controls (319.42,193.51) and (320.17,192.75) .. (321.1,192.75) .. controls (322.02,192.75) and (322.78,193.51) .. (322.78,194.44) .. controls (322.78,195.36) and (322.02,196.12) .. (321.1,196.12) .. controls (320.17,196.12) and (319.42,195.36) .. (319.42,194.44) -- cycle ;
\draw  [fill={rgb, 255:red, 0; green, 0; blue, 0 }  ,fill opacity=1 ] (481.96,189.09) .. controls (481.96,188.17) and (482.71,187.41) .. (483.64,187.41) .. controls (484.57,187.41) and (485.32,188.17) .. (485.32,189.09) .. controls (485.32,190.02) and (484.57,190.78) .. (483.64,190.78) .. controls (482.71,190.78) and (481.96,190.02) .. (481.96,189.09) -- cycle ;
\draw  [fill={rgb, 255:red, 0; green, 0; blue, 0 }  ,fill opacity=1 ] (498.32,140) .. controls (498.32,139.07) and (499.07,138.32) .. (500,138.32) .. controls (500.93,138.32) and (501.68,139.07) .. (501.68,140) .. controls (501.68,140.93) and (500.93,141.68) .. (500,141.68) .. controls (499.07,141.68) and (498.32,140.93) .. (498.32,140) -- cycle ;
\draw  [fill={rgb, 255:red, 0; green, 0; blue, 0 }  ,fill opacity=1 ] (328.32,150) .. controls (328.32,149.07) and (329.07,148.32) .. (330,148.32) .. controls (330.93,148.32) and (331.68,149.07) .. (331.68,150) .. controls (331.68,150.93) and (330.93,151.68) .. (330,151.68) .. controls (329.07,151.68) and (328.32,150.93) .. (328.32,150) -- cycle ;
\draw [color={rgb, 255:red, 245; green, 166; blue, 35 }  ,draw opacity=1 ][fill={rgb, 255:red, 255; green, 255; blue, 255 }  ,fill opacity=1 ][line width=0.75]    (190,50) .. controls (200.2,69.35) and (215.6,99.37) .. (234.14,107.01) .. controls (250.74,113.84) and (269.85,102.74) .. (290,50) ;
\draw [shift={(205.67,77.91)}, rotate = 56.65] [fill={rgb, 255:red, 245; green, 166; blue, 35 }  ,fill opacity=1 ][line width=0.08]  [draw opacity=0] (8.04,-3.86) -- (0,0) -- (8.04,3.86) -- (5.34,0) -- cycle    ;
\draw [shift={(267.85,93.02)}, rotate = 305.66] [fill={rgb, 255:red, 245; green, 166; blue, 35 }  ,fill opacity=1 ][line width=0.08]  [draw opacity=0] (8.04,-3.86) -- (0,0) -- (8.04,3.86) -- (5.34,0) -- cycle    ;
\draw  [color={rgb, 255:red, 0; green, 0; blue, 0 }  ,draw opacity=1 ] (262.83,50) -- (600,50) -- (600,260) -- (262.83,260) -- cycle ;
\draw [color={rgb, 255:red, 245; green, 166; blue, 35 }  ,draw opacity=1 ]   (550,150) .. controls (579.79,179.69) and (550.59,233.28) .. (580,260) ;
\draw [shift={(565.15,208.56)}, rotate = 269.68] [fill={rgb, 255:red, 245; green, 166; blue, 35 }  ,fill opacity=1 ][line width=0.08]  [draw opacity=0] (8.04,-3.86) -- (0,0) -- (8.04,3.86) -- (5.34,0) -- cycle    ;
\draw [color={rgb, 255:red, 245; green, 166; blue, 35 }  ,draw opacity=1 ] [dash pattern={on 0.84pt off 2.51pt}]  (520,140) .. controls (536.19,142.89) and (538.19,142.89) .. (550,150) ;

\draw (279.6,192.1) node [anchor=north] [inner sep=0.75pt]  [color={rgb, 255:red, 144; green, 19; blue, 254 }  ,opacity=1 ]  {$\tilde{\gamma }_{*}$};
\draw (330.48,169.77) node [anchor=west] [inner sep=0.75pt]  [font=\normalsize,color={rgb, 255:red, 65; green, 117; blue, 5 }  ,opacity=1 ]  {$\tilde{\lambda }_{2}$};
\draw (480,154.1) node [anchor=west] [inner sep=0.75pt]  [font=\normalsize,color={rgb, 255:red, 65; green, 117; blue, 5 }  ,opacity=1 ]  {$\tilde{\lambda }_{3}$};
\draw (367.97,81.86) node [anchor=east] [inner sep=0.75pt]  [font=\normalsize,color={rgb, 255:red, 74; green, 144; blue, 226 }  ,opacity=1 ]  {$\tilde{\gamma }_{2}$};
\draw (469.99,85.11) node [anchor=south west] [inner sep=0.75pt]  [font=\normalsize,color={rgb, 255:red, 74; green, 144; blue, 226 }  ,opacity=1 ]  {$\tilde{\gamma }_{3}$};
\draw (472.21,111.16) node [anchor=north east] [inner sep=0.75pt]  [font=\normalsize]  {$\tilde{\gamma }_{3}( a_{3})$};
\draw (445,80) node [anchor=west] [inner sep=0.75pt]  [font=\normalsize,color={rgb, 255:red, 245; green, 166; blue, 35 }  ,opacity=1 ] {$\tilde \phi_0$};
\draw (353.19,138.65) node [anchor=north west][inner sep=0.75pt]  [font=\normalsize]  {$\tilde{\gamma }_{2}( b_{2})$};
\draw (319.75,189.59) node [anchor=south east] [inner sep=0.75pt]  [font=\normalsize]  {$\tilde{z}_{2}$};
\draw (484.44,191.17) node [anchor=north west][inner sep=0.75pt]  [font=\normalsize]  {$\tilde{z}_{3}$};
\draw (470.31,237.09) node [anchor=west] [inner sep=0.75pt]  [font=\normalsize,color={rgb, 255:red, 245; green, 166; blue, 35 }  ,opacity=1 ]  {$\tilde \phi _{\tilde{\gamma }_{3}( b_{3})}^{+}$};
\draw (323.91,240.89) node [anchor=east] [inner sep=0.75pt]  [font=\normalsize,color={rgb, 255:red, 245; green, 166; blue, 35 }  ,opacity=1 ]  {$\tilde \phi _{\tilde{\gamma }_{2}( a'_{2})}^{+}$};
\draw (411,215.4) node [anchor=north west][inner sep=0.75pt]  [color={rgb, 255:red, 194; green, 120; blue, 0 }  ,opacity=1 ]  {$\tilde{B}$};
\draw (502,136.6) node [anchor=south west] [inner sep=0.75pt]  [font=\normalsize]  {$\tilde{\gamma }_{3}( b_{3})$};
\draw (328,147.6) node [anchor=south east] [inner sep=0.75pt]  [font=\normalsize]  {$\tilde{\gamma }_{2}( a'_{2})$};
\draw (568.67,220.17) node [anchor=west] [inner sep=0.75pt]  [color={rgb, 255:red, 245; green, 166; blue, 35 }  ,opacity=1 ]  {$\tilde{\phi }$};

\end{tikzpicture}
\caption{The configuration of the proof of Lemma~\ref{lem:NotLeft}.
\label{Fig:NotLeft}}
\end{center}
\end{figure}
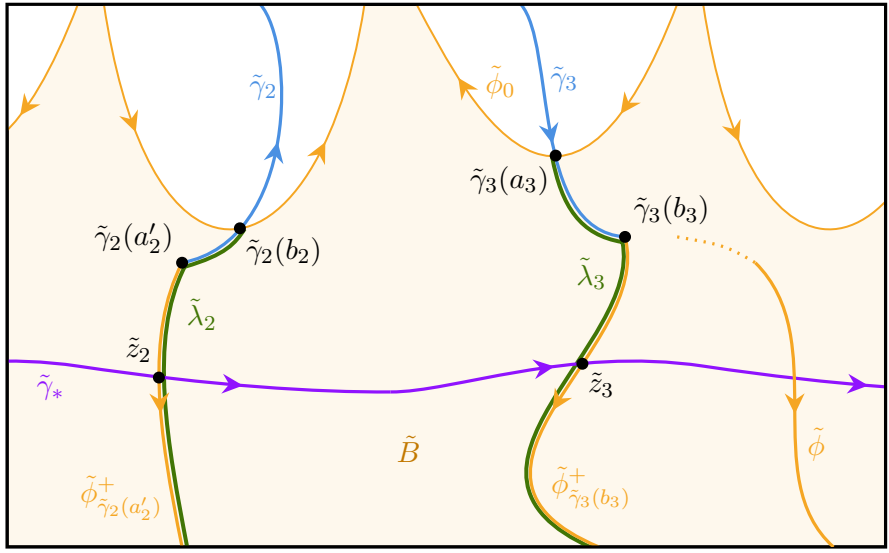

The sub-path of $\tilde\gamma_*$ that joins $\tilde z_2$ to $T^n\tilde z_3$ is disjoint from $\tilde\lambda_2$ and $T^n\tilde\lambda_3$  but at the endpoints, entering in $L_{\tilde B}( \tilde\lambda_2)$ at $\tilde z_2$ and leaving $R_{\tilde B}( T^n\tilde\lambda_3)$ at $T^n\tilde z_3$. Consequently the following inclusion $\overline{L_{\tilde B}(T^n\tilde\lambda_3)}\subset L_{\tilde B}(\tilde\lambda_2)$ holds.  Every leaf $\tilde\phi\subset L_{\tilde B}(\tilde\phi_{T^n\tilde z_3})$ is disjoint from  $T^n\tilde\lambda_3$. It is contained in $L(T^n\tilde\lambda_3)$ because the sub-path of $\tilde\gamma_*$ that joins $\tilde\phi_{T^n(\tilde z_3)}$ to $\tilde\phi$  is disjoint from  $T^n\tilde\lambda_3$ but at $T^n\tilde z_3$ and enters in $L_{\tilde B}(T^n\tilde\lambda_3)$ at $T^n\tilde z_3$.
The contradiction comes from the fact that $\tilde \phi$ must intersect $\tilde\gamma_2\vert_{[a'_2,b_2)}$ because $\tilde\phi\subset L_{\tilde B}(\tilde\phi_{z_2}).$
\end{proof}

\begin{lemma}\label{Lem:ProjAnnulus}
The set $B$ is an open annulus of $\Sigma$.
\end{lemma}

\begin{proof}
Suppose it is not. Then there exists a deck transformation $R\notin\langle T\rangle$ of $\tilde\Sigma$ such that $R\tilde B\cap \tilde B\neq \emptyset$.  As $\tilde B$ is the set of leaves met by $\tilde\gamma_*$, it implies the existence of $t\in\R$ such that $R\tilde\gamma_*(t) \in \tilde B$.
The line $\tilde\gamma_*$ lifts the simple loop $\Gamma_*$  and so we have $R\tilde\gamma_*\cap \tilde\gamma_*=\emptyset$. Moreover, there is at least one leaf of $\tilde{\F}$ that is met both by $\tilde\gamma_*$ and $R\tilde\gamma_*$. Consequently, one of the following inclusions $\overline{L(R\tilde\gamma_*)}\subset L(\tilde\gamma_*)$, $\overline{L(\tilde\gamma_*)}\subset L(R\tilde\gamma_*)$ holds. Replacing $R$ by $R{}^{-1}$ if necessary, one can suppose that the first inclusion holds, which implies that $R\tilde\gamma_*\subset L(\tilde\gamma_*)$.

Note that $R\tilde\gamma_*$ cannot accumulate on $\tilde\gamma_*$ (neither positively nor negatively) because the natural lift $\gamma_*$ of $\Gamma_*$ is recurrent and so, by Lemma~\ref{LemEquivSubpath}, cannot accumulate on itself.  Moreover it cannot be equivalent to $\tilde\gamma_*$ neither at $+\infty$ nor at $-\infty$ (by using Lemma~\ref{LemEquivSubpath}). It cannot cross $\tilde B$ because $R\tilde\gamma_*\cap \tilde\gamma_*=\emptyset$.  It remains to prove that it cannot visit $\tilde B$. 

Using the fact  that $R\tilde\gamma_*\subset L(\tilde\gamma_*)$, the line $R\tilde\gamma_*$ must visit $\tilde B$ by the left if it visits $\tilde B$. This contradicts Lemma~\ref{lem:NotLeft}: no transverse trajectory enters in $\tilde B$ by the left side.
\end{proof}

To prove Proposition \ref{prop: nontransitivity3}, it remains to prove that $\gamma_1$ is entirely contained in $B$ (which will imply that  $\tilde\gamma_1$ is entirely contained in $\tilde B$). But this is implied by the facts that $\gamma_1\vert_{[a_1,+\infty)}$ is contained in $B$ and that $\gamma_1$ is recurrent. This finishes the proof of Proposition~\ref{prop: nontransitivity3}.
\end{proof}

The following results (and others related to the accumulation property) were already stated by Lellouch in \cite[Section 2.1.1]{Lel}. Using the precise description given here, we get them as a trivial corollary.

\begin{corollary}\label{coro:recurrence}
Suppose that $\gamma_1:\R\to \Sigma$ is a positively recurrent transverse path that accumulates positively on a transverse path $\gamma_2:\R\to \Sigma$. Then there is no positively or negatively recurrent transverse path  $\gamma_0:\R\to \Sigma$ that accumulates positively or negatively on $\gamma_1$. In particular a positively recurrent transverse path does not accumulate on itself. Also, the accumulated leaf $\phi_{\gamma_2(b_2)}$ is not met by $\gamma_1$. 
\end{corollary}

\begin{proof}
To prove the first point, it suffices to note that by Proposition~\ref{prop: nontransitivity3}, the function $t\mapsto \phi_{\gamma_1(t)}$ is locally injective. The last point comes from the fact that $\gamma_1$ is contained in  $B$ while $\phi_{\gamma_2(b_2)}$ is contained in the frontier of $B$.
\end{proof}

\section{Forcing theory}\label{s.forcingtheory}

\subsection{Maximal isotopies and transverse foliations}\label{ss.maximal}

Let $\Sigma$ be an oriented boundaryless surface, not necessarily closed, not necessarily connected and $f$ a homeomorphism isotopic to the identity.  Recall that if $I=(f_t)_{t\in[0,1]}$ is an identity isotopy of $f$, the trajectory $I(z)$ of a point $z\in \Sigma$ is the path $t\mapsto f_t(z)$ defined on $[0,1]$. We can define the {\it whole trajectory} of $z$ as being the path $$ I^{\Z}(z)=\prod_{k\in\Z} I(f^k(z))$$ constructed by concatenation. More precisely,  on every interval $[k,k+1]$, $k\in\Z$, it is defined by the formula:
$$ I^{\Z}(z): t\mapsto f_{t-k}(f^k(z)).$$

We define the {\it fixed point set} and the {\it domain} of $I$ as follows:
$$\mathrm{fix}(I)=\bigcap_{t\in[0,1]} \mathrm{fix}(f_t)\, ,\enskip \mathrm{dom}(I)=\Sigma\setminus \mathrm{fix}(I).$$
Denote $\mathcal I$ the set of identity isotopies of $f$. We have a preorder on $\mathcal I$ defined as follows: say that $I\preceq I'$ if
\begin{itemize}
\item $\mathrm{fix}(I)\subset\mathrm{fix}(I')$;
\item $I'$ is homotopic to $I$ relative to $\mathrm{fix}(I)$.
\end{itemize}

Let us state two  important results. The first one is due to B\'eguin-Crovisier-Le Roux \cite{BeCLer} (see also \cite{J} for a weaker version). The second can be found in \cite{Lec1}.

\begin{theorem}\label{th:BCL}
For every $I\in\mathcal I$, there exists $I'\in\mathcal I$ such that $I\preceq I'$ and such that $I'$ is maximal for the preorder.
\end{theorem}

\begin{remark*}
An isotopy $I$ is maximal if and only if, for every $z\in\mathrm{fix}(f)\setminus\mathrm{fix}(I)$, the loop $I(z)$ is not contractible in $\mathrm{dom}(I)$. Equivalently, if we lift the isotopy $I\vert_{\mathrm{dom}(I)}$ to an identity isotopy $\widetilde I=(\widetilde f_t)_{t\in[0,1]}$ on the universal covering space $\widetilde{\mathrm{dom}}(I)$ of $\mathrm{dom}(I)$, the maximality of $I$ means that $\widetilde f_1$ is fixed point free. Note that every connected component of $\widetilde{\mathrm{dom}}(I)$ must be a topological plane. 
\end{remark*}

\begin{theorem}\label{th.transversefoliation}
If $I\in\mathcal I$ is maximal, then there exists a topological oriented singular foliation $\F$ on $M$ such that 

\begin{itemize}
\item the singular set $\mathrm{sing}(\F)$ coincides with $\mathrm{fix}(I)$;
\item for every $z\in \mathrm{dom}(I)$, the trajectory $I(z)$ is homotopic in  $\mathrm{dom}(I)$, relative to the ends, to a transverse path $\gamma$ joining $z$ to $f(z)$.
\end{itemize}
\end{theorem}

We will say that $\F$ is {\it transverse to $I$}. It can be lifted to a non singular foliation $\widetilde{\F}$ on $\widetilde{\mathrm{dom}}(I)$ which is transverse to $\widetilde I$. This last property is equivalent to saying that every leaf $\widetilde\phi$ of $\widetilde{\F}$ is a Brouwer line of the lift $\tilde f$ induced by $I$, as defined in Section \ref{ss:loopsandpaths}. The path $\gamma$ is uniquely defined up to equivalence: if  $\gamma_1$ and $\gamma_2$ are two such paths and if $z\in\widetilde{\mathrm{dom}}(I)$ lifts $z\in{\mathrm{dom}}(I)$, then the respective lifts  $\tilde\gamma_1$, $\tilde\gamma_2$ of  $\gamma_1$, $\gamma_2$ starting at $\tilde z$ join this point to $\tilde f(\tilde z)$ and consequently meet the same leaves of $\tilde {\F}$. We will write $\gamma=I_{\F}(z)$ and call this path the {\it transverse trajectory of $z$}. It is defined, up to equivalence, on [$0,1]$. For every $n\geq 1$, we will define by concatenation the path
$$I^n_{\F}(z)= I_{\F} (z) I_{\F}(f(z))\cdots I_{\F}(f^{n-1}(z)).$$
We can also define the {\it whole transverse trajectory} of $z$ as being the path 
$$ I^{\Z}_{\F}(z)=\prod_{k\in\Z} I_{\F}(f^k(z))$$ 
coinciding on $[k,k+1]$, $k\in\Z$,  with  $I_{\F}(f^k(z))$ after translation by $-k$. 
Similarly, we define
$${\tilde I}^n_{\tilde {\F}}(\tilde z)= \tilde I_{\tilde {\F}} (\tilde z) \tilde I_{\tilde{\F}}(\tilde f(\tilde z))\cdots \tilde I_{\tilde {\F}}(\tilde f^{n-1}(\tilde z))$$
and
$$\tilde I^{\Z}_{\tilde{\F}}(\tilde z)=\prod_{k\in\Z} \tilde I_{\tilde{\F}}(\tilde f^k(\tilde z)).$$

Recall that a {\it flow-box} of $\tilde{\F}$ is an open disk $\tilde U$ of $\widetilde{\mathrm{dom}}(I)$ such that the foliation $\tilde{\F}|_{\tilde U}$ is homeomorphic to the foliation of $\R^2$ by verticals. The following results,  easy to prove (see \cite{LecT1}), will be useful in the article. 

\begin{proposition}
\label {prop:stability} For every $\tilde z\in\widetilde{\mathrm{dom}}(I)$ and every pair of integers $k_1<k_2$ there exists a neighborhood $\tilde U$ of $\tilde z$ such that $\tilde I^{\Z}_{\tilde {\F}}(\tilde z)|_{ [k_1,k_2]}$ is a subpath (up to equivalence) of $\tilde I^{\Z}_{\tilde {\F}}(\tilde z')|_{ [k_1-1,k_2+1]}$.
\end{proposition}

\begin{proposition}
\label{prop:flowbox} For every $\tilde z\in\widetilde{\mathrm{dom}}(I)$ and every neighborhood $\tilde V$ of $\tilde z$, there exists a flow-box $\tilde U\subset \tilde V$ containing $\tilde z$, such that for every $\tilde z'\in \tilde U$, the path $\tilde I^{\Z}_{\tilde {\F}}(\tilde z')$ intersects every leaf that meets $\tilde U$.
\end{proposition}

Remind that if $f$ is a homeomorphism of $\Sigma$, a point $z$ is {\it positively recurrent} if $z\in\omega(z)$ and  {\it negatively recurrent} if $z\in\alpha(z)$. In the case where $z\in \alpha(z)\cap\omega(z)$, we say that $z$ is {\it recurrent}. For instance, if $\mu$ is an invariant finite Borel measure on $S$, then $\mu$-almost every point is recurrent.
The following result is an immediate consequence of Proposition \ref{prop:stability}.

\begin{proposition}\label{prop:recurrence} 
If $z\in\mathrm{dom}(I)$ is positively recurrent, then $I^{\Z}_{\F}(z)$ is positively recurrent.  If $z$ is negatively recurrent, then $I^{\Z}_{\F}(z)$ is negatively recurrent. 
\end{proposition}

Let us state now the key lemma of \cite{LecT1} (Proposition 20) that is the elementary brick of the forcing theory and which will be used later. 

\begin{lemma}\label{le:forcing} 
Suppose that there exist $\tilde z_1$, $\tilde z_2$ in $\widetilde{\mathrm{dom}}(I)$ and positive integers  $n_1$, $n_2$ such that
$\tilde I^{n_1}_{\tilde {\F}}(\tilde z_1)$ and $\tilde I^{n_2}_{\tilde {\F}}(\tilde z_2)$ have an $\tilde{\F}$-transverse intersection at 
$\tilde I^{n_1}_{\tilde {\F}}(\tilde z_1)(t_1)=\tilde I^{n_2}_{\tilde {\F}}(\tilde z_2)(t_2)$. Then there exists  $\tilde z_3\in \widetilde{\mathrm{dom}}(I)$ such that $\tilde I^{n_1+n_2}_{\tilde {\F}}(\tilde z_3)$ is equivalent to $ \tilde I^{n_1}_{\tilde {\F}}(\tilde z_1)|_{[0,t_1]}\tilde I^{n_2}_{\tilde {\F}}(\tilde z_2)|_{[t_2,n_2]}$.
\end{lemma}

Let us give now the principal result of \cite{LecT2}. Here, $\mathcal G$ is the group of covering automorphisms of $\widetilde{\mathrm{dom}}(I)$ and $[T]_{\mathcal {FHL}}\in \mathcal {FHL}(S)$ is the free homotopy class (in $S$) of a loop $\Gamma\subset \mathrm{dom}(I)$ naturally defined by $T$ (see Paragraph~\ref{ss:horseshoe}).

\begin{theorem} \label{th:rotationalhorseshoe1}  
Suppose that there exists $\tilde z\in \widetilde{\mathrm{dom}}(I)$, $T\in \mathcal G\setminus\{\mathrm{Id}\}$ and $r\geq 1$ such that ${\tilde I}^r_{\tilde {\F}}(\tilde z)$ and $T{\tilde I}^r_{\tilde {\F}}(\tilde z)$ have an $\tilde{\F}$-transverse intersection at ${\tilde I}^r_{\tilde {\F}}(\tilde z)(a)=T({\tilde I}^r_{\tilde {\F}}(\tilde z))(a')$ where $a'<a$. Then  $f$ admits a rotational horseshoe of type $([T]_{\mathcal {FHL}}, r)$.
\end{theorem}

\begin{proof} 
What is proved in \cite{LecT2} is the following, where $\widehat{\mathrm{dom}}(I)=\widetilde{\mathrm{dom}}(I)/T$ and $\hat f$ is the homeomorphism of $\widehat{\mathrm{dom}}(I)$ induced by $\tilde f$.

There exists an $\hat f^r$-invariant compact set $\hat Y$ such that
\begin{itemize}
\item $\hat f^r$ is an extension of the Bernouilli shift $\sigma :\{1, 2\}^{\Z}\to \{1, 2\}^{\Z}$;
\item the preimage of every $q$-periodic sequence of $\{1, 2\}^{\Z}$  by the factor map contains  at least one $q$-periodic point of $\hat f^r$;
\item for every  $p/q\in[0,1]\cap \Q$ written in an irreducible way, there exists $\hat z_{p/q}\in \hat Y$ such that $\tilde f^{rq}(\tilde z_{p/q})=T^p(\tilde z_{p/q})$ if $\tilde z_{p/q}\in \widetilde{\mathrm{dom}}(I)$ lifts $\hat z_{p/q}$.
\end{itemize}
The image $Y$ of $\hat Y$ by the covering projection $\hat\pi: \widehat{\mathrm{dom}}(I)\to\mathrm{dom}(I)$ is invariant by $f^r$. It is a topological horseshoe because $\hat\pi|_{ \hat Y}$ is a semi-conjugacy from $\tilde f^r{}|_{ \hat Y}$ to $f^r|_{ Y}$ and because every $z\in Y$ has finitely many lifts in $\hat Y$ (with an uniform bound $s$) because $\hat Y$ is compact. The loop of $S$ naturally defined by $I^{rq}(z_{p/q})$, where $z_{p/q}=\tilde\pi( \tilde z_{p/q})$, belongs to $[T]_{\FHL}^p$. Moreover, the $\hat f^r$-orbit of $\hat z_{p/q}$ has $q$ points because $p$ and $q$ are relatively prime. It projects onto the $f^r$-orbit of $z_{p/q}$, which has at least $q/s$ points. So, the period of $z_{p/q}$ (for $f$) is at least $q/s$.
\end{proof}

\begin{remark*} 
In particular, the theorem asserts the existence of a topological horseshoe, and so the positiveness of the topological entropy, in the case where there exists $z\in \mathrm{dom}(I)$ such that $I^{\Z}_{\F}(z)$ has an $\F$-transverse self-intersection. It was proved in  \cite{LecT1} that such a situation occurs in the case where there exist two positively (or negatively) recurrent points $z_1$, $z_2$ in $\mathrm{dom}(I)$ such that $I^{\Z}_{\F}(z_1)$ and $I^{\Z}_{\F}(z_2)$ have an $\F$-transverse intersection. For example this happens if $f$ preserves a Borel probability measure with total support and if there exist two points $z_1$, $z_2$ in $\mathrm{dom}(I)$ such that $I^{\Z}_{\F}(z_1)$ and $I^{\Z}_{\F}(z_2)$ have an $\F$-transverse intersection. Indeed, by Proposition \ref{prop:stability}, it is also the case for $I^{\Z}_{\F}(z'_1)$ and $I^{\Z}_{\F}(z'_2)$ if $z'_1$, $z'_2$ are close to $z_1$, $z_2$ respectively. But if  $f$ preserves a Borel probability measure $\lambda$ with total support, then $\lambda$-almost every point is recurrent and so, the set of recurrent points is dense.
\end{remark*}

What follows, which is stronger than what is said in the previous remark, is crucial  in \cite{Lel} and will also be fundamental in our study.

\begin{corollary} \label{th:rotationalhorseshoe2}  
Suppose that $\Sigma$ is a closed surface and that $\nu_1, \nu_2$ are ergodic invariant probability measures. If there exists $\tilde z_1\in\mathrm{dom}(I)\cap \operatorname{supp}(\nu_1)$ and $\tilde z_2\in\mathrm{dom}(I)\cap \operatorname{supp}(\nu_2)$ such that ${\tilde I}^{\Z}_{\tilde {\F}}(\tilde z_1)$ and ${\tilde I}^{\Z}_{\tilde {\F}}(\tilde z_2)$ intersect $\F$-transversally, then for every neighborhood $\mathcal U$ of $\mathrm{rot}_{f}(\nu_1)$ in $H_1(S,\R)$, there exists $T\in \mathcal G\setminus\{\mathrm{Id}\}$ and  $r\geq 1$ such $[T]/r\in\mathcal U$ and such that $f$ admits a rotational horseshoe of type $([T]_{\FHL}, r)$.
\end{corollary}

Note that this corollary can be applied in the case where $\nu_1=\nu_2$ and some $\tilde z\in\mathrm{dom}(I)\cap \mathrm{supp}(\nu_1)$ is such that ${\tilde I}^{\Z}_{\tilde {\F}}(\tilde z)$ has an $\F$-transverse self-intersection.

\begin{proof} 
Let $j\in\{1,2\}$. One knows that $\nu_j$-almost every point $z_j'$ satisfies the following properties:
 
\begin{itemize}
\item $z_j'$ is recurrent;
\item its orbit is dense in $\mathrm{supp}(\nu_j)$;
\item if $\tilde z_j'\in \widetilde{\mathrm{dom}}(I)$ is a lift of $z_j'$, then there exists a sequence $(T_{j,i})_{i\geq 0}$ in $\mathcal G$ and a sequence $(n_{j,i})_{i\geq 0}$ in $\N\setminus\{0\}$ such that 
$$\lim_{i\to+\infty} n_{j,i}=+\infty\,,\quad \lim_{i\to+\infty} \frac{[T_{j,i}]}{n_{j,i}}=\mathrm{rot}_f(\nu_j)\,,\quad \lim_{i\to +\infty} T_{j,i}^{-1}f^{n_{j,i}}(\tilde z_j')=\tilde z_j'.$$
\end{itemize}
 
By Proposition \ref{prop:stability} and the hypothesis of the corollary we know that ${I}^{\Z}_{{\F}}(z_1')$ and ${I}^{\Z}_{{\F}}(z_2')$ intersect $\F$-transversally.
So there exists $r'\in\N\setminus\{0\}$, $s_1,s_2\in\Z$ and two lifts $\tilde z_1'$ and $\tilde z_2'$ of $z_1'$ and $z_2'$ such that ${\tilde I}^{r'}_{\tilde {\F}}(\tilde f^{s_1}(\tilde z_1'))$ and ${\tilde I}^{r'}_{\tilde {\F}}(\tilde f^{s_2}(\tilde z_2'))$ intersect $\F$-transversally. Denote $\tilde z_j'' = \tilde f^{s_j}(\tilde z_j')$. See Figure~\ref{Fig:rotationalhorseshoe2} for a description of the proof configuration.

By Proposition \ref{prop:stability}, if $i$ is large enough then, up to equivalence, ${\tilde I}^{r'}_{\tilde {\F}}(\tilde z''_j)$  is a subpath of $T_{j,i}^{-1}{\tilde I}^{r'+2}_{\tilde {\F}}(\tilde f^{n_{j,i}-1}(\tilde z''_j))$. So $T_{1,i}^{-1}{\tilde I}^{r'+2}_{\tilde {\F}}(\tilde f^{n_{1,i}-1}(\tilde z''_1))$ and ${\tilde I}^{r'}_{\tilde {\F}}(\tilde z''_2)$ have an $\tilde{\F}$-transverse intersection at $T_{1,i}^{-1}{\tilde I}^{r'+2}_{\tilde {\F}}(\tilde f^{n_{1,i}-1}(\tilde z''_1))(a) = {\tilde I}^{r'}_{\tilde {\F}}(\tilde z''_2)(b)$, 
as well as $T_{2,i}^{-1}{\tilde I}^{r'+2}_{\tilde {\F}}(\tilde f^{n_{2,i}-1}(\tilde z''_2))$ and ${\tilde I}^{r'}_{\tilde {\F}}(\tilde z''_1)$ have an $\tilde{\F}$-transverse intersection at $T_{2,i}^{-1}{\tilde I}^{r'+2}_{\tilde {\F}}(\tilde f^{n_{2,i}-1}(\tilde z''_2))(c) = {\tilde I}^{r'}_{\tilde {\F}}(\tilde z''_1)(d)$ (we omit here the dependences on $i,i'$ for briefness of notations).

\begin{figure}
\begin{center}

\tikzset{every picture/.style={line width=1.2pt}} 

\begin{tikzpicture}[x=0.75pt,y=0.75pt,yscale=-1.3,xscale=1.3]
\clip  (95,82.6) rectangle (435,291) ;

\draw [color={rgb, 255:red, 208; green, 2; blue, 27 }  ,draw opacity=1 ]   (133.72,231.83) .. controls (177.9,223.46) and (280.37,241.89) .. (344.97,238.79) .. controls (338.17,189.89) and (347.87,135.19) .. (342.27,102.89) ;
\draw [color={rgb, 255:red, 7; green, 90; blue, 190 }  ,draw opacity=1 ]   (133.6,233.2) .. controls (201.4,223.4) and (315,253.8) .. (403.6,233.2) ;
\draw [shift={(272.35,238.13)}, rotate = 184.21] [fill={rgb, 255:red, 7; green, 90; blue, 190 }  ,fill opacity=1 ][line width=0.08]  [draw opacity=0] (8.04,-3.86) -- (0,0) -- (8.04,3.86) -- (5.34,0) -- cycle    ;
\draw [color={rgb, 255:red, 245; green, 166; blue, 35 }  ,draw opacity=1 ]   (137.57,182.14) .. controls (160.14,213.86) and (104.71,257.57) .. (134.31,269.06) ;
\draw [shift={(134.41,230.72)}, rotate = 293.42] [fill={rgb, 255:red, 245; green, 166; blue, 35 }  ,fill opacity=1 ][line width=0.08]  [draw opacity=0] (8.04,-3.86) -- (0,0) -- (8.04,3.86) -- (5.34,0) -- cycle    ;
\draw [color={rgb, 255:red, 245; green, 166; blue, 35 }  ,draw opacity=1 ]   (320.71,189) .. controls (329,217.57) and (351.57,254.43) .. (370.14,264.14) ;
\draw [shift={(341.57,233.64)}, rotate = 238.22] [fill={rgb, 255:red, 245; green, 166; blue, 35 }  ,fill opacity=1 ][line width=0.08]  [draw opacity=0] (8.04,-3.86) -- (0,0) -- (8.04,3.86) -- (5.34,0) -- cycle    ;
\draw [color={rgb, 255:red, 245; green, 166; blue, 35 }  ,draw opacity=1 ]   (299.74,92.49) .. controls (327.74,85.06) and (377,133.29) .. (412.43,111.57) ;
\draw [shift={(358.78,109.6)}, rotate = 202.01] [fill={rgb, 255:red, 245; green, 166; blue, 35 }  ,fill opacity=1 ][line width=0.08]  [draw opacity=0] (8.04,-3.86) -- (0,0) -- (8.04,3.86) -- (5.34,0) -- cycle    ;
\draw [color={rgb, 255:red, 96; green, 117; blue, 5 }  ,draw opacity=1 ]   (138.31,219.34) .. controls (319.17,244.49) and (346.89,241.63) .. (330.31,98.49) ;
\draw [shift={(293.22,227.89)}, rotate = 167.58] [fill={rgb, 255:red, 96; green, 117; blue, 5 }  ,fill opacity=1 ][line width=0.08]  [draw opacity=0] (8.04,-3.86) -- (0,0) -- (8.04,3.86) -- (5.34,0) -- cycle    ;
\draw  [draw opacity=0][fill={rgb, 255:red, 0; green, 0; blue, 0 }  ,fill opacity=1 ] (130.96,233.2) .. controls (130.96,231.74) and (132.14,230.56) .. (133.6,230.56) .. controls (135.06,230.56) and (136.24,231.74) .. (136.24,233.2) .. controls (136.24,234.66) and (135.06,235.84) .. (133.6,235.84) .. controls (132.14,235.84) and (130.96,234.66) .. (130.96,233.2) -- cycle ;
\draw [color={rgb, 255:red, 1; green, 167; blue, 179 }  ,draw opacity=1 ]   (353.6,263.2) .. controls (335,231.8) and (350.2,137.8) .. (343.6,103.2) ;
\draw [shift={(344.18,180.59)}, rotate = 91.22] [fill={rgb, 255:red, 1; green, 167; blue, 179 }  ,fill opacity=1 ][line width=0.08]  [draw opacity=0] (8.04,-3.86) -- (0,0) -- (8.04,3.86) -- (5.34,0) -- cycle    ;
\draw  [draw opacity=0][fill={rgb, 255:red, 0; green, 0; blue, 0 }  ,fill opacity=1 ] (350.8,262.9) .. controls (350.8,261.44) and (351.98,260.26) .. (353.44,260.26) .. controls (354.9,260.26) and (356.09,261.44) .. (356.09,262.9) .. controls (356.09,264.36) and (354.9,265.54) .. (353.44,265.54) .. controls (351.98,265.54) and (350.8,264.36) .. (350.8,262.9) -- cycle ;
\draw [color={rgb, 255:red, 96; green, 117; blue, 5 }  ,draw opacity=0.65 ]   (277.91,120.54) .. controls (458.77,145.69) and (486.49,142.83) .. (469.91,-0.31) ;
\draw [shift={(432.82,129.09)}, rotate = 167.58] [fill={rgb, 255:red, 96; green, 117; blue, 5 }  ,fill opacity=0.65 ][line width=0.08]  [draw opacity=0] (8.04,-3.86) -- (0,0) -- (8.04,3.86) -- (5.34,0) -- cycle    ;
\draw  [draw opacity=0][fill={rgb, 255:red, 0; green, 0; blue, 0 }  ,fill opacity=1 ] (135.76,220) .. controls (135.76,218.54) and (136.94,217.36) .. (138.4,217.36) .. controls (139.86,217.36) and (141.04,218.54) .. (141.04,220) .. controls (141.04,221.46) and (139.86,222.64) .. (138.4,222.64) .. controls (136.94,222.64) and (135.76,221.46) .. (135.76,220) -- cycle ;

\draw (330,237.76) node [anchor=south east] [inner sep=0.75pt]  [color={rgb, 255:red, 208; green, 2; blue, 27 }  ,opacity=1 ]  {$\tilde{\gamma }_{i,i'}$};
\draw (208.86,223.17) node [anchor=south] [inner sep=0.75pt]  [color={rgb, 255:red, 96; green, 117; blue, 5 }  ,opacity=1 ]  {$\tilde{I}_{\F}^{\Z}( \tilde z_{3})$};
\draw (193.84,237.46) node [anchor=north] [inner sep=0.75pt]  [color={rgb, 255:red, 7; green, 90; blue, 190 }  ,opacity=1 ]  {$T_{1,i}^{-1}\tilde{I}_{\F}^{\Z}(\tilde z''_{1})$};
\draw (349.64,172.08) node [anchor=west] [inner sep=0.75pt]  [color={rgb, 255:red, 1; green, 167; blue, 179 }  ,opacity=1 ]  {$\tilde{I}_{\F}^{\Z}( \tilde z''_{2})$};
\draw (128.96,240) node [anchor=east] [inner sep=0.75pt]    {$T_{1,i}^{-1} \tilde z''_{1}$};
\draw (353.6,266.6) node [anchor=north] [inner sep=0.75pt]    {$\tilde z''_{2}$};
\draw (308.93,123.34) node [anchor=north east] [inner sep=0.75pt]  [color={rgb, 255:red, 96; green, 117; blue, 5 }  ,opacity=0.8 ]  {$T_{1,i}^{-1} T_{2,i'}^{-1}\tilde{I}_{\F}^{\Z}(\tilde z_{3})$};
\draw (134.16,209.6) node [anchor=east] [inner sep=0.75pt]    {$\tilde z_{3}$};

\end{tikzpicture}

\caption{The configuration of the proof of Corollary~\ref{th:rotationalhorseshoe2}. The orange lines are leaves.
\label{Fig:rotationalhorseshoe2}}
\end{center}
\end{figure}
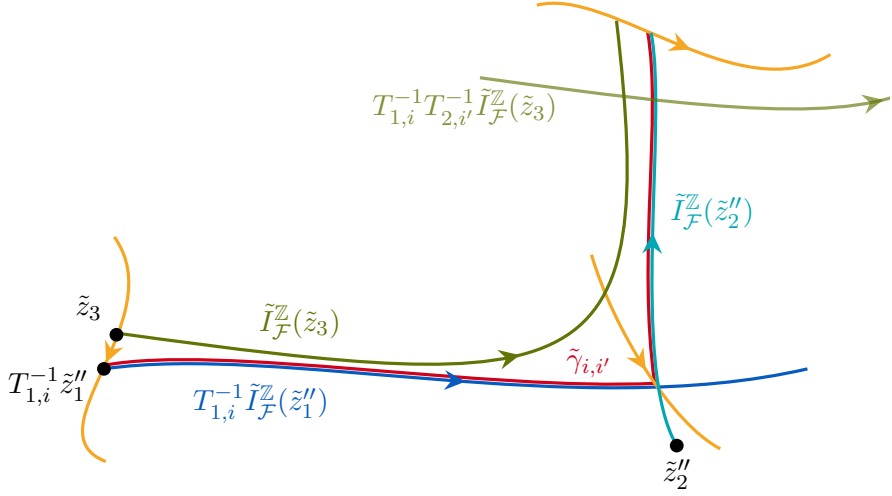

Lemma~\ref{le:forcing} then implies that for any $i,i'$, there exists $\tilde z_3\in \widetilde{\mathrm{dom}}(I)$ such that $\tilde I^{2r'+2+n_{1,i}+n_{2,i'}}_{\tilde {\F}}(\tilde z_3)$ is equivalent to the path 
\[\tilde\gamma_{i,i'} = T_{1,i}^{-1}{\tilde I}^{r'+1+n_{1,i}}_{\tilde {\F}}(\tilde z''_1)|_{[0,n_{1,i}-1+a]}
\cdot 
{\tilde I}^{r'+1+n_{2,i'}}_{\tilde {\F}}(\tilde z''_2)|_{[b,r'+1+n_{2,i'}]}.\]
Consider the parameter $e\in[0, 2r'+3+n_{1,i}+n_{2,i'}]$ such that
$$\tilde I^{\Z}_{\tilde {\F}}(\tilde z_3)(e)= T_{1,i}^{-1}{\tilde I}^{\Z}_{\tilde {\F}}(\tilde z''_1)(n_{1,i}-1+a)={\tilde I}^{\Z}_{\tilde {\F}}(\tilde z''_2)(b).$$
Note that if $i,i'$ are large enough, then $n_{1,i}-1+a\ge d$, and $b\le n_{2,i'}-1+c$. It implies that $T_{1,i}\tilde\gamma_{i,i'} $ has an $\tilde\F$-transverse intersection with $T_{2,i'}^{-1}\tilde\gamma_{i,i'}$ at a point $T_{1,i}\tilde\gamma_{i,i'}(e') = T_{2,i'}^{-1}\tilde\gamma_{i,i'}(e'')$, where $e'<e<e''$.  So, $\tilde\gamma_{i,i'}$ has an $\tilde\F$-transverse intersection with $T_{2,i'}T_{1,i}\tilde\gamma_{i,i'}$ at a point $\tilde\gamma_{i,i'}(e'')= T_{2,i'}T_{1,i}\tilde\gamma_{i,i'}(e')$, where $e'<e''$.
By Theorem~\ref{th:rotationalhorseshoe1}, there exists $s\ge 1$ such that $f$ admits a rotational horseshoe of type $([T_{2,i'}T_{1,i}]_\FHL, 2r'+2+n_{1,i}+n_{2,i'})$. If $i$ is large enough ($i'$ being fixed but large enough to ensure that the above properties hold), then we have $[T_{2,i'}T_{1,i}]_\FHL/(2r'+2+n_{1,i}+n_{2,i'})\in \mathcal U$.
\end{proof}

Let us finish this quick introduction to some forcing theory tools by the following theorem of Lellouch's thesis \cite[Théorème C]{Lel}:

\begin{theorem} \label{th:lellouch}
Suppose that $g\geq 2$. If $f\in \mathrm{Homeo}_*(S)$ preserves two Borel probability measures $\mu_1$ and $\mu_2$ such that $\mathrm{rot}_f(\mu_1)\wedge \mathrm{rot}_f(\mu_2)\not=0$, then $f$ has a topological horseshoe. In particular, $f$ has infinitely many periodic points.

Moreover, if $\mu_1$ is ergodic, then these periodic points can be supposed to have rotation vectors arbitrarily close to $\mathrm{rot}_f(\mu_1)$ and with arbitrarily large period: for every neighbourhood $\mathcal U$ of $\mathrm{rot}_f(\mu_1)$ in $H_1(S,\R)$, there exists a rotational horseshoe of type $(\kappa,r)$ with $[\kappa]/r\in \mathcal U$.
\end{theorem}

Here $\wedge$ is the {\it intersection form}. It is the symplectic form on $H_1(S, \R)$ defined by the property that if $\Gamma_1$ and $\Gamma_2$ are two loops in $S$, then $[\Gamma_1]\wedge[\Gamma_2]$ is the algebraic intersection number between $ \Gamma_1$ and $\Gamma_2$. Equivalently, up to a multiplicative constant, it is the form induced \emph{via} Poincaré duality by $\wedge : H^1(S,\R) \times H^1(S,\R) \to H^2(S,\R)$.

\subsection{Forcing theory in the annular covering space}\label{sec:Dyna}

We suppose now that $\Sigma$ is an oriented closed surface and denote it $S$. We keep the other notations. We consider $ T\in{\mathcal G}\setminus\mathrm {Id}$ and a $T$-strip $\tilde B\subset\widetilde{ \mathrm{dom}}(I)$ (we suppose that $T$ coincides with the identity on the connected components of $\mathrm{dom}(I)$ that do not contain $\tilde B$). We fix a $T$-invariant line $\tilde\gamma_*\subset \tilde B$. We define
\begin{itemize}
\item the surface $\widehat{\mathrm{dom}}(I)=\widetilde{\mathrm{dom}}(I)/T$;
\item the projections $\pi: \widetilde{\mathrm{dom}}(I)\to \widehat{\mathrm{dom}}(I)$ and $\hat\pi: \widehat{\mathrm{dom}}(I)\to  \mathrm{dom}(I)$;
\item the identity isotopy  $\hat I$ on $\widehat{\mathrm{dom}}(I)$ lifted by $\tilde I$;
\item  the lift $\hat f$ of $f|_{ \mathrm{dom}(I)}$  to $\widehat{\mathrm{dom}}(I)$ lifted by $\tilde f$;
\item the foliation $\hat{\F}$ on $\widehat{\mathrm{dom}}(I)$ lifted by $\tilde{\F}$;
\item the loop $\hat\Gamma_*=\pi(\tilde\gamma_*)$.
\end{itemize}
The complement of $\hat\Gamma_*$ in its connected component has two annular connected components $L(\hat\Gamma_*)$ and $R(\hat\Gamma_*)$. We denote $\hat \infty_L$ the common end of $\widehat{\mathrm{dom}}(I)$ and $L(\hat\Gamma_*)$ and $\hat \infty_R$ the common end of $\widehat{\mathrm{dom}}(I)$ and $R(\hat\Gamma_*)$. 

We consider 
\begin{itemize}
\item the set $\tilde W^{R\to L}$ of points $\tilde z\in\widetilde{\mathrm{dom}}(I)$ such that  $\tilde I_{\tilde {\F}}^{\Z}(\tilde z)$ crosses $\tilde B$ from the right to the left;
\item the set $\tilde W^{L\to R}$ of points $\tilde z\in\widetilde{\mathrm{dom}}(I)$ such that  $\tilde I_{\tilde {\F}}^{\Z}(\tilde z)$ crosses $\tilde B$ from the left to the right;
\item the set $\tilde W^{R\to R}$ of points $\tilde z\in\widetilde{\mathrm{dom}}(I)$ such that  $\tilde I_{\tilde {\F}}^{\Z}(\tilde z)$ visits $\tilde B$ on the right;
\item the set $\tilde W^{L\to L}$ of points $\tilde z\in\widetilde{\mathrm{dom}}(I)$ such that  $\tilde I_{\tilde {\F}}^{\Z}(\tilde z)$ visits $\tilde B$ on the left;
\item the set $\tilde W^D$ of points $\tilde z\in\widetilde{\mathrm{dom}}(I)$ such that  $\tilde I_{\tilde {\F}}^{\Z}(\tilde z)$ draws $\tilde B$.
\end{itemize}
Note that all these sets are invariant by $\tilde f$ and by $T$. Note also that they are open, as a consequence of Proposition \ref{prop:stability}. 

We define the respective projections in $\widehat{ \mathrm{dom}}(I)$
$$\hat W^{R\to L}\,,\enskip \hat W^{L\to R}\,, \enskip \hat W^{R\to R}, \enskip \hat W^{L\to L},\enskip \hat W^{D},$$ 
that are open and invariant by $\hat f$ and the respective projections in $\mathrm{dom}(I)$
$$W^{R\to L}\,,\enskip W^{L\to R}\,, \enskip  W^{R\to R}, \enskip W^{L\to L},\enskip  W^{D},$$ 
that are open and invariant by $f$.

Finally, we define
\begin{itemize}
\item the set $\hat \infty_R\to\hat \infty_L$ of points  $\hat z\in\widehat{\mathrm{dom}}(I)$ such that $$\lim_{k\to -\infty}\hat f^k(\hat z)= \hat \infty_R\,, \enskip \lim_{k\to +\infty}\hat f^k(\hat z)= \hat \infty_L;$$
\item the set $\hat \infty_L\to\hat \infty_R$ of points  $\hat z\in\widehat{\mathrm{dom}}(I)$ such that $$\lim_{k\to -\infty}\hat f^k(\hat z)= \hat \infty_L\,,\enskip\lim_{k\to +\infty}\hat f^k(\hat z)= \hat \infty_R.$$
\end{itemize}

We will state some results that have been proven in \cite{Lel} and will add some others that do not explicitely appear there. The following result has been proved in \cite{Lel}
(Proposition 2.2.12).

\begin{lemma} \label{le:equivalence}
Suppose that $\nu\in{\mathcal M}(f)$ is ergodic and that $\nu$-almost every point $z$ has a lift $\tilde z\in\widetilde{\mathrm{dom}}(I)$ such that $\tilde I^{\Z}_{\tilde{\F}}(\tilde z)$ is equivalent in $+\infty$ or $-\infty$ to $\tilde\gamma_{*}$. Then there exists $a\geq 0$\footnote{The proof given in \cite{Lel} says that $a\geq 0$ but we will slightly improve it in Lemma~\ref{claim:proportional} to obtain $a>0$.} such that $\mathrm{rot}(\nu)=a[T]$.
\end{lemma}

The next one also has been proved in \cite{Lel} (Lemma 2.2.3 and Proposition 2.2.4).

\begin{lemma} \label{le:travelling}
Suppose that $\nu\in{\mathcal M}(f)$ is ergodic. We have the following:
\begin{enumerate}
\item if $[T]\wedge \mathrm{rot}_f(\nu)>0$, then  $\nu(\hat\pi(\hat \infty_R\to\hat \infty_L))=1$;
\item if $[T]\wedge \mathrm{rot}_f(\nu)<0$, then $\nu(\hat\pi( \hat\infty_L\to\hat \infty_R))=1$.
\end{enumerate}
\end{lemma}

Let us prove now:

\begin{lemma} \label{le:intersecting}
If there exists $\mu\in{\mathcal M}(f)$ with total support such that $[T]\ \wedge \mathrm{rot}_f(\mu)=0$, then every essential simple loop of $\widehat{\mathrm{dom}}(I)$ meets its image by $\hat f$.
\end{lemma}

\begin{proof}
Suppose that there exists an essential simple loop $\hat \Gamma$ such that $\hat f(\hat\Gamma)\cap \hat \Gamma=\emptyset$. Orient $\hat\Gamma$ in such a way that $\hat \infty_L$ is the common end of $\widehat{\mathrm{dom}}(I)$ and $L(\hat\Gamma)$ and $\hat \infty_R$ the common end of $\widehat{\mathrm{dom}}(I)$ and $R(\hat\Gamma)$. There is no loss of generality by supposing that $\hat f(\hat \Gamma)$ is included in  $L(\hat\Gamma)$. Consider the line $\tilde \gamma$ of $\tilde S$ that lifts $\hat\Gamma$. We have $\tilde f(\overline{L(\tilde\gamma))}\subset L(\tilde\gamma)$ and more generally $\tilde f(\overline{L(T'(\tilde\gamma)))}\subset L(T'(\tilde\gamma))$ for every $T'\in \mathcal G$ because $\tilde f$ commutes with $T'$.

If $\tilde \gamma'$ is an oriented line of $\widetilde{\mathrm{dom}}(I)$, recall that $\widetilde{\mathrm{dom}}(I)_{\tilde\gamma'}$ is the connected component of $\widetilde{\mathrm{dom}}(I)$ that contains $\tilde \gamma'$. Denote $\eta_{\tilde \gamma'}$ the function defined on $\widetilde{\mathrm{dom}}(I)_{\tilde\gamma'}$ that is equal to $0$ on $R(\tilde\gamma')$, to $1$ on $L(\tilde\gamma')$ and to $1/2$ on $\tilde\gamma'$. Noting that $T''(\tilde \gamma) =T'(\tilde \gamma)$ if $T'' {}^{-1}T'\in\langle T\rangle$, one deduces that the notation $\tau\tilde \gamma$ has a sense for every left coset  $\tau\in \mathcal G/\langle T\rangle$. Furthermore, if $\nu\in\mathcal M(f)$ is ergodic, then for $\nu$-almost every point $z$, the following holds for every lift $\tilde z$ of $z$:
$$ [T]\ \wedge \mathrm{rot}_f(\nu)=\lim_{n\to +\infty } \frac1n \sum_{\tau\in \mathcal G/\langle T\rangle} \Big( \eta_{\tau\tilde\gamma} (\tilde f^n(\tilde z))-  \eta_{\tau\tilde\gamma} ( \tilde z)\Big).$$

Indeed, if one considers the loop $\Gamma=\hat\pi(\hat\Gamma)$ of $S$, then $\sum_{\tau\in \mathcal G/\langle T\rangle} \eta_{\tau\tilde\gamma} (\tilde f^n(\tilde z))-  \eta_{\tau(\tilde\gamma)} ( \tilde z)$ (note that the sum is finite) is equal to the sum of the algebraic intersection numbers between all lifts of $\Gamma$  with the trajectory $\tilde I^n(\tilde z) $ (at least when $z$ and $f^n(z)$ are not on $\Gamma)$, meaning the algebraic intersection number between $\Gamma$ and $I^n(z)$.

Observe that for every $\tau\in \mathcal G/\langle T\rangle$, the function $\eta_{\tau\tilde\gamma}\circ \tilde f-\eta_{\tau\tilde\gamma} $ is non negative on $\widetilde{\mathrm{dom}}(I)_{\tilde\gamma}$ and positive in the strip between $\tilde\gamma$ and $\tilde f(\tilde\gamma)$.  We deduce that for every ergodic invariant probability measure $\nu$ it holds that $[T]\ \wedge \mathrm{rot}_f(\nu)\geq 0$. Moreover, we have a strict inequality if the measure of the strip between $\tilde\gamma$ and $\tilde f(\tilde\gamma)$ is non zero for the measure $\tilde\nu$ that lifts $\nu$. By using the ergodic decomposition of $\mu$, we deduce that $[T]\ \wedge \mathrm{rot}_f(\mu)> 0$, which contradicts the hypothesis.
\end{proof}

\begin{lemma}\label{le:signofintersection}
Suppose that $\nu\in{\mathcal M}(f)$ and $\nu'\in{\mathcal M}(f)$ are ergodic and satisfy 
$$\nu(W^{R\to L}\cap W^D)=1\,,\enskip [T]\, \wedge \mathrm{rot}_f(\nu')<0.$$ 
Then one of the following assertions holds:

\begin{itemize} 
\item for $\nu$-almost every point $z$ and $\nu'$-almost every point $z'$, the paths $I_{\F}^{\Z}(z)$ and $I_{\F}^{\Z}(z')$ have an $\F$-transverse intersection;
\item for $\nu$-almost every point $z$ and $\nu'$-almost every point $z'$, the path $I_{\F}^{\Z}(z')$ accumulates on $I_{\F}^{\Z}(z)$.
\end{itemize}
\end{lemma}

\begin{proof}
Define three $f$-invariant sets $W_1$, $W_2$, $W_3$ as follows:
\begin{itemize}
\item $z'\in W_1$ if it has a lift $\tilde z'$ such that $\tilde I_{\tilde{\F}}^{\Z}(\tilde z')$ is equivalent to $\tilde \gamma_*$ at $+\infty$ or at $-\infty$;
\item $z'\in W_2$ if it has a lift $\tilde z'$ such that $\tilde I_{\tilde{\F}}^{\Z}(\tilde z')$ accumulates on $\tilde\gamma_*$ positively or negatively;
\item $z'\in W_3$ if it has a lift $\tilde z'$ such that $\tilde I_{\tilde{\F}}^{\Z}(\tilde z')$ crosses $\tilde B$ from the left to the right.
\end{itemize}
By Lemma~\ref{le:travelling}, we know that $\nu'$-almost every point $z'$ has a lift $\hat z'\in \widehat{\mathrm{dom}}(I)$ that belongs to $\hat \infty_L\to\hat \infty_R$. Consequently $\nu' (W_1\cup W_2\cup W_3) =1$, which implies by ergodicity of $\nu'$ that one of the sets $W_1$, $W_2$, $W_3$ has $\nu'$-measure $1$. 
By Lemma~\ref{le:equivalence}, $\nu'(W_1)\not=1$ because $\mathrm{rot}_f(\nu')\notin \R [T]$ (by the hypothesis $[T]\, \wedge \mathrm{rot}_f(\nu')<0$). If $\nu'(W_2)=1$, then the second item of the lemma holds because for every leaf $\tilde \phi \subset \tilde B$,  $\nu$-almost every point $z$ belongs to $W_D$ and so has a lift $\tilde z\in \widetilde{\mathrm{dom}}(I)$ such that $\tilde I_{\tilde{\F}}^{\Z}(\tilde z)$ meets $\tilde \phi$. By Proposition~\ref{prop:transversestrips}, if $\nu'(W_3)=1$, then the first item of the lemma holds. 
\end{proof}

\begin{corollary}\label{co:signofintersection}
Suppose that $\nu\in{\mathcal M}(f)$ is ergodic and satisfies
$$\nu(W^{R\to L}\cap W^D)=1\,,\quad [T] \wedge \mathrm{rot}_f(\nu)<0.$$ 
Then, for $\nu$-almost every point $z$, the path $I_{\F}^{\Z}(z)$ has an $\F$-transverse self intersection.
\end{corollary}

\begin{proof} 
Let us apply Lemma \ref{le:signofintersection} with $\nu'=\nu$ and use the fact that a recurrent transverse path does not accumulate on itself (Corollary~\ref{coro:recurrence}).
\end{proof}

This result is still true if $\nu(W^{R\to L}\cap W^D)=1$ and $ [T]\ \wedge \mathrm{rot}_f(\nu)=0$. More precisely we have (see \cite{Lel}, Proposition 3.3.1).

\begin{lemma}\label{co:signofintersectionzero}
Suppose that $\nu\in{\mathcal M}(f)$ is ergodic and satisfies $$\nu(W^{R\to L}\cap W^D)=1\,,\enskip [T]\ \wedge \mathrm{rot}_f(\nu)=0.$$ 
Then $\nu(W^{L\to R})=1$ and for $\nu$-almost every point $z$, the path $I_{\F}^{\Z}(z)$ has an $\F$-transverse self intersection.
\end{lemma}

\begin{remark*} 
The conclusion $\nu(W^{L\to R})= 1$ is not explicitely stated in  \cite{Lel}, Proposition 3.3.1. But, as explained by the author at the beginning of the proof, it is the key point that permits to get the second conclusion. The first condition says that there are points ``that go up'', which implies by the second condition, that there are points ``that go down''. We have a situation very similar to the one that occurs under the hypothesis of Corollary \ref{co:signofintersection}, but more subtle arguments of ergodic theory are needed. 
\end{remark*}

\begin{lemma} \label{le:nointersection}
Suppose that there exist $\lambda\in{\mathcal M}(f)$ such that $\mathrm{supp}(\lambda)=S$. 
If $\nu\in{\mathcal M}(f)$ is ergodic and satisfies $$\nu(W^{R\to L}\cap W^D)=1\,, \enskip[T]\wedge\mathrm{rot}_f(\nu) >0,$$ then there exists $\nu'\in{\mathcal M}(f)$ ergodic, such that one of the following assertions holds:
\begin{itemize} 
\item for $\nu$-almost every point $z$ and $\nu'$-almost every point $z'$, the paths $I_{\F}^{\Z}(z)$ and $I_{\F}^{\Z}(z')$ have an $\F$-transverse intersection;
\item for $\nu$-almost every point $z$ and $\nu'$-almost every point $z'$, the path $I_{\F}^{\Z}(z')$ accumulates on $I_{\F}^{\Z}(z)$.
\end{itemize}
\end{lemma}

\begin{proof} 
By hypothesis $W^{R\to L}\cap W^D$ is a non empty invariant open set and so we have
$$\lambda(W^{R\to L}\cap W^D)>0.$$

Suppose first that $[T]\wedge\mathrm{rot}_f(\lambda_{W^{R\to L}\cap W^D}) \leq 0$. Using the ergodic decomposition of $\lambda_{W^{R\to L}\cap W^D}$, we deduce that there exists $\nu'\in{\mathcal M}(f)$ ergodic such that $\nu'(W^{R\to L}\cap W^D)=1$ and $[T]\wedge\mathrm{rot}_f(\nu') \leq 0$.  If  $[T]\wedge\mathrm{rot}_f(\nu')<0$, we can apply Lemma \ref{le:signofintersection} and so the conclusion of Lemma \ref{le:nointersection} holds. If  $[T]\wedge\mathrm{rot}_f(\nu')=0$ we know that  $\nu'(W^{L\to R})= 1$ by Lemma \ref{co:signofintersectionzero} and so the first item of the conclusion of Lemma \ref{le:nointersection} holds thanks to Proposition \ref{prop:transversestrips}.

Suppose now that $[T]\wedge\mathrm{rot}_f(\lambda_{W^{R\to L}\cap W^D}) > 0$. From the equalities 
$$[T]\wedge\mathrm{rot}_f(\lambda)=0$$
and
$$ \mathrm{rot}_f(\lambda_{\mathrm{fix}(I)})=0\enskip\mathrm{ if} \enskip\lambda(\mathrm{fix}(I))\not=0,$$ we deduce that 
$$\lambda\left(\mathrm{dom}(I)\setminus (W^{R\to L}\cap W^D)\right)>0$$ and  
$$[T]\wedge\mathrm{rot}_f(\lambda_{\mathrm{dom}(I)\setminus (W^{R\to L}\cap W^D)})<0.$$
Using the ergodic decomposition of $\lambda_{\mathrm{dom}(I)\setminus (W^{R\to L}\cap W^D)}$, we deduce that there exists $\nu'\in{\mathcal M}(f)$ such that $[T]\wedge\mathrm{rot}_f(\nu') < 0$. Here again we refer to Lemma \ref{le:signofintersection} to ensure that the conclusion of Lemma \ref{le:nointersection} holds.
\end{proof}

Let us conclude this section with a new result that will be useful for our purpose.

\begin{proposition} \label{prop:conditionhorsehoestrip}  
Suppose that $\nu\in{\mathcal M}(f)$ and $\nu'\in{\mathcal M}(f)$ are ergodic and that for $\nu$-almost every point $z\in {\mathrm{dom}}(I) $ and $\nu'$-almost every point $z'\in {\mathrm{dom}}(I) $, the path $I_{\F}^{\Z}(z')$ accumulates on  $I_{\F}^{\Z}(z)$, then $\mathrm{rot}_{f}(\nu)\wedge \mathrm{rot}_{f}(\nu')\neq 0$.
\end{proposition}

\begin{proof}
There is no loss of generality by supposing that $I_{\F}^{\Z}(z')$ accumulates positively on  $I_{\F}^{\Z}(z)$. By Proposition
\ref{prop: nontransitivity3}, there exists a transverse simple loop $\Gamma_*\subset\Sigma$ such that 
\begin{itemize}
\item $I_{\F}^{\Z}(z')$ is equivalent to the natural lift of $\Gamma_*$;
\item  the union $B$ of leaves met by $\Gamma_*$ is an open annulus of $S$;
\item if $\tilde\gamma_*$ is a lift  of $\Gamma_*$ to $\widetilde{\mathrm{dom}}(I)$, then for $\nu$-almost every point $z\in {\mathrm{dom}}(I) $, there is a lift $\tilde z\in \widetilde{\mathrm{dom}}(I)$ such that $\tilde I_{\F}^{\Z}(\tilde z)$ meets  $\partial \tilde B^L$;
\item  for every $\tilde \phi\subset \partial \tilde B^L$ it holds that  $\tilde B\subset R(\tilde \phi)$.
\end{itemize}
The point $z$ can be chosen recurrent and so every leaf of $\F$ met by $I_{\F}^{\Z}(z)$ is met infinitely many often in the past and in the future. In particular, $I_{\F}^{\Z}(z)$ goes in and out of $B$ infinitely many times, but it never enters in $B$ on the left because $\tilde B\subset R(\tilde \phi)$
for every $\tilde \phi\subset \partial \tilde B^L$. We deduce that every lift  $\tilde z\in \tilde B$ of $z$ crosses $\tilde B$ from the right to the left. So, referring to the notations of the whole section, we have $\nu(W^{R\to L})=1$.

\begin{lemma}\label{claim:proportional} 
Let $\nu'$ be an $f$-invariant ergodic probability measure such that $\nu'(\mathrm{dom}(I))=1$. Suppose that there is some deck transformation $T\in\G\setminus\{\mathrm{Id}\}$ and a $T$-strip $\tilde B$ projecting an an open annulus $B$ of $S$ such that $\nu'$-almost every point $z'\in \dom(I)$ satisfies $I^{\Z}_{\F}(z')\subset B$.
Then there exists $a>0$ such that $\mathrm{rot}_f(\nu')=a[T]$.
\end{lemma}

\begin{proof}
By Lemma \ref{le:equivalence} there exists $a\geq 0$ such that $\mathrm{rot}_f(\nu')=a[T]$. We need to prove that $a\neq 0$. 
Let $U'\subset B$ be a topological open disk such that $\nu'(U')\not=0$. We can suppose that $U'$ is a flow-box that satisfies the conclusion of Proposition \ref{prop:flowbox}. Write $\varphi_U':U'\to U'$ for the first return map of $f$ and $\tau_U' :U'\to \N\setminus\{0\}$ for  the time of first return map, which are defined $\nu'$-almost everywhere on $U'$. Note that $\nu'\vert_{U'}$ is an ergodic invariant measure for $\varphi_U'$. 
Fix a lift $\tilde U'\subset \tilde B$ of $U'$. For every point $z\in U'$ such that $\tau_U'(z)$ exists, denote $\tilde z$ the lift of $z$ that is in $\tilde U'$ and $\delta_{U'}(z)$ the integer such that $\tilde f^{\tau_{U'}(z)}(\tilde z)\in T^{\delta(z)} \tilde U'$. One gets a map $\delta_{U'}:U'\to \Z$ defined $\nu'$-almost everywhere on $U'$. Remind that a map $\rho_{U'}: U'\to H_1(S,\Z)$ has been defined in the introduction and that $\rho_{U'}(z)=\delta(z)[T]$. Note also that $\delta(z)>0$. The measure $\nu'$ being ergodic, by Kac's theorem one knows that 
$$\int_{U'} \tau_{U'} \, d\nu'=\nu'\left(\bigcup_{k\geq 0} f^k({U'})\right)=\nu'\left(\bigcup_{k\in\Z} f^k({U'})\right)=1,$$
and consequently that  $\tau_{U'}^*(z)=1/\nu'({U'})$ for $\nu'$-almost every point $z\in {U'}$, where $\tau_{U'}^*$ and $\rho_{U'}^*$ has been defined in \eqref{eq:def*} (page \pageref{eq:def*}). Furthermore, for $\nu'$-almost every point $z\in {U'}$, it holds that 
$$\mathrm{rot}_f(\nu')=\mathrm{rot}_f(z)=\rho_{U'}{}^*(z)/\tau_{U'}^*(z)=\nu'({U'})\rho_{U'}^*(z)=\left(\int_{U'} \delta(z)\, d\nu'(z)\right) [T].$$
Observe now that $\int_{U'} \delta(z)\, d\nu'(z)>0$. This proves the lemma.
\end{proof}

To prove Proposition \ref{prop:conditionhorsehoestrip} it remains to prove that $\mathrm{rot}_{f}(\nu')\wedge \mathrm{rot}_{f}(\nu)> 0$ which would lead to the result with Theorem~\ref{th:lellouch}. Let $U\subset B$ be a topological open disk such that $\nu(U)\not=0$ and that is a flow-box that satisfies the conclusion of Proposition \ref{prop:flowbox}. Perturbing $\Gamma_*$ and reducing $U$ if necessary, one can suppose that $U\cap \Gamma_*=\emptyset$. Write $\varphi_U:U\to U$ for the first return map of $f$ and $\tau_U :U\to \N\setminus\{0\}$ for  the time of first return map, which are defined $\nu$-almost everywhere on $U$. We will define a function $\delta_U:U\to \Z$ in a different way. For every point $z\in U$ such that $\tau_U(z)$ exists, set $m=\tau_U(z)$ and consider the set
$$X_z=\big\{t\in[0, m]\,\vert\, I_{\F}^{t}(z)\subset U\big\}.$$
Suppose first that $X_z\not=[0,\tau_U(z)]$. Then denote $(J_{\xi})_{\xi\in\Xi}$ the family of connected components of $X_z$. One component $J_{\xi^-}$ can be written $J_{\xi^-}=[0, b_{\xi^-})$, one component $J_{\xi^+}$ can be written $J_{\xi^-}=(a_{\xi^+}, m]$ and the remaining components can be written  $J_{\xi}=(a_{\xi}, b_{\xi})$. 
Consider such a component $J_{\xi}$. The path $I_{\F}^{m}(z)$ can be lifted to a path $\tilde I_{\tilde{\F}}^{m}(\tilde z)$ (the lift depending on $\xi$) such that $\tilde I_{\tilde{\F}}^{n}(\tilde z)((a_{\xi}, b_{\xi}))\subset \tilde B$.  
By assumptions, one knows that $\tilde I_{\tilde{\F}}^{m}(\tilde z)(a_{\xi})\in \partial \tilde B^R$ and we set 
$$\delta_{\xi}=\begin{cases} 
0 & \mathrm{if}\enskip \tilde I_{\tilde{\F}}^{m}(\tilde z)(b_{\xi})\in \partial \tilde B^R,\\  
1 & \mathrm{if}\enskip \tilde I_{\tilde{\F}}^{m}(\tilde z)(b_{\xi})\in \partial \tilde B^L. 
\end{cases}$$
In the first situation $\tilde I_{\tilde{\F}}^{m}(\tilde z)|_{[a_{\xi}, b_{\xi}]}$ visits $\tilde B$ on the right, in the second one it crosses $\tilde B$ for the right to the left. 
Note that there are finitely many $\xi\in\Xi$ such that $\delta_{\xi}=1$ because there are finitely many $\xi\in\Xi$ such that $I_{{\F}}^{m}(z)([a_{\xi}, b_{\xi}])\cap \Gamma_*\not=\emptyset$. Indeed, $\tilde \gamma_*$ is contained in $\tilde B$, while each such $I_{{\F}}^{m}(z)([a_{\xi}, b_{\xi}])$ meets $\partial \tilde B$; the conclusion follows by a compactness argument.

The path $I_{\F}^{m}(z)$ can be lifted to a path $\tilde I_{\tilde{\F}}^{m}(\tilde z)$ such that $\tilde I_{\tilde{\F}}^{n}(\tilde z)([0, b_{\xi_-}))\subset \tilde B$. Set 
$$\delta_{\xi_-}=\begin{cases} 
1/2 & \mathrm{if}\enskip  \tilde I_{\tilde{\F}}^{m}(\tilde z)(b_{\xi_-})\in \partial B^L,\\  
-1/2  &\mathrm{if}\enskip \tilde I_{\tilde{\F}}^{n}(\tilde z)(b_{\xi_-})\in \partial B^R. 
\end{cases}$$
Finally, set $\delta_{\xi_+}=1/2$.
Observe now that we have 
$$ [\Gamma_*]\wedge\rho_U(z)= \delta_U(z),$$
where $\rho_U$ is defined page~\pageref{eq:def*}, and
$$\delta_{U}(z)=\begin{cases} 
\sum_{\xi\in\Xi}\delta_i & \mathrm{if}\enskip X_z\not=[0,\tau_U(z)],\\  
0 & \mathrm{if}\enskip X_z=[0,\tau_U(z)] .
\end{cases}$$ 
The function $\delta_U$ is non negative but does not vanishes almost $\nu_U$-everywhere because $I^{\Z}_{\mathcal Z}(z)$ does not stay in $B$ for $\nu$-almost every point. So, we have
$$[\Gamma]\wedge\mathrm{rot}_f(\nu)=\nu(U)[\Gamma]\wedge\rho_U^*(z)=\int_U \delta_U(z)\, d\nu(z)>0.$$
\end{proof}

\begin{remark*}
Using Lellouch's techniques \cite[Section 3.4]{Lel}, one can more generally show that if $z$ and $z'$ are recurrent points (not necessarilly trajectories of typical points for ergodic measures) and if  $I_{\F}^{\Z}(z')$ accumulates on  $I_{\F}^{\Z}(z)$, then $f$ has a topological horseshoe\footnote{Be careful, in this case we do not have that $I_{\F}^{\Z}(z')$ and  $I_{\F}^{\Z}(z)$ intersect $\F$-transversally.}. However, we will not use this property in the sequel.
\end{remark*}

\section{Proof of the main theorem}\label{s.beginningproof}

We suppose in this section that the hypotheses of Theorem \ref{th:main} are satisfied. 
We consider an oriented closed surface $S$ of genus $g\geq 2$ and a homeomorphism $f$ of $S$ isotopic to the identity that preserves a Borel probability measure $\lambda$ with total support such that $\mathrm{rot}_f(\lambda)=s\rho$, with $\rho\in H_1(S,\Z)\setminus\{0\}$ and $s\in\R$. We keep the notations of the article. 
We consider a Borel probability measure $\nu$, invariant by $f$ and ergodic. We consider a neighborhood $\mathcal U$ of $\mathrm{rot}_f(\nu)$ in $H_1(S,\R)$ and want to prove that there exists a homotopical interval of rotation $(\kappa,r)$ such that  $[\kappa]/r\in \mathcal U$.

There is no loss of generality by supposing that $f$ is not the identity map; in this case one can consider a maximal isotopy $I$ of $f$ by Theorem \ref{th:BCL} with non empty domain. By Theorem \ref{th.transversefoliation}, one can find a non singular foliation $\F$ on $\mathrm{dom}(I)$ transverse to $I$. Remind that: 

\begin{itemize}
\item $\widetilde{\mathrm{dom}}(I)$ is the universal covering space of $\mathrm{dom}(I)$;
\item $\widetilde{\mathrm{dom}}(I)_X$ is the connected component of $\widetilde{\mathrm{dom}}(I)$ that contains a given connected set $X\subset\widetilde{\mathrm{dom}}(I)$;
\item $\tilde \pi : \widetilde{\mathrm{dom}}(I)\to\mathrm{dom}(I)$ is the covering projection;
\item $\mathcal G$ is the group of covering automorphism of $\tilde\pi$;
\item $[T]\in H_1(S,\Z)$ is the homology class of a loop $\Gamma\subset \mathrm{dom}(I)$ associated to $T\in\mathcal G$;
\item $\tilde I$ is the lift of $I|_{\mathrm{dom}(I)}$ to $\widetilde{\mathrm{dom}}(I)$ that starts from the identity;
\item $\tilde f$ is the lift of $f|_{ \mathrm{dom}(I)}$  to $\widetilde{\mathrm{dom}}(I)$ that is the end point of $\tilde I$;
\item $\tilde{\F}$ is the lift of $\F$ to $\widetilde{\mathrm{dom}}(I)$;
\item $I_{\F}^{\Z}(z)$ is the whole $\F$-transverse trajectory of a point $z\in \mathrm{dom}(I)$;
\item $\tilde I_{\tilde{\F}}^{\Z}( \tilde z)$ is the whole $\tilde {\F}$-transverse trajectory of a point $\tilde z\in \widetilde{\mathrm{dom}}(I)$. 
\end{itemize}

Suppose first that $\mathrm{rot}_f(\nu)\wedge\mathrm{rot}_f(\lambda)\not =0$. Using the  ergodic decomposition of $\lambda$, we deduce that there exists $\nu'\in {\mathcal M}(f)$ ergodic such that $\mathrm{rot}_f(\nu)\wedge\mathrm{rot}_f(\nu')\not =0$. By Theorem~\ref{th:lellouch}, we know that $f\vert_{\mathrm{dom}(I)}$ has a rotational topological horseshoe of type $(\kappa,r)$ with $[\kappa]/r\in \mathcal U$. If $\Gamma\subset \mathrm{dom}(I)$ is a loop associated to $T$, then for every $p/q\in[0,1]$ written in an irreducible way, there exists a periodic point $z\in\mathrm{dom}(I)$ of period $rq$ such that $I^{rq}(z)$ is freely homotopic to $[\Gamma]^p$ in  $\mathrm{dom}(I)$: it is freely homotopic to $[\Gamma]^p$ in $S$. Hence, $f$ has a homotopical interval of rotation of type $(\kappa, r)$ such that $[\kappa]/r\in \mathcal U$, and the conclusion of Theorem \ref{th:main} holds. 

It remains to study the case where  $\mathrm{rot}_f(\nu)\wedge \mathrm{rot}_f(\lambda)=0$. 

\begin{lemma} \label{le:goodstrip}  
Suppose that $\mathrm{rot}_f(\nu)\wedge \mathrm{rot}_f(\lambda)=0$. There exists $T\in \G\setminus\{\mathrm {Id}\}$ satisfying $[T]\wedge \mathrm{rot}_f(\lambda)=0$ and a $T$-strip $\tilde B$ such that $\nu$-almost every point $z\in \mathrm{dom}(I)$ has a lift  $\tilde z$ such that $\tilde I_{\tilde{\F}}^{\Z}(\tilde z)$ draws $\tilde B$.
Moreover if $\mathcal U$ is a neighborhood of $\mathrm{rot}_f(\nu)$, one can suppose that there exists $r\geq1$ such that $[T]/r\in \mathcal U$.
\end{lemma}

\begin{proof} 
Fix $z_0\in\mathrm{supp}(\nu)\cap \mathrm{dom}(I)$ and a lift $\tilde z_0\in  \widetilde{\mathrm{dom}}(I)$ of $z_0$. One can find a topological open disk $U\subset \mathrm{dom}(I)$ containing $z_0$ such that the connected component $\tilde U$ of $\tilde\pi^{-1}(U)$ containing $\tilde z_0$ is a flow-box that satisfies the conclusion of Proposition \ref{prop:flowbox}. Write $\varphi_U:U\to U$ for the first return map of $f$ and $\tau_U :U\to \N\setminus\{0\}$ for  the time of first return map, which are defined $\nu$-almost everywhere on $U$. Note that $\nu\vert_{U}$ is an ergodic invariant measure of $\varphi_U$. 
Remind that a map $\rho_U: U\to H_1(S,\Z)$ has been defined in the introduction. For every point $z\in U$ such that $\tau_U(z)$ exists, denote $\tilde z$ the preimage of $z$ by $\tilde\pi$ that is in $\tilde U$ and $\delta_U(z)$ the automorphism such that $\tilde f^{\tau_U(z)}(\tilde z)\in \delta_U(z)(\tilde U)$. One gets a map $\delta_U:U\to \mathcal G$ defined $\nu$-almost everywhere on $U$ such that $\rho_U(z)=[\delta_U(z)]$. The measure $\nu$ being ergodic, one knows that $$\int_U \tau_U \, d\nu=\nu\left(\bigcup_{k\geq 0} f^k(U)\right)=\nu\left(\bigcup_{k\in\Z} f^k(U)\right)=1,$$
and consequently that  $\tau_U{}^*(z)=1/\nu(U)$ for $\nu$-almost every point $z\in U$, where $\tau_U{}^*$ and $\rho_U{}^*$ has been defined in \eqref{eq:def*} (page \pageref{eq:def*}). Furthermore, for $\nu$-almost every point $z\in U$, it holds that 
$$\int_U\rho_U (z)\, d\nu(z)=\nu(U)\rho_U{}^*(z)=\mathrm{rot}_f(z)=\mathrm{rot}_f(\nu),$$
which implies that
$$\int_U\rho_U (z)\wedge \mathrm{rot}_f(\lambda)\, d\mu(z)=\nu(U)\rho_U{}^*(z)\wedge \mathrm{rot}_f(\lambda)=\mathrm{rot}_f(\nu)\wedge \mathrm{rot}_f(\lambda)=0.$$
By Atkinson's  Theorem \cite{At}, one knows that if $\varepsilon>0$ is fixed, then for $\nu|_{ U}$ almost every point $z$, there exists $n\geq 1$ such that
$$\left\vert\sum_{k=0}^{n-1} \rho_U(\varphi_U{}^k(z))\wedge  \mathrm{rot}_f(\lambda)\right \vert<\varepsilon.$$
As observed by Lellouch \cite{Lel}, we can slightly improve this result: for $\nu|_{ U}$ almost every point $z$, it holds that
$$\liminf_{n\to+\infty} \,\left\vert\sum_{k=0}^{n-1} \rho_U(\varphi_U{}^k(z))\wedge  \mathrm{rot}_f(\lambda)\right \vert=0.$$
So, if we fix a norm $\Vert\enskip\Vert$ on $H_1(S,\R)$ and $\eta>0$, we can find $z_1\in \mathrm{supp}(\mu)\cap U$ and $n\geq 1$ such that (recall that $\mathrm{rot}_f(\lambda)=s\rho$, with $\rho\in H_1(S,\Z)\setminus\{0\}$ and $s\in\R$)
$$ \left\vert\sum_{k=0}^{n-1}  \rho_U(\varphi_U{}^k(z_1))\wedge  \mathrm{rot}_f(\lambda)\right\vert<s,$$
and such that
$$\left\Vert {1\over n} \sum_{k=0}^{n-1} \rho_U(\varphi_U{}^k(z_1))-\mathrm{rot}_f(\nu)\right\Vert < \eta.$$ 

Every number  $\rho_U(\varphi_U{}^k(z_1))\wedge  \mathrm{rot}_f(\lambda)$ belonging to $s\Z$ we deduce that
$$ \sum_{k=0}^{n-1}  \rho_U(\varphi_U{}^k(z_1))\wedge  \mathrm{rot}_f(\lambda)=0.$$ 
Set 
$$r=\sum_{0\leq k<n} \tau_U({\varphi_U{}^k(z_1)})$$ and denote $\tilde z_1$ the lift of $z_1$ that belongs to $\tilde U$. The automorphism $T$ such that $\tilde f^{r}(\tilde z_1)\in T(\tilde U)$ can be written
$$T=T_{n-1}\circ \dots \circ T_1,$$ where $T_k$ is an automorphism conjugated to $ \delta_U(\varphi_U{}^{k}(z_1))$, so we have
$$[T]=\sum_{0\leq k<n} \big[\delta_U(\varphi_U{}^{k}(z_1))\big].$$ 
Consequently, it holds that 
$$ [T]\wedge \mathrm{rot}_f(\lambda)=0\, , \qquad \big\Vert[T]/r -\mathrm{rot}_f(\nu)\big\Vert<\eta.$$

Note that we have $\tilde f^{r}(\tilde z_1)\in T(\tilde U)$ if $\tilde z_1$ is the lift of $z_1$ that belongs to $\tilde U$. The property of $\tilde U$ stated in Proposition \ref{prop:flowbox} tells us that $\tilde I_{\tilde{\F}}^{\Z}(\tilde z_1)$ intersects every leaf that meets $\tilde U$ and every leaf that meets $T(\tilde U)$. So, there is subpath $\tilde\gamma_1$ of $\tilde I_{\tilde{\F}}^{\Z}(\tilde z_1)$ that joins $\phi_{\tilde z_1}$ to  $T(\phi_{\tilde z_1})$. 
Of course we have $T\not=\mathrm{Id}$. Moreover $\tilde I_{\tilde{\F}}^{\Z}(\tilde z_1)$ draws the $T$-strip $\tilde B$ defined by the line $\tilde\gamma_*$ obtained by concatenating\footnote{Strictly speaking one has to modify the path $\gamma_1$ lifted by $\tilde\gamma_1$ to be able to concatenate $T^k(\tilde\gamma_1)$ with $T^{k+1}(\tilde\gamma_1)$: it is sufficient to move it along the leaves so that the last endpoint of $\tilde \gamma_1$ with the first endpoint of $T(\tilde\gamma_1)$ coincide.} the paths $T^k(\tilde\gamma_1)$, $k\in\Z$.
As explained before, Proposition \ref{prop:stability} tells us  that the set $W^D$ of points $z\in U$  that have a lift $\tilde z$ such that $\tilde I_{\tilde{\F}}^{\Z}(\tilde z)$ draws $\tilde B$ is open. It is $T$-invariant and contains $z_1\in \mathrm{supp}(\nu)$. The measure $\nu$ being ergodic, it holds that $\nu(W^D)=1$.
\end{proof}

\begin{proof}[Proof of Theorem \ref{th:main}]

Let us summarize in which cases the results we have already proved allow us to get Theorem \ref{th:main}. Recall that the sets $W^*$ are defined in Paragraph~\ref{sec:Dyna}.
\begin{itemize}
\item If $\nu(W^{R\to R}\cap W^D)=1$ or $\nu(W^{L\to L}\cap W^D)=1$, then by Proposition~\ref{prop:transversestrips} for $\nu$-almost every point $z$, the path $I_{\F}^{\Z}(z)$ has an $\F$-transverse self intersection; this allows to apply Corollary~\ref{th:rotationalhorseshoe2} and to get a suitable rotational horseshoe.
\item If $\nu(W^{R\to L}\cap W^D)=1$, there are three cases:
\begin{itemize}
\item If $[T]\wedge\mathrm{rot}_f(\nu)<0$, then one can apply Corollary~\ref{co:signofintersection} which shows that for $\nu$-almost every point $z$, the path $I_{\F}^{\Z}(z)$ has an $\F$-transverse self intersection; this allows to apply Corollary~\ref{th:rotationalhorseshoe2} and to get a suitable rotational horseshoe.
\item If $[T]\wedge\mathrm{rot}_f(\nu)=0$, then one can apply Lemma~\ref{co:signofintersectionzero} which shows that for $\nu$-almost every point $z$, the path $I_{\F}^{\Z}(z)$ has an $\F$-transverse self intersection; as before this allows to apply Corollary~\ref{th:rotationalhorseshoe2} and to get a suitable rotational horseshoe.
\item If $[T]\wedge\mathrm{rot}_f(\nu)>0$, then one can apply Lemma~\ref{le:nointersection}. It tells us that there exists an ergodic invariant probability measure $\nu'$ such that for $\nu$-almost every point $z$ and $\nu'$-almost every point $z'$, either the paths $I_{\F}^{\Z}(z)$ and $I_{\F}^{\Z}(z')$ have an $\F$-transverse intersection, or the path $I_{\F}^{\Z}(z')$ accumulates on $I_{\F}^{\Z}(z)$. In the first case one can apply Corollary~\ref{th:rotationalhorseshoe2} to get a suitable rotational horseshoe. In the second case Proposition~\ref{prop:conditionhorsehoestrip} tells us that $\mathrm{rot}_{f}(\nu)\wedge \mathrm{rot}_{f}(\nu')\neq 0$. Lellouch's Theorem~\ref{th:lellouch} then gives us a suitable rotational horseshoe.
\end{itemize}
\item The case $\nu(W^{L\to R}\cap W^D)=1$ is identical to the case $\nu(W^{R\to L}\cap W^D)=1$.
\end{itemize}

In all these cases the existence of a suitable homotopical interval of rotation is due to the presence of a rotational topological horseshoe. To get Theorem \ref{th:main} it remains to study a last case where the existence of a suitable homotopical interval of rotation will have another reason. One can write $T=T'{}^{m}$, $m\geq 1$, where $T'\in\mathcal G$ is irreductible. The following proposition will permit us to finish the proof of Theorem \ref{th:main}. 
Indeed, let $\mathcal U$ be a neigborhood of $\mathrm{rot}_f(\nu)$ in $H_1(S,\R)$. One can find $p_0/q_0\in(0,a)$ written in an irreducible way such that $p_0[T']/q_0\in \mathcal U$. By Proposition \ref{prop:existlargeperiod}, for every $p/q\in[0,1]$ written in an irreducible way, there exists $\tilde z_{p/q}$ such that $\tilde f^{qq_0}(\tilde z) =T'{}^{pp_0}(\tilde z)$. The image $z_{p/q}=\tilde\pi(\tilde z_{p/q})\in S$ is fixed by $f^{qq_0}$ and the loop of $S$ defined by  $I^{qq_0}(z_{p/q})$ belongs to $[T']_{\mathcal{FHL}}{}^{pp_0}$. 
Denote $q'=qq_0/s$ the period of $z_{p/q}$. There exists $R\in\mathcal G$ such that $\tilde f^{q'}(\tilde z_{p/q})=R(\tilde z_{p/q})$. We deduce that $T'{}^{pp_0}(\tilde z_{p/q}) =\tilde f^{qq_0}(\tilde z_{p/q})=R^{s}(\tilde z_{p/q})$. It implies that $T'{}^{pp_0}=R^{s}$. The group $\langle T' , R\rangle$ being a free group, it must be infinite cyclic. We deduce that $R$ is a power of $T'$ because $T'$ is irreducible and so $s$ divides $pp_0$ and $qq_0$. 
The integers $p_0$ and $q_0$ being relatively prime, it holds that $s\, gcd(s,p_0) ^{-1}\, gcd(s,q_0)^{-1}$ is an integer. Moreover it is relatively prime with $p_0$ and with $q_0$. So it divides $p$ and $q$. These integers being relatively prime, we have $s=gcd(s, p_0) \, gcd (s,q_0)\leq p_0 q_0$ and hence the period $q'=qq_0/r=s$ of  $z_{p/Q}$ satisfies $q'\geq q/p_0$. We deduce that $([T']_{\mathcal{FHL}}{}^{p_0}, q_0, p_0)$ is a homotopical interval of rotation.
\end{proof}

\begin{proposition} \label{prop:existlargeperiod} 
If the sets
$$W^{R\to L}\cap W^D\,,\quad W^{L\to R}\cap W^D\,,\quad W^{R\to R}\cap W^D\,,\quad W^{L\to L}\cap W^D$$ 
are $\nu$-null sets, then there exists $a>0$ such that :
\begin{itemize}
\item one has $\mathrm{rot}_f(\nu)=a[T']$;
\item for every $p/q\in[0,a)\cap \Q$, written in an irreducible way, there exists $\tilde z$ such that $\tilde f^q(\tilde z) =T'{}^p(\tilde z)$.
\end{itemize}
\end{proposition}

\begin{proof} 
Recall that $T=T'{}^m$, where $m\geq 1$. Let us begin by proving that $\tilde B$ is invariant by $T'$. It is sufficient to prove that for every $n>0$ we have $\overline{L(T'{}^n\tilde\phi})\subset L(\tilde\phi)$. If $L(\tilde\phi)\subset L(T'{}^n\tilde\phi)$, then for every $k\geq 1$ we have  $L(T'{}^{nk}\tilde\phi)\subset L(T'{}^{n(k+1)}\tilde\phi)$ and so we deduce that  $L(\tilde\phi)\subset L(T'{}^{nm}\tilde\phi)$, which contradicts the inclusion $\overline{L(T'{}^{nm}\tilde\phi)}\subset L(\tilde\phi)$. If $L(\tilde\phi)\cap L(T'{}^n\tilde\phi)=\emptyset$, then $L(\tilde\phi)$ is disjoint from its image by $T'{}^n$. The map $T'{}^n$ being fixed point free, Brouwer Translation Theorem \cite{Br} tells us that $L(\tilde\phi)$ is disjoint from its image by $T'{}^{nm}$, which contradicts the inclusion $\overline{L(T'{}^{nm}\tilde\phi)}\subset L(\tilde\phi)$. Similarly, if $R(\tilde\phi)\cap R(T'{}^n\tilde\phi)=\emptyset$, then $R(\tilde\phi)$ is disjoint from its image by $T'{}^{nm}$, which contradicts the inclusion $\overline{R(\tilde\phi)}\subset R(T'{}^{nm}\tilde\phi)$. The only remaing case is the case where $\overline{L(T'{}^n\tilde\phi})\subset L(\tilde\phi)$.

In the following instead of seeing $\tilde B$ as a $\tilde T$-strip, we will see it as a $T'$-strip:  one can choose $\tilde \gamma_*$ to be  invariant by $T'$ and suppose that $\tilde\gamma_*(t)=T'\gamma_*(t)$ for every $t\in\R$.

By construction of $\tilde B$ we know that $\nu(W^D)=1$.
So, $\nu$-almost every point $z\in \mathrm{dom}(I)$ is recurrent and has a lift $\tilde z\in \widetilde{\mathrm{dom}}(I)$ such that $\tilde I_{\tilde{\F}}^{\Z}(\tilde z)$ is equivalent to $\tilde\gamma_*$ at $+\infty$ or $-\infty$. Indeed if $z\in W^D$ is recurrent and if $\tilde z$ was a lift of $z$  such that $\tilde I_{\tilde{\F}}^{\Z}(\tilde z)$ accumulates on $\tilde\gamma_*$, then there would exist $k\in\Z$ such that $\tilde I_{\tilde{\F}}^{\Z}(\tilde z)$ accumulates on $\tilde I_{\tilde{\F}}^{\Z}(T'{}^k(\tilde z))$. It is impossible because $z$ is recurrent and so has no self-accumulation by Corollary~\ref{coro:recurrence}.
Hence, $\tilde I_{\tilde{\F}}^{\Z}(\tilde z)$ does not accumulate on $\tilde\gamma_*$, and by the hypothesis of the proposition it cannot go out of $\tilde B$ both before and after it draws $\tilde B$. This implies that it has to be equivalent to $\tilde\gamma_*$ at $+\infty$ or $-\infty$.

In fact we can be more precise: if there are $a<a'$ and $b\in\R$ such that $\tilde I_{\tilde{\F}}^{\Z}(\tilde z)|_{[a,a']}$ is equivalent to $\tilde \gamma_*|_{[b,b+1]}$, then either $\tilde I_{\tilde{\F}}^{\Z}(\tilde z)|_{[a,+\infty)}$ is equivalent to a subpath of $\tilde \gamma_*$  
(and equivalent to $\tilde\gamma_*$ at $+\infty$ but we will not use this property) or $\tilde I_{\tilde{\F}}^{\Z}(\tilde z)|_{(-\infty,a']}$ is equivalent to a subpath of $\tilde \gamma_*$. From this we will deduce the following lemma.

\begin{lemma} \label{le:whole equivalence}
The transverse path $\tilde I_{\tilde{\F}}^{\Z}(\tilde z)$ is equivalent to $\tilde \gamma_*$. Moreover there is neighborhood $\tilde U$ of $\tilde z$ such that if the orbit of $\tilde z$ meets $R\tilde U$ for some $R\in \mathcal G$, then $R$ is a power of $T'$.
\end{lemma}

\begin{proof}
Let us treat the case where $\tilde I_{\tilde{\F}}^{\Z}(\tilde z)|_{[a,+\infty)}$ is equivalent to a subpath of $\tilde \gamma_*$, the other case being identical.

Suppose that $\tilde I_{\tilde{\F}}^{\Z}(\tilde z)$ is not equivalent to $\tilde \gamma_*$. Then, as we have already seen that it cannot accumulate in $\tilde \gamma_*$, this means that there exists $b<a$, $b\in\Z$, such that $\tilde I_{\tilde{\F}}^b(\tilde z)\notin \tilde B$. 

By recurrence of the point $z$, there exists a sequence of integers $n_k\to -\infty$, and a sequence of deck transformations $(R_k)_{k\in\N}\in\G$ such that $R_k\tilde f^{n_k}(\tilde z)$ tends to $z$; in particular for any $k$ large enough:

\begin{itemize}
\item $\tilde \gamma_*|_{[b,b+1]}$ is equivalent to a subpath of $R_k \tilde I_{\tilde{\F}}^{\Z}(\tilde z)|_{[n_k+a-1,+\infty)}$ (and in particular this path draws $\tilde B$);
\item  $R_k\tilde I_{\tilde{\F}}^{n_k+b}(\tilde z)\notin \tilde B$.
\end{itemize}
By the same reasoning as before the lemma, we deduce that either the trajectory $R_k\tilde I_{\tilde{\F}}^{\Z}(\tilde z)|_{[n_k+a+1,+\infty)}$ is equivalent to a subpath of $\tilde \gamma_*$, or $R_k\tilde I_{\tilde{\F}}^{\Z}(\tilde z)|_{(-\infty,n_k+a'-1]}$ is  equivalent to a subpath of $\tilde \gamma_*$. By the second point above, the second situation is impossible. Hence, $R_k\tilde I_{\tilde{\F}}^{\Z}(\tilde z)|_{[n_k+a+1,+\infty)}$ is equivalent to a subpath of $\tilde \gamma_*$.

In particular, this implies that $R_k \tilde\gamma_*$ is equivalent at $+\infty$ to $\tilde\gamma_*$. By Lemma~\ref{LemEquivSubpath}, this implies that $R_k\tilde\gamma_*\cap\tilde\gamma_*\neq\emptyset$; more precisely it implies that for any $n$ large enough, $R_k\tilde\gamma_*\cap\tilde\gamma_*|_{[b+n, b+n+1)}\neq\emptyset$, hence that $R_k\tilde\gamma_*\cap\tilde\gamma_*$ is infinite. This implies that $R_k\tilde\gamma_* = \tilde\gamma_*$, in other words $R_k=T'{}^{i_k}$ for some $i_k\in\Z$. 

We deduce that $T'{}^{i_k}\tilde I_{\tilde{\F}}^{\Z}(\tilde z)|_{[n_k+a+1,+\infty)}$ is equivalent to a subpath of $\tilde \gamma_*$, equivalently (as $\tilde\gamma_*$ is $T'$-invariant), for any $k\in\N$, the path $\tilde I_{\tilde{\F}}^{\Z}(\tilde z)|_{[n_k+a+1,+\infty)}$ is equivalent to a subpath of $\tilde \gamma_*$.

This proves that $\tilde I_{\tilde{\F}}^{\Z}(\tilde z)$ is equivalent to a subpath of $\tilde \gamma_*$. As it cannot accumulate in $\tilde\gamma_*$, this proves that $\tilde I_{\tilde{\F}}^{\Z}(\tilde z)$ is equivalent to $\tilde \gamma_*$. 

To get the second part of the lemma, consider a neighborhood $\tilde U$ of $\tilde z$ such that for every $\tilde z'\in \tilde U$, the path  $\tilde I_{\tilde{\F}}^{\Z}(\tilde z')$ draws $\tilde\gamma_*$. If $\tilde f^k(\tilde z)\in R\tilde U$, $R\in \mathcal G$, then  $\tilde I_{\tilde{\F}}^{\Z}(\tilde z)$ draws $R(\tilde\gamma_*)$. We deduce that $\tilde I_{\tilde{\F}}^{\Z}(\tilde z)$ is equivalent to $R\tilde\gamma_*$. What was done above tells us that $R\in\langle T'\rangle$. 
\end{proof}

Now, let us consider
\begin{itemize}
\item the connected component  $\widetilde{\mathrm{dom}}(I)_{\tilde \gamma_*}$ of $\widetilde{\mathrm{dom}}(I)$ that contains $\tilde\gamma_*$,
\item the quotient space  $\widehat{\mathrm{dom}}(I)=\widetilde{\mathrm{dom}}(I)/T$,
\item the foliation $\hat{\F}$ of $\widehat{\mathrm{dom}}(I)$ lifted by $\tilde{\F}$,
\item the  covering projection  $\hat\pi: \widehat{\mathrm{dom}}(I)\to \mathrm{dom}(I)$,
\item the annulus $\widehat{\mathrm{dom}}(I)_{\tilde \gamma_*}=\widetilde{\mathrm{dom}}(I)_{\tilde \gamma_*}/T'$, 
\item the universal covering projection  $\pi: \widetilde{\mathrm{dom}}(I)_{\tilde \gamma_*}\to \widehat{\mathrm{dom}}(I)_{\tilde \gamma_*}$.
\end{itemize}

\begin{lemma} \label{le:rotationnumberannulus} 
It holds that $\nu$-almost every point $z$ has a lift in $ \widehat{\mathrm{dom}}(I)_{\tilde \gamma_*}$ that is positively recurrent and has a rotation number $a>0$ (in the annulus). Moreover we have $\mathrm{rot}_f(\nu)=a[T']$.
\end{lemma}

\begin{proof} 
We know that $\nu$-almost every point $z$ is positively recurrent and has a lift $\tilde z$ in $ \widetilde{\mathrm{dom}}(I)_{\tilde \gamma_*}$ that draws $\tilde B$. We have seen in Lemma  \ref{le:whole equivalence} that  $\tilde I_{\tilde{\F}}^{\Z}(\tilde z)$ is equivalent to $\tilde\gamma_*$ and that there exists a neighborhood $\tilde U$ of $\tilde z$ such that if the orbit of $\tilde z$ meets $R\tilde U$,  for some $R\in \mathcal G$, then $R$ is a power of $T'$. Using the fact that $z$ is recurrent, we deduce that $\widehat z=\pi(\widetilde z)$ is positively recurrent.  By the argument given in the proof of Lemma~\ref{claim:proportional}, we deduce that $z$ has rotation number $a>0$. Moreover we have $\mathrm{rot}_f(\nu)=a[T]$.
\end{proof}

Now there are two cases to consider. The first case is the case where the stabilizer of $\widetilde{\mathrm{dom}}(I)_{\tilde\gamma_*}$ is generated by $T'$ and the second case is when it is larger. In the first case, $\hat\pi$ sends homeomorphically $\widehat{\mathrm{dom}}(I)_{\tilde\gamma_*}$ onto a connected component of $\mathrm{dom}(I)$.  Moreover, the frontier of this annulus is made of contractible fixed points of $f$.
In the second case, $\hat\pi$ sends $\widehat{\mathrm{dom}}(I)_{\tilde\gamma_*}$ onto a hyperbolic surface whose universal covering space is $\widetilde{\mathrm{dom}}(I)_{\tilde\gamma_*}$ and the group of covering automorphisms is the stabilizer of $\widetilde{\mathrm{dom}}(I)_{\tilde\gamma_*}$ in $\mathcal G$. 
In both cases, there exists an extension $\overline{\widehat{\mathrm{dom}}}(I)_{\tilde\gamma_*}$ of $\widehat{\mathrm{dom}}(I)_{\tilde\gamma_*}$ obtained by blowing at least one end $e$ with a circle $\hat \Gamma_e$ and $\hat f$ extends to a homeomorphism $\overline{\hat f}$ of $\overline{\widehat{\mathrm{dom}}}(I)_{\tilde\gamma_*}$ (see Paragraph~\ref{ss:Caratheodory}). Furthermore, the rotation number(s) induced on  the added circle(s) by the lift of $\overline{\hat f}$  that extends $\tilde f$ are equal to $0$. 

By Lemma \ref{le:rotationnumberannulus}, there exist positively recurrent points with rotation number $a>0$ where $\mathrm{rot}_f(\nu)=a[T']$.
Consequently, according to Theorem~\ref{th:PB} that can be applied thanks to Lemma~\ref{le:intersecting}, for every rational number $p/q\in(0,a)$, written in an irreducible way, there exists a point $\tilde z$ such that $\tilde f ^q(\tilde z)=T^p(\tilde z)$. As $\tilde f$ also has a fixed point by Lefschetz index theorem, this means that $f$ has a homotopical interval of rotation of type $(\kappa, r)$ such that $[\kappa]/r\in \mathcal U$.
\end{proof}

\end{document}